\documentclass[a4paper,12pt]{article}

\usepackage{amsfonts,graphicx,amsmath,amssymb,amsthm,color,nicefrac,enumerate, layout, bbm, listings,  hyperref}
\usepackage[latin1]{inputenc}
\usepackage[numbers]{natbib}
\usepackage{bbm}

\usepackage{geometry}
\newtheorem{theorem}{Theorem}[section]
\newtheorem{prop}[theorem]{Proposition}
\newtheorem{lemma}[theorem]{Lemma}
\newtheorem{cor}[theorem]{Corollary}
\newtheorem{sett}[theorem]{Setting}

\newcommand{\defeq}{\curvearrowleft}

\newcommand{\1}{\ensuremath{\mathbbm{1}}}
\newcommand{\eps}{\varepsilon}
\newcommand{\R}{\mathbb{R}}
\newcommand{\Q}{\mathbb{Q}}
\newcommand{\N}{\mathbb{N}}
\newcommand{\E}{\mathbb{E}}
\renewcommand{\P}{\mathbb{P}}

\newcommand{\HS}{\operatorname{HS}}
\newcommand{\constFun}{\phi}

\newcommand{\F}[0]{\ensuremath{\mathcal{F}}}

\newcommand{\smallsum}{{\textstyle\sum}}

\allowdisplaybreaks

\title{On moments and strong local H\"older regularity\\
of solutions of stochastic differential equations\\
and
of their spatial derivative processes} 
\author{Anselm Hudde$^1$, Martin Hutzenthaler$^1$,
and Sara Mazzonetto$^{2}$ 
\bigskip
\\
\small{$^1$Faculty of Mathematics, University of Duisburg-Essen,
Germany}
\\
\small{$^2$Universit\'e de Lorraine, CNRS, Inria, IECL, F-54000 Nancy, France}
\smallskip
}

\begin{document}

\maketitle
\makeatletter
\let\@Xmakefnmark\@makefnmark
\let\@Xthefnmark\@thefnmark
\let\@makefnmark\relax
\let\@thefnmark\relax
\@footnotetext{\emph{Key words and phrases:}
strong completeness, strong local H\"older continuity, differentiable solution.}
\@footnotetext{\emph{AMS 2010 subject classification}: 60H10} 
\let\@makefnmark\@Xmakefnmark
\let\@thefnmark\@Xthefnmark
\makeatother

\maketitle

\begin{abstract}
  Spatial differentiability of solutions of stochastic differential
  equations (SDEs)
  is a classical question in stochastic analysis.
  The case of coefficients with globally Lipschitz continuous
  derivatives is well understood in the literature.
  Counterexamples with smooth and bounded coefficients
  demonstrate that the non-globally Lipschitz case
  is more subtle.
  In this article we establish conditions, including a suitable local monotonicity property, which provide existence of continuously differentiable solutions of SDEs, moment estimates and strong local H\"older regularity.  
\end{abstract}

\tableofcontents


\section{Introduction}
For a number of tools
such as Taylor expansions,
the
It\^o-Alekseev-Gr\"obner formula
in \cite{HuddeHutzenthalerJentzenMazzonetto2018},
or the backward It\^o-Ventzell formula
in \cite{DelMoralSingh2019}
it is convenient to have a solution
of a stochastic differentiable equation (SDE)
which is continuously differentiable in the starting point.

In the literature it is well-known that
spatially differentiable solutions of SDEs exist
if the first derivatives of the coefficient functions exist
and are globally H\"older continuous;
see, e.g., \cite[Theorem 4.6.5]{Kunita1990}.
The classical approach is to prove with
a Kolmogorov-Chentsov continuity argument
that difference
quotients have continuous versions and to infer
from this the existence of differentiable solutions.
We note that the proof of 
\cite[Theorem 4.6.5]{Kunita1990} leaves the gap
that 
\cite[Theorem 1.4.1]{Kunita1990} does not directly guarantee
existence of a continuous extension of the difference quotient at $h=0$ since $\R^d\times(\R\setminus\{0\})$
is not a domain
(Proposition \ref{p:KolChen} below closes this gap)
and
the proof of 
\cite[Theorem 4.6.5]{Kunita1990} does not explain
how to get a differentiable solution
from a continuous version of difference quotients
(Lemma \ref{lem:gradient} below closes this gap).
Moreover, we note that the conditions
of
\cite[Theorem 1.4.1]{Kunita1990} are rarely satisfied
by SDEs from applications; cf.\ examples
in \cite[Chapter 4]{CoxHutzenthalerJentzen2013v3}.

In the non-globally Lipschitz case one might also
hope for differentiable solutions if the coefficient
functions are smooth.
However, this is not the case.
Example 2.6 in \cite{LiScheutzow2011} shows
that there exist SDEs with smooth and bounded 
coefficients for which there exists no solution
which is continuous in the starting point.
In particular, SDEs with smooth coefficients do
not necessarily have differentiable solutions.

The main contribution of this article is to derive conditions which allow at the same time to ensure existence of continuously differentiable solutions of order $0$, $1$, or $2$,
to bound moments and to establish strong local H\"older regularity.
Our central assumption is the local monotonicity-type
assumption \eqref{eq:intro} below
which seems to be new.
Roughly speaking,
this condition requires $\mu'$ and $\|\sigma'\|$
to be bounded from above in a suitable sense by
a Lyapunov-type function
where $\mu$ is the drift coefficient and
$\sigma$ is the diffusion coefficient of the SDE.
The following theorem, Theorem~\ref{thm:intro},
illustrates our main results and is
a special case of Corollary~\ref{c:C2} below
which in turn is derived from 
Theorem \ref{thm:C2} below.
The proof of 
Theorem~\ref{thm:intro} is therefore omitted.
We note that Theorem \ref{thm:intro}
and Theorem \ref{thm:exists:C0} below
in particular imply \emph{strong completeness} of the SDE.


\begin{theorem}[Existence of a $C^2$-solution]\label{thm:intro}
  Let $d\in\N$,
  $T,c,\alpha\in(0,\infty)$, $p\in(6(d+3)(1+1/c)^3,\infty)$,
  let $\|\cdot\|$, $\langle,\rangle$ denote the standard norm and the standard scalar product on $\R^d$,
  let $\|\cdot\|_{\textup{F}}$ denote the Frobenius norm on $\R^{d\times d}$,
  let $(\Omega, \F, \P, (\mathbb{F}_{t})_{t\in [0,T]})$ be a filtered probability space
  satisfying the usual conditions\footnote{'usual conditions' means that for all $t\in[0,T]$ it holds
that $\{A\in\mathcal{F}\colon\P(A)=0\}\subseteq\mathbb{F}_t=\cap_{s\in(t,T]\cup\{T\}}\mathbb{F}_s$.},
  let 
  $
    W\colon[0,T]\times\Omega\to\R^d
  $
  be a standard $(\mathbb{F}_t)_{t\in[0,T]}$-Wiener process,
  let $ \mu \in C^3(\R^d,\R^d)$, 
  $ \sigma \in C^3(\R^d, \R^{d\times d})$,
  $ 
    V  \in C^{ 2 }( \R^d , [0,\infty) ) 
  $,
  $ \bar{V}\in C(\R^d,[0,\infty) )$,
  assume that $\mu'''$ and $\sigma'''$ grow at most polynomially at infinity,
  assume that $\exists \gamma\in(0,\infty)\colon \sup_{x\in\R^d}\|x\|^{\gamma}/(1+V(x))<\infty$,
  and assume for all $x,y,v\in \R^d$ that
  {\small
  \begin{align}
    &\nonumber\Big\langle v,\smallint_0^1\mu'\big(\lambda (x-y)+y\big)\,d\lambda \,\,v
    \Big\rangle
    +\tfrac{p-1}{2}
    \Big\|\smallint_0^1\sigma'\big(\lambda (x-y)+y\big)\,d\lambda \,\,v
    \Big\|_{\textup{F}}^2
  \leq \|v\|^2\cdot\Big(c
    +\tfrac{V(x)+V(y)}{4c p T e^{\alpha T}}
    +\tfrac{\bar{V}(x)+\bar{V}(y)}{4c p e^{\alpha T}}
    \Big),
\\
    &\Big\langle
  \mu( x )
  ,
  (\nabla V)(x)
  \Big\rangle
  +
  \tfrac{ 1 }{ 2 }
  \operatorname{trace}\!\Big(
    \sigma(x) [\sigma(x)]^* 
    ( \operatorname{Hess} V )( x )
  \Big)
  +
  \tfrac{ 
    1
  }{ 
    2 
  }
    \|
      \sigma( x )^* ( \nabla V )( x )
    \|^2
  +
  \bar{V}(x)
\leq
  \alpha V(x)
  +
  c.
  \label{eq:intro}
  \end{align}}
Then there exists a measurable
function
$
  X \colon \{(s,t)\in[0,T]\colon s\leq t\} \times \R^d \times \Omega \to \R^d
$
such that
\begin{enumerate}[(i)]
\item
for every $ \omega \in \Omega $
it holds
that
$
  X(\omega) \in C^{0,2}( \{(s,t)\in[0,T]\colon s\leq t\} \times \R^d, \R^d)
$ and
  \item  
for all 
$
  x \in \R^d 
$,
$ 
	s\in [0,T]
$,
$t\in[s,T]$
it holds a.s.~that 
\begin{equation}  \begin{split}
  X^x_{s,t} = 
  x
  + \int_s^{ t } \mu(X^x_{s,r} ) \, dr
  +
  \int_s^{ t } \sigma(X^x_{s,r} ) \, dW_r.
\end{split}     \end{equation}
\end{enumerate}
\end{theorem}

The paper is organized as follows.
In Section \ref{sec:2} we provide preliminary results used for the proofs of the main results.
In Sections \ref{sec:3}, \ref{sec:4}, and \ref{sec:5} we establish 
existence of
 $C^0$, $C^1$, and $C^2$ solutions respectively of SDEs using suitable strong local H\"older estimates
 obtained in Subsections \ref{sec:3.1}, \ref{sec:4.1}, and \ref{sec:5.1}.

\section{Preliminary results}\label{sec:2}

In Subsection~\ref{ssec:inferring} we prove that if there exists a continuous version
of difference quotients, then there exists a differentiable version.
Moreover, we recall for the convenience of the reader results from the literature
which are used in our proofs of existence of solutions of SDEs which are differentiable in
the initial value.
More precisely, we will use in our proofs
a local Kolmogorov-Chentsov continuity theorem (see Subsection~\ref{ssec:KolmogorovChentsov} below),
a stochastic Gronwall lemma (see Subsection~\ref{ssec:Gronwall} below),
and
exponential moment estimates (see Subsection~\ref{ssec:moments} below).

\subsection{Inferring a differentiable version from continuity of difference quotients}\label{ssec:inferring}
The following lemma shows, informally speaking, that if the
difference quotient of a continuous 
random field
has a continuous version, then the random field is differentiable
almost surely.
\begin{lemma}[Continuity of the difference quotient implies differentiability] \label{lem:gradient}
  Let $(\Omega,\mathcal{F},\P)$ be a probability space,
  let $(H,\langle \cdot,\cdot\rangle_H,\|\cdot\|_H)$ and 
  $(U,\langle \cdot,\cdot\rangle_U,\|\cdot\|_U)$ be
  separable $\R$-Hilbert spaces,
  assume that $H$ is finite-dimensional,
  let $\mathbb{H}\subseteq H$ be an orthonormal basis of $H$,
  let $O \subseteq H$ be an open subset,
  let $D \subseteq O$ be a countable dense subset,
  let $(T,\mathcal{T})$ be a topological space,
  let $\mathbb{T}\subseteq T$ be a countable dense subset,
  let $\mathcal{O}\subseteq H\times\R$ be the set
  $\mathcal{O}=\cap_{h\in\mathbb{H}}\{ (x,p) \in O \times \mathbb{\R}
  \colon x + h p \in O \} $,
  and 
  let $\mathcal{X}\colon T\times O\times\Omega\to U$
  and $\mathcal{Z}\colon  T\times
  \mathcal{O}
  \times H \times\Omega\to U$
  be random fields satisfying 
  that
  for all $(x,p)\in \mathcal{O}\cap\big(D\times(\Q\setminus\{0\})\big)$,
  $h\in\mathbb{H}$, $t\in\mathbb{T}$
  it holds a.s.~that
  \begin{equation}  \label{lem:gradient:ZX}
  \begin{split}
    \mathcal{Z}_t(x,p,h)=\tfrac{\mathcal{X}_t^{x+p h}-\mathcal{X}_t^x}{p}
  \end{split}     
  \end{equation}
  and satisfying that
  for
  all $h\in\mathbb{H}$
  and
  almost all $\omega\in\Omega$
  it holds that
  the functions
  $T\times O\ni (t,x)\mapsto \mathcal{X}_t^x(\omega)\in U$
  and 
  $T\times \mathcal{O}
  \ni (t,x,p)
  \mapsto
  \mathcal{Z}_t(x,p,h,\omega) \in U$
  are continuous.
  Then there exists a set ${\Omega_0}\in\mathcal{F}$ such that
  \begin{enumerate}[(i)]
  \item \label{item:gradient:1}
  $\P({\Omega_0})=1$, and
  \item \label{item:gradient:3}
  it holds for all $\omega\in{\Omega_0}$ and all $t\in T$ that the mapping
    $O\ni x\mapsto \mathcal{X}_t^x(\omega)\in U$ is continuously
    differentiable
    and it holds for all $\omega\in\Omega_0$, $t\in T$, $x\in O$, $v\in H$
    that $\frac{\partial }{\partial  x} \mathcal{X}_t^x(\omega)v
    =\sum_{h\in\mathbb{H}}\langle v,h\rangle_H \mathcal{Z}_t(x,0,h,\omega)$.
  \end{enumerate}
\end{lemma}
\begin{proof}[Proof of Lemma~\ref{lem:gradient}]
  Without loss of generality we assume that $H\neq \{0\}$ and that $T\neq \emptyset$.
  Throughout this proof let $d\in\N$ denote the dimension of $H$
  and let 
  $h_1,\ldots,h_d \in H$ satisfy that
  $\{h_1,\ldots,h_d\}=\mathbb{H}$.
  By assumption there exists a set $\Omega_1\in\mathcal{F}$ satisfying that $\P(\Omega_1)=1$
  and that for all $\omega\in\Omega_1$, $h\in\mathbb{H}$ it holds that
  the functions
  $T\times O\ni (t,x)\mapsto \mathcal{X}_t^x(\omega)\in U$
  and 
  $T\times \mathcal{O}
  \ni (t,x,p)\mapsto \mathcal{Z}_t(x,p,h,\omega)
  \in U$
  are continuous.
  Let $\Omega_0\subseteq\Omega$ be the set satisfying that
  \begin{equation}  \label{lem:gradient:Omega1}
  \begin{split}
   {\Omega_0} =  \Omega_1 \cap \bigcap_{(x,p)\in  \mathcal{O}\cap (D\times\Q)}\bigcap_{h\in\mathbb{H}}
    \bigcap_{t\in \mathbb{T}}
    \Big\{\omega\in\Omega\colon \mathcal{X}_t^{x+p h}(\omega)-\mathcal{X}_t^x(\omega)
     =p\mathcal{Z}_t({x,p,h,\omega})
    \Big\}.
  \end{split}
  \end{equation}
  The fact that $\mathcal{X},\mathcal{Z}$ are random fields,
  the fact that for all 
  $(x,p) \in \mathcal{O}\cap (D\times\Q)$, $h\in\mathbb{H}$, $t\in\mathbb{T}$
  it holds a.s.~that
  $p\mathcal{Z}_t(x,p,h) =\mathcal{X}_t^{x+p h}-\mathcal{X}_t^x$,
  the fact that $D\times\Q\times\mathbb{H}\times \mathbb{T}$
  is a countable set,
  the fact that $\Omega_1\in\mathcal{F}$,
  and the fact that $\P(\Omega_1)=1$
  imply that ${\Omega_0}\in\mathcal{F}$ and that $\P({\Omega_0})=1$.
  This proves item~\eqref{item:gradient:1}.

Next, we prove item~\eqref{item:gradient:3}. 
  For this we first
  observe that density of $\mathbb{T}$ in $T$
  and the fact that for all $h\in\mathbb{H}$, $\omega\in\Omega_1$
  the functions
  $T\times O\ni (t,x)\mapsto \mathcal{X}_t^x(\omega)\in U$
  and 
  $T\times \mathcal{O}\ni (t,x,p)\mapsto \mathcal{Z}_t(x,p,h,\omega) \in U$
  are continuous
  imply that
  \begin{equation}  \label{lem:gradient:Omega1b}
  \begin{split}
   {\Omega_0} =  \Omega_1 \cap \bigcap_{(x,p)\in  \mathcal{O}}\bigcap_{h\in\mathbb{H}}
    \bigcap_{t\in T}
    \Big\{\omega\in\Omega\colon \mathcal{X}_t^{x+p h}(\omega)-\mathcal{X}^x(\omega)
     =p\mathcal{Z}_t({x,p,h,\omega})
    \Big\}.
  \end{split}
  \end{equation}
For the rest of the proof, 
 	let $\omega\in{\Omega_0}$, $t\in T$, $x\in O$, and $r\in(0,\infty)$ with $\{y\in H\colon\|y-x\|<2r\}\subseteq O$ be fixed.
  %
  Now for all $v\in H$, $i\in\{1,\ldots,d\}$ with $\|v\|_{H}<r$ it holds that
  \begin{equation}  \begin{split}
  \big\|x+\sum_{j=1}^{i-1}\langle v,h_j\rangle_{H}h_j-x\big\|_H^2+|\langle v,h_i\rangle_{H}|^2
  =\sum_{j=1}^{i}\left(\langle v,h_j\rangle_{H}\right)^2\leq \|v\|_{H}^2<  r^2.
  \end{split}     \end{equation}
  This, a telescoping sum, the fact that $\omega\in\Omega_0$,
  and~\eqref{lem:gradient:Omega1b}
  yield that for all $v\in H$ with $\|v\|_H<r$ it holds that
  \begin{equation}  \begin{split}
     &\mathcal{X}_t^{x+v}(\omega)-\mathcal{X}_t^x(\omega)
    -\sum_{h\in\mathbb{H}}\langle v,h\rangle_H \mathcal{Z}_t(x,0,h,\omega)
     \\&=
     \sum_{i=1}^{d}\left(
     \mathcal{X}_t^{x+\sum_{j=1}^i\langle v,h_j\rangle_{H}h_j}(\omega)
    -\mathcal{X}_t^{x+\sum_{j=1}^{i-1}\langle v,h_j\rangle_{H}h_j}(\omega)
    \right)
    -\sum_{i=1}^{d}\langle v,h_i\rangle_H \mathcal{Z}_t(x,0,h_i,\omega)
     \\&=
     \sum_{i=1}^{d}
     \langle v,h_i\rangle_{H}
     \mathcal{Z}_t\Big(x+\smallsum_{j=1}^{i-1}\langle v,h_j\rangle_{H}h_j,\langle v,h_i\rangle_{H},h_i,\omega\Big)
     -
     \sum_{i=1}^{d}
     \langle v,h_i\rangle_{H}
     \mathcal{Z}_t(x,0,h_i,\omega).
  \end{split}     \end{equation}
  This, the triangle inequality, and the Cauchy-Schwarz inequality show that for all $v\in H$ with $\|v\|_{H}\in(0,r)$ it holds that
  \begin{equation}  \begin{split}
     &\tfrac{\left\|\mathcal{X}_t^{x+v}(\omega)-\mathcal{X}_t^x(\omega)
    -\sum_{h\in\mathbb{H}}\langle v,h\rangle_H \mathcal{Z}_t(x,0,h,\omega)
     \right\|_{U}}{\|v\|_{H}}
     \\&=
     \tfrac{1}{\|v\|_{H}}
     \left\|\sum_{i=1}^{d}\langle v,h_i\rangle_{H}
     \bigg(
     \mathcal{Z}_t\Big(x+\smallsum_{j=1}^{i-1}\langle v,h_j\rangle_{H}h_j,\langle v,h_i\rangle_{H},h_i,\omega\Big)
     -
     \mathcal{Z}_t(x,0,h_i,\omega)
    \bigg)
    \right\|_{U}
     \\&\leq
     \tfrac{\|v\|_{H}}{\|v\|_{H}}
     \left(
     \sum_{i=1}^{d}
     \left\|\mathcal{Z}_t\Big(x+\smallsum_{j=1}^{i-1}\langle v,h_j\rangle_{H}h_j,\langle v,h_i\rangle_{H},h_i,\omega\Big)
     -
     \mathcal{Z}_t(x,0,h_i,\omega)
    \right\|_{U}^2\right)^{\!\!\!\nicefrac{1}{2}}.
  \end{split}     \end{equation}
  This, the fact that $H$ is finite-dimensional, and the fact that for all $h\in\mathbb{H}$
  the function
  $\mathcal{O}\ni (y,p)\mapsto \mathcal{Z}(y,p,h,\omega) \in U$
  is continuous in $(x,0)$ yield that
  \begin{equation}  \begin{split}
    \lim_{H\ni v\to 0}\tfrac{\left\|\mathcal{X}_t^{x+v}(\omega)-\mathcal{X}_t^x(\omega)
    -\sum_{h\in\mathbb{H}}\langle v,h\rangle_H \mathcal{Z}_t(x,0,h,\omega)
      \right\|_{U}}{\|v\|_{H}}=0.
  \end{split}     \end{equation}
This together with continuity of the function
  $O\ni y\mapsto (H\ni v\mapsto
  \sum_{h\in\mathbb{H}}\langle v,h\rangle_H \mathcal{Z}_t(y,0,h,\omega)
  \in U) \in L(H,U)$
proves item~\eqref{item:gradient:3} and thus completes the proof of Lemma~\ref{lem:gradient}.
\end{proof}

\subsection{A local Komogorov-Chentsov continuity theorem}\label{ssec:KolmogorovChentsov}

For the convenience of the reader the following proposition, Proposition~\ref{p:KolChen},
reformulates part of
Corollary 3.12 in~\cite{CoxHutzenthalerJentzen2013v3}
which is a version of the
Kolmogorov-Chentsov continuity theorem.
\begin{prop}[A local Komogorov-Chentsov continuity theorem]\label{p:KolChen}
Let
$(\Omega, \mathcal{F}, \P)$ be a probability space,  
let $( H, \left< \cdot , \cdot \right>_H, \left\| \cdot \right\|_H )$
be a finite-dimensional $\R$-Hilbert space,
let $ D  \subseteq H$ be a set,
let $ ( E, \left\| \cdot \right\|_E )$ be a separable $\R$-Banach space,
let $ F \subseteq E $ be a closed subset,
let
$ p \in (\dim(H),\infty) $,  
$ 
  \alpha \in ( \frac{ \dim(H) }{ p } , 1]
$,
and let 
$ 
  X \colon D \times\Omega\to E
$
be a random field which satisfies for all $n\in\N$, $z\in D$ that
$\P(X(z)\in F)=1$, $\E[\|X(z)\|_E^p]<\infty$, and
that 
\begin{equation}  \begin{split}\label{eq:locallyHoelder}
  \sup\Big(&
  \Big\{\tfrac{\left(\E\left[\|X(x)-X(y)\|_E^p\right]\right)^{\frac{1}{p}}}
  {\|x-y\|_H^{\alpha }}
  \colon x,y\in D, \|x\|_H\leq n, \|y\|_H\leq n,x\neq y \Big\}
  \cup\{0\}
  \Big)
  <\infty.
\end{split}     \end{equation}
Then there exists 
a measurable function 
$
  \mathcal{X} \colon \overline{D} \times \Omega \to F
$
which satisfies
\begin{enumerate}[(i)]
  \item 
that for all $\omega\in\Omega$ it holds that
the function
$\overline{D}\ni x\mapsto \mathcal{X}(x,\omega)\in F$ is continuous
and
\item that for all $x\in D$ it holds a.s.~that
$
  \mathcal{X}(x) = X(x)
$.
\end{enumerate}
\end{prop}

\subsection{A stochastic Gronwall inequality}\label{ssec:Gronwall}

For the convenience of the reader the following proposition,
Proposition \ref{prop:moments:hilbert},
reformulates part of Corollary 2.5 in~\cite{HHM2021}.
It will be used to prove strong local H\"older estimates for difference quotients.
\begin{prop}[A stochastic Gronwall inequality]\label{prop:moments:hilbert}
Let
$( H, \left< \cdot , \cdot \right>_H, \left\| \cdot \right\|_H )$
and
$( U, \left< \cdot , \cdot \right>_U, \left\| \cdot \right\|_U )$
be separable $\R$-Hilbert spaces,
let 
$p\in[2,\infty)$, $T\in(0,\infty)$,
let $(\Omega, \F, \P, (\mathbb{F}_{t})_{t\in [0,T]})$ be a filtered probability space
satisfying the usual conditions,
let $(W_t)_{t \in [0, T]}$
  be an $\textup{Id}_U$-cylindrical $(\mathbb{F}_t)_{t\in[0,T]}$-Wiener process,
let $X,a\colon[0,T]\times\Omega\to H$,
$b\colon[0,T]\times\Omega\to \HS(U,H)$,
$\alpha,\beta \colon [0,T] \times \Omega \to [0,\infty]$
be
measurable and adapted stochastic processes
which satisfy that $X$ has continuous sample paths,
which satisfy that for all $t\in[0,T]$ it holds a.s.\ that
$\int_0^T\|a_s\|_H+\|b_s\|_{\HS(U,H)}^2+|\alpha_s|\,ds<\infty$
and
$X_t=X_0+\int_0^t a_s\,ds +\int_0^tb_s\,dW_s$,
and which satisfy that it holds a.s.\ for Lebesgue-almost all $t\in[0,{T}]$ that
\begin{equation}  \begin{split}\label{eq:lin.growth}
  \langle X_t,a_t\rangle_H+\tfrac{1}{2}\|b_t\|_{\HS(U,H)}^2
  +\tfrac{p-2}{2}\tfrac{\|\langle X_t,b_t\rangle_H\|_{\HS(U,\R)}^2}{\|X_t\|_H^2}
  \leq \alpha_t\|X_t\|_H^2+\tfrac{1}{2}|\beta_t|^2.
\end{split}     \end{equation}
   Then it holds for all $t\in[0,T]$, $q_1,q_2\in(0,\infty]$
   with $\tfrac{1}{q_1}=\tfrac{1}{q_2}+\tfrac{1}{p}$ that
   \begin{equation}  \begin{split}\label{eq:prop:moments:hilbert}
     &\|X_{t}\|_{L^{q_1}(\P;H)}
     \leq
     \bigg\|
     \exp\Big(\smallint_0^{ t } \alpha_u\,du \Big)
     \bigg\|_{L^{q_2}(\P;\R)}
     \left(
     \left\|
     X_0
     \right\|_{L^p(\P;H)}^2
     +\int_0^t
     \big\|
     \beta_s
     \big\|_{L^{p}(\P;\R)}^2\,ds
     \right)^{\frac{1}{2}}.
   \end{split}     \end{equation}
\end{prop}

\subsection{Exponential moment estimates}\label{ssec:moments}
In this subsection we collect two results from the literature which formalize
a Lyapunov-method to derive (exponential) moment estimates.
We will use these estimates to prove condition~\eqref{eq:locallyHoelder}
for suitable difference quotients.
In this subsection we frequently use the following setting.
\begin{sett}\label{sett:integrability}
Let 
  $( H, \left< \cdot , \cdot \right>_H, \left\| \cdot \right\|_H )$
  and $( U, \left< \cdot , \cdot \right>_U, \left\| \cdot \right\|_U )$ be separable $\R$-Hilbert spaces,
  let $T\in [0,\infty)$, $s\in[0,T]$,
let $(\Omega, \F, \P, (\mathbb{F}_{t})_{t\in [0,T]})$ be a filtered probability space
satisfying the usual conditions,
  let 
  $
    W \colon [s,T] \times \Omega \to U
  $
  be an adapted stochastic process such that $(W_{s+t}-W_s)_{t\in[0,T-s]}$ is an $\textup{Id}_U$-cylindrical
  $(\mathbb{F}_{s+t})_{t\in[0,T-s]}$-Wiener process,
  let $O \subseteq H$ be an open set,
  let $\mathcal{O}\in\mathcal{B}(O)$,
  let $ \mu \colon \mathcal{O} \to H $
  and 
  $ \sigma \colon \mathcal{O} \to \HS(U,H)$
  be measurable functions,
  and
  let
  $
    X \colon [s,T] \times \Omega \to \mathcal{O}
  $
  be an adapted stochastic
  process with continuous sample paths
  which satisfies that for all $t\in[s,T]$ it holds a.s.\ that
  $
    \smallint_s^{ T } \| \mu( X_r ) \|_H
    + \| \sigma( X_r ) \|_{\HS(U,H)}^2
    \, dr < \infty
  $
  and
  $
    X_{t } = 
    X_s
    + \smallint_s^{ t } \mu(X_r ) \, dr
    +
    \smallint_s^{ t } \sigma(X_r ) \, dW_r
  $.
\end{sett}

The next result, Lemma~\ref{l:exp_mom}, 
is a slight generalization of
Corollary 2.4 in~\cite{CoxHutzenthalerJentzen2013v3} to arbitrary nonnegative starting times.

\begin{lemma}[Exponential moment estimates]
\label{l:exp_mom}
  Assume Setting~\ref{sett:integrability},
  let $ \alpha,\beta \in \R $, $ V \in C^2( O, \R ) $,
  and let
  $ \bar{V} \colon [s,T] \times \mathcal{O}\to\R $ be a measurable function
  which satisfies that it holds a.s.\ that $\int_s^T| \bar{V}(r, X_r) |\,dr<\infty$
  and that
  it holds
  for all
  $ (t,x)\in \cup_{\omega\in \Omega} \cup_{r\in [s,T]} \{ (r, X_{r}(\omega)) \in [s,T] \times \mathcal{O} \}$
  that 
  \begin{equation}  \begin{split}
  \label{eq:exp_mom_assumption}
  &\Big\langle
  \mu( x )
  ,
  (\nabla V)(x)
  \Big\rangle_H
  +
  \tfrac{ 1 }{ 2 }
  \operatorname{trace}\!\Big(
    \sigma(x) [\sigma(x)]^* 
    ( \operatorname{Hess} V )( x )
  \Big)
    + 
    \tfrac{ 
      1 
    }{ 
      2 e^{ \alpha t } 
    } 
    \left\| 
      \sigma(x)^*
      \left(\nabla V\right)\!(x) 
    \right\|_U^2 
    +
    \bar{V}(t, x )
    \\&
  \leq 
    \alpha V(x) +\beta.
  \end{split}     \end{equation}
  Then
  \begin{equation}
  \label{eq:exp:mom:estimates}
  \begin{split}
  &
    \E\!\left[
      \exp\!\left(
        \tfrac{   
          V( X_{ T } )
        }{ 
          e^{ \alpha T }
        }    
        +
        \smallint_s^{ T }
          \tfrac{ 
            \bar{V}( r,  X_r ) 
          }{ 
            e^{ \alpha r } 
          }
        \, dr
      \right)
    \right]
  \leq
    \E\!\left[\!
      \exp\!\left( 
        \tfrac{V(X_s)}{e^{\alpha s}} 
        +
        \smallint_s^{T}\tfrac{\beta}{e^{\alpha r}}\,dr
      \right)
    \right]
    \in [0,\infty].
  \end{split}
  \end{equation}
\end{lemma}
\begin{proof}[Proof of Lemma~\ref{l:exp_mom}]
Corollary 3.3 in~\cite{HHM2021}
(applied in the case $T>s$ with $T\defeq T-s$, 
$(\mathbb{F}_t)_{t\in [0,T-s]} \defeq  (\mathbb{F}_{t+s})_{t\in [0,T-s]} $,
$(W_t)_{t\in [0,T-s]} \defeq  (W_{s+t}-W_s)_{t\in [0,T-s]}$,
$\tau\defeq T-s$,
$(X_{t})_{t\in [0,T-s]}\defeq (X_{t+s})_{t\in [0,T-s]}$, 
$\bar{U}\defeq  \left( [0,T-s] \times \mathcal{O} \ni (t,x) \mapsto e^{-\alpha s} \bar V(t+s, x)-e^{-\alpha s}\beta \in \R \right)$,
$U\defeq \big( O \ni x \mapsto e^{-\alpha s} V(x) \in \R \big)$
in the notation of
Corollary 3.3 in~\cite{HHM2021})
implies \eqref{eq:exp:mom:estimates}.
This proves Lemma~\ref{l:exp_mom}.
\end{proof}

The next result, Lemma~\ref{l:exp_mom.moments}, 
is a consequence of
Corollary 3.4 in~\cite{HHM2021}.

\begin{lemma}[Exponential moment condition implies moments]
\label{l:exp_mom.moments}
  Assume Setting~\ref{sett:integrability},
  let $ \alpha,\beta \in [0,\infty) $, $ V \in C^2( O, [0,\infty) ) $
  satisfy
  for all $x\in \mathcal{O}$
  that 
  \begin{equation}  \begin{split}
  &\Big\langle
  \mu( x )
  ,
  (\nabla V)(x)
  \Big\rangle_H
  +
  \tfrac{ 1 }{ 2 }
  \operatorname{trace}\!\Big(
    \sigma(x) [\sigma(x)]^* 
    ( \operatorname{Hess} V )( x )
  \Big)
    + 
    \tfrac{ 
      1 
    }{ 
      2 e^{ \alpha s } 
    } 
    \left\| 
      \sigma(x)^*
      \left(\nabla V\right)\!(x) 
    \right\|_U^2 
    \\&
  \leq 
    \alpha V(x) +\beta,
  \end{split}     \end{equation}
  and let $t\in[s,T]$, $p\in[1,\infty)$.
  Then
  it holds that
  \begin{equation}  \begin{split}\label{eq:moment.exp.marginal.condition}
    &\left\|1+V(X_{t})\right\|_{L^p(\P;\R)}
    \leq
    e^{\alpha t}\left(p+\smallint_s^t \tfrac{\beta}{e^{\alpha r}}\,dr
    +e^{-\alpha s}\|V(X_s)\|_{L^p(\P;\R)}\right).
  \end{split}     \end{equation}
\end{lemma}
\begin{proof}[Proof of Lemma~\ref{l:exp_mom.moments}]
Corollary 3.4 in~\cite{HHM2021}
(applied in the case $t>s$ with $T\defeq t-s$, 
$(\mathbb{F}_u)_{u\in [0,t-s]} \defeq  (\mathbb{F}_{u+s})_{u\in [0,t-s]} $,
$(W_u)_{u\in [0,t-s]} \defeq  (W_{s+u}-W_s)_{u\in [0,t-s]}$,
$\mu\defeq ([0,t-s]\times\mathcal{O}\ni (r,x)\mapsto \mu(x)\in H)$,
$\sigma\defeq ([0,t-s]\times\mathcal{O}\ni (r,x)\mapsto \sigma(x)\in \HS(U,H))$,
$\tau\defeq t-s$,
$(X_{u})_{u\in [0,t-s]}\defeq (X_{s+u})_{u\in [0,t-s]}$, 
$U\defeq \big( [0,t-s]\times O \ni (r,x) \mapsto e^{-\alpha r}e^{-\alpha s} V(x) \in [0,\infty) \big)$,
$\beta\defeq \beta e^{-\alpha s}$
in the notation of
Corollary 3.4 in~\cite{HHM2021})
implies that
\begin{equation}  \begin{split}
    &\left\|1+V(X_{t})\right\|_{L^p(\P;\R)}
    \leq e^{\alpha t}\left\|p+e^{-\alpha(t-s)}e^{-\alpha s}V(X_{t})\right\|_{L^p(\P;\R)}
    \\&
    \leq
    e^{\alpha t}
    \left\|p+e^{-\alpha s}V(X_{s})+\smallint_0^{t-s}\tfrac{\beta e^{-\alpha s}}{e^{\alpha r}}\,dr\right\|_{L^p(\P;\R)}
    \leq
    e^{\alpha t}
    \left(p+\smallint_s^t \tfrac{\beta}{e^{\alpha r}}\,dr
    +e^{-\alpha s}\|V(X_s)\|_{L^p(\P;\R)}\right).
\end{split}     \end{equation}
This proves 
\eqref{eq:moment.exp.marginal.condition}
and finishes the proof of
Lemma~\ref{l:exp_mom.moments}.
\end{proof}


The following result, Lemma~\ref{lem:multiple_exp:guess}, generalizes the $k=1$ case of \cite[Lemma~2.23]{CoxHutzenthalerJentzen2013v3}
which is a special case of Lemma~\ref{lem:multiple_exp:guess} with
$H=\R^d$, $s=0$, $\bar{V}(t,\cdot)=\bar{V}(\cdot)$ for all $t\in [0,T]$, $X^{(1)}=X^{x}$, $X^{(3)}= X^{x}$, $X^{(2)}=X^{y}$,  and $X^{(4)}=X^{y}$
for $d\in \N$, $x,y\in O$.
\\

\begin{lemma}
\label{lem:multiple_exp:guess}
Assume Setting~\ref{sett:integrability}, 
let
  $
    X^{(j)} \colon [s,T] \times \Omega \to \mathcal{O}
  $, $j\in \{1,2,3,4\}$
  be $(\mathbb{F}_{t})_{t\in [s,T]} $-adapted stochastic
  processes with continuous sample paths
  satisfying that for all $j\in \{1,2,3,4\}$,
  $t\in [s,T]$ it holds a.s.\ that
  $
    \smallint_s^{ T } \| \mu( X_r^{(j)} ) \|_H 
    + \| \sigma( X_r^{(j)} ) \|_{\HS(U,H)}^2
    \, dr < \infty
  $
  and that
  $
    X_{t }^{(j)} = 
    X_{s}^{(j)}
    + \smallint_s^{ t } \mu(X_r^{(j)} ) \, dr
    +
    \smallint_s^{ t } \sigma(X_r^{(j)}) \, dW_r,
  $
let $ \alpha_0,\,\alpha_1,$ $\beta_0,\,\beta_1 \in \R$,
let
$ 
  V_0 
$,
$
  V_1  \in C^{ 2 }( O , [0,\infty) ) 
$,
let
$ 
  \bar{V}\colon[s,T]\times\mathcal{O}\to\R
$ 
be measurable
and satisfy
$\P\big(\sum_{j=1}^4\int_s^T|\bar{V}(r,X_r^{(j)})|\,dr<\infty\big)=1$
and
for all
$ i \in \{ 0, 1 \} $,
$ (t,x)\in \cup_{\omega\in \Omega} \cup_{r\in [s,T]} \cup_{j=1}^4 \{ (r, X^{(j)}_{r}(\omega)) \in [s,T] \times \mathcal{O} \}$
that
\begin{equation}
\label{eq:multiple_exp_est2}
\begin{split}
  &\Big\langle
  \mu( x )
  ,
  (\nabla V_i)(x)
  \Big\rangle_H
  +
  \tfrac{ 1 }{ 2 }
  \operatorname{trace}\!\Big(
    \sigma(x) [\sigma(x)]^* 
    ( \operatorname{Hess} V_i )( x )
  \Big)
  \\&
  +
  \tfrac{ 
    1
  }{ 
    2 
    e^{ 
      \alpha_{ i } {t} 
    }
  }
    \|
      \sigma( x )^* ( \nabla V_{ i } )( x )
    \|_U^2
  +
  \mathbbm{1}_{
    \{ 1 \}
  }(i)
  \cdot
  \bar{V}(t,x)
\leq
  \alpha_{ i } V_{ i }(x)
  +
  \beta_{ i },
\end{split}
\end{equation}
let 
$q, q_0,\,q_1\in(0,\infty]$
satisfy
$
  \tfrac{ 1 }{ q_{ 0} }+\tfrac{ 1 }{ q_{ 1} } = \tfrac{ 1 }{ q }
$,
and
let 
$
  \constFun \colon [s,T] \to \R
$
be a measurable function
with
$
  \int_s^{ T }
    |\constFun(r)|
    \,
  dr
  < \infty
$.
Then it holds that
\begin{equation}
\begin{split}\label{eq:2terms}
&
    \left\|
      \exp\!\left(
 \int_s^T
 \Big(
  \constFun(r)
  +
      \sum_{j=1}^4
  \left[
    \tfrac{ 
      V_{ 0 }( X_r^{(j)} ) 
    }{ 4 q_{ 0 } (T-s) e^{ \alpha_{ 0} r } } 
    +
    \tfrac{ 
     \bar{V}( r, X_r^{(j)} ) 
    }{ 
      4 q_{ 1} 
      e^{ \alpha_{ 1} r }
    }
  \right]
  \Big)\,dr
  \right)
    \right\|_{
      L^q( \P; \R )
    }
\\ &
\leq
  \exp\!\left(
        \int_s^T  
        \Big(
        \constFun(r)
        +
        \tfrac{
          \beta_{ 0 } \, ( 1 - \frac{ r-s }{ T-s } )
        }{
          q_{0} e^{ \alpha_{ 0 } r }
        }
        + 
        \tfrac{ \beta_{ 1 } }{ 	q_{ 1 } e^{ \alpha_{ 1 } r } }
        \Big)
        \,
        dr
        \right)
  \prod_{ i = 0 }^1 \prod_{j=1}^4
  \left\| \exp\left(  \tfrac{
          V_{ i }( X^{(j)}_s ) 
        }{
          4 q_{ i } e^{\alpha_{ i } s 
        } } \right) \right\|_{L^{4 q_{ i }}(\P;\R)}
  \!\! .
\end{split}
\end{equation}
\end{lemma}
\begin{proof}[Proof of Lemma~\ref{lem:multiple_exp:guess}]
Without loss of generality we assume for the rest of the proof that $\constFun \equiv 0$,
otherwise divide by $\exp(\int_s^T \constFun(r)\,dr)\in(0,\infty)$.
H{\"o}lder's inequality together with
$\tfrac{ 1 }{ q } =4\tfrac{ 1 }{ 4q_{ 0} }+4\tfrac{ 1 }{ 4q_{ 1} }$,
the fact that
$
\int_s^{T} \tfrac{\beta_0(1-\frac{r-s}{T-s})}{q_0 e^{\alpha_0 r }} \, dr
=
\int_s^{T} \int_s^r \tfrac{\beta_0}{q_0(T-s)e^{\alpha_0 u}} \, du \, dr
$,
Jensen's inequality,
nonnegativity of $V_1$,
and
Lemma~\ref{l:exp_mom}
(applied
for every $j\in\{1,2,3,4\}$, $t\in[s,T]$
with 
$T \defeq  t$,
$\alpha\defeq \alpha_0$,
$\beta\defeq 0$,
$V\defeq V_0$,
$\bar{V}\defeq  -\beta_0$,
$X \defeq  X^{(j)}$ 
and applied
for every $j\in\{1,2,3,4\}$
with
$\alpha\defeq \alpha_1$,
$\beta\defeq 0$,
$V\defeq V_1$,
$\bar{V}\defeq \bar{V}-\beta_1$,
$X \defeq  X^{(j)}$
in the notation of Lemma~\ref{l:exp_mom})
show that
\begin{equation}
\begin{split}
&
    \left\|
      \exp\!\left(
 \int_s^T
      \sum_{j=1}^4
  \left[
    \tfrac{ 
      V_{ 0 }( X_r^{(j)} ) 
    }{ 4 q_{ 0 } (T-s) e^{ \alpha_{ 0} r } } 
    +
    \tfrac{ 
     \bar{V}(r, X_r^{(j)} ) 
    }{ 
      4 q_{ 1} 
      e^{ \alpha_{ 1} r }
    }
  \right]dr
  \right)
    \right\|_{
      L^q( \P; \R )
    }
  \exp\left(
    -\int_s^{T}\sum_{i=0}^1\tfrac{\beta_i(1-\frac{r-s}{T-s})^{1-i}}{q_ie^{\alpha_ir}}\,dr\right)
    \\
    &
\leq
 \prod_{j=1}^4
 \left[
    \left\|
      \exp\!\left(
      \tfrac{1}{T-s}
 \int_s^{T}
 \Big(
    \tfrac{ 
      V_{ 0 }( X_{r}^{(j)} ) 
    }{  4q_0  e^{ \alpha_{ 0} r } } 
    -\int_s^r\tfrac{\beta_0}{4 q_0  e^{\alpha_0 u}}\,du
  \Big)
    \,dr
    \right)
    \right\|_{
      L^{4q_0}( \P; \R )
    }
    \right.
    \\ & \qquad \cdot 
    \left.
 \left\|
   \exp\!\left(
   \int_s^{T} \!\!
    \tfrac{ 
     \bar{V}(r, X_{r}^{(j)} )-\beta_1 
    }{ 
      4 q_1 e^{ \alpha_{ 1} r }
    }
  \,dr
  \right)
    \right\|_{
      L^{4q_1}( \P; \R )
    }
    \right]
    \\
    &
\leq
\prod_{j=1}^4
\left[\sup_{t\in[s,T]}
\left(
    \E\!\left[
      \exp\!\left(
    \tfrac{ 
      V_{ 0 }( X_{t}^{(j)} ) 
    }{   e^{ \alpha_{ 0} t } } 
    -\int_s^t \tfrac{\beta_0}{e^{\alpha_0 u}}\,du
    \right)
    \right]
    \right)^{\!\!\nicefrac{1}{4q_0}}
    \right. 
    \\ & \qquad \cdot 
 \left. \left(\E\!\left[
   \exp\!\left(
    \tfrac{ 
       V_{ 1 }( X_T^{(j)} ) 
    }{ e^{ \alpha_{ 1} T } } 
    +
   \int_s^{T}
    \tfrac{ 
     \bar{V}(r, X_{r}^{(j)} ) -\beta_1
    }{ 
      e^{ \alpha_{ 1} r } 
    }
  \,dr
  \right)
  \right]\right)^{\!\!\nicefrac{1}{4q_1}}
  \right]
\\ &
\leq
\prod_{j=1}^4
\left[
\left( \E\!\left[ \exp\!\Big(\tfrac{V_0(X^{(j)}_s)}{e^{\alpha_{ 0 } s}}\Big) \right]\right)^{\!\!\nicefrac{1}{4q_0}}
\left(
\E\!\left[
\exp\!\left(
          \tfrac{V_{ 1 }( X^{(j)}_s )}{e^{\alpha_{ 1 } s}} 
    \right)
\right]
\right)^{\!\!\nicefrac{1}{4q_1}}
\right]
=
  \prod_{ i = 0 }^1 \prod_{j=1}^4 
  \left\| \exp\left(  \tfrac{
          V_{ i }( X^{(j)}_s ) 
        }{
          4 q_{i} e^{\alpha_{ i } s 
        } } \right) \right\|_{L^{ 4 q_{ i } }(\P;\R)}
  \!\! .
\end{split}
\end{equation}
This implies~\eqref{eq:2terms}.
The proof of Lemma~\ref{lem:multiple_exp:guess} is thus completed.
\end{proof}


\section{Existence of a $C^0$-solution}\label{sec:3}
In this section we prove 
a strong local H\"older estimate 
for solutions of SDEs in Lemma \ref{lem:Hoelder} below,
a moment estimate for the first derivative process
in Lemma \ref{lem:moments:Zv} below,
and
establish existence of a continuous solution under suitable assumptions
in Theorem \ref{thm:exists:C0} below.
First, we introduce the setting for these results.
\begin{sett} \label{s:exists:C0}
Let $( H, \left< \cdot , \cdot \right>_H, \left\| \cdot \right\|_H )$ and $( U, \left< \cdot , \cdot \right>_U, \left\| \cdot \right\|_U )$
be separable $\R$-Hilbert spaces,
let
$ T \in (0,\infty) $,
let $O \subseteq H$ be an open set,
let $\mathcal{O}\in\mathcal{B}(O)$,
let $ \mu \in C(\mathcal{O},H)$,
$ \sigma \in C(\mathcal{O}, \HS(U,H))$,
let $ \alpha_0,\alpha_1,\beta_0,\beta_1\in [0,\infty)$,
$ 
  V_0 
$,
$
  V_1  \in C^{ 2 }( O , [0,\infty) ) 
$,
let
$ 
  \bar{V} \colon [0,T] \times \mathcal{O} \to [0,\infty) 
$
be a measurable function,
assume
for all 
$ i \in \{ 0, 1 \} $,
$t\in[0,T]$,
$x\in \mathcal{O}$
that
\begin{equation}
\begin{split}
  &\Big\langle
  \mu( x )
  ,
  (\nabla V_i)(x)
  \Big\rangle_H
  +
  \tfrac{ 1 }{ 2 }
  \operatorname{trace}\!\Big(
    \sigma(x) [\sigma(x)]^* 
    ( \operatorname{Hess} V_i )( x )
  \Big)
  \\&
  +
  \tfrac{ 
    1
  }{ 
    2 
    e^{ 
      \alpha_{ i } {t} 
    }
  }
    \|
      \sigma( x )^* ( \nabla V_{ i } )( x )
    \|_U^2
  +
  \mathbbm{1}_{
    \{ 1 \}
  }(i)
  \cdot
  \bar{V}(t,x)
\leq
  \alpha_{ i } V_{ i }(x)
  +
  \beta_{ i },
\end{split}
\end{equation}
let 
$
  \constFun \colon [0,T] \to [0,\infty)
$
be a measurable function satisfying that $\int_0^T \constFun(r) \, dr <\infty$,
let $p\in[2,\infty)$,
$\theta\in[0,\infty)$, $q_0,q_1\in(0,\infty)$
satisfy $\frac{\theta}{p(1+\theta)}=\frac{1}{q_0}+\frac{1}{q_1}$,
assume for all $t\in[0,T]$, $x,y\in \mathcal{O}$ that
\begin{equation}  \begin{split}
  &\big\langle x-y,\mu(x)-\mu(y)\big\rangle_H+\tfrac{1}{2}\big\|\sigma(x)-\sigma(y)\big\|^2_{\HS(U,H)}
  +\tfrac{p(1+\theta)-2}{2}
    \tfrac{
    \left\|\left\langle x-y,\sigma(x)-\sigma(y)\right\rangle_H\right\|_{\HS(U,\R)}^2}{\|x-y\|_H^2}
  \\&
  \leq \|x-y\|^2_H\cdot\Big(\phi(t)
  +\tfrac{V_0(x)+V_0(y)}{2q_0T e^{\alpha_0t}}
  +\tfrac{\bar{V}(t,x)+\bar{V}(t,y)}{2q_1e^{\alpha_1t}}
  \Big),
\end{split}     \end{equation}
let $\gamma\in[\tfrac{1}{p},\infty)$, $c\in[0,\infty)$ satisfy for all $x\in\mathcal{O}$ that
\begin{align}
\begin{split}
 \max\left\{ \|\mu(x)\|_H , \|\sigma(x)\|_{\HS(U,H)} \right\}
 \leq  c(1+V_0(x))^{\gamma},
\end{split}
\end{align}
let $(\Omega, \F, \P, (\mathbb{F}_{t})_{t\in [0,T]})$ be a filtered probability space
satisfying the usual conditions,
let 
$
  (W_t)_{t\in[0,T]}
$
  be an $\textup{Id}_U$-cylindrical $(\mathbb{F}_t)_{t\in[0,T]}$-Wiener process,
for all $s\in [0,T]$, $x\in\mathcal{O}$
let $ X^x_{s,\cdot} \colon [s,T] \times \Omega \to \mathcal{O} $
be an $(\mathbb{F}_t)_{t\in[s,T]}$-adapted stochastic process
with continuous sample paths
satisfying that for all $t\in [s,T]$ it holds a.s.~that
$
  \int_s^T | \bar{V}(r, X_{s,r}^x) | d r <\infty
$
and 
\begin{align}
  X^x_{s,t} = 
  x
  + \int_s^{ t } \mu(X^x_{s,r} ) \, dr
  +
  \int_s^{ t } \sigma(X^x_{s,r} ) \, dW_r,
\end{align}
and
let $\Delta_T =\{(s,t)\in [0,T]^2 \colon s \leq t\}$.
\end{sett}

\subsection{Strong local H\"older estimate}\label{sec:3.1}
The following lemma proves strong local H\"older continuity
of SDE solutions in the starting point, starting time, and terminal time.
Lemma \ref{lem:Hoelder}
improves existing results in
in~\cite{CoxHutzenthalerJentzen2013v3,
FangImkellerZhang2007,
Li1994,%
Zhang2010,%
HHM2021
}.

\begin{lemma}[Strong local H\"older estimate] \label{lem:Hoelder}
Assume Setting~\ref{s:exists:C0}
and let 
$s_1,s_2,t_1,t_2\in [0,T]$, $x_1,x_2\in \mathcal{O}$
satisfy that $s_1\leq t_1$, $s_2\leq t_2$
and $s_1\leq s_2$.
Then it holds that 
\begin{equation}  \begin{split}
  &\|X_{s_1,t_1}^{x_1}-X_{s_2,t_2}^{x_2}\|_{L^{p}(\P;H)}
  \leq
    \sqrt{|t_1-t_2|}c e^{\alpha_0\gamma T}\Big|p\gamma+e^{-\alpha_0 s_1}V_0(x_1)
    +\textstyle\int_{s_1}^T\tfrac{\beta_0}{e^{\alpha_0u}}\,du\Big|^{\gamma}
    \big(\sqrt{T}+p\big)
\\&\quad
+\|x_1-x_2\|_H
  \exp\!\left(
        \int_{s_1}^{T}  
        \Big(
        \constFun(r)
        +
        \tfrac{
          \beta_{ 0 } 
        }
        { q_0 e^{ \alpha_{ 0 } r } }
        + 
        \tfrac{ \beta_{ 1 } }
        { q_1 e^{ \alpha_{ 1 } r } }
        \Big)
        \,
        dr
        +
        \sum_{ i = 0 }^1 
        \tfrac{V_i(x_1)+V_i(x_2)}{2q_ie^{\alpha_i s_1} } 
      \right)
  \\&\quad+
c e^{\alpha_0 \gamma|s_2-s_1|}\Big|p(1+\theta)\gamma+e^{-\alpha_0 s_1}V_0(x_2)
+\textstyle\int_{s_1}^T\tfrac{\beta_0}{e^{\alpha_0u}}\,du\Big|^{\gamma}
\big(\sqrt{T}+p(1+\theta)\big)\sqrt{|s_2-s_1|}
\\&
\qquad\quad\cdot
  \exp\!\left(
        \int_{s_1}^{T}  
        \Big(
        \constFun(r)
        +
        \tfrac{
          \beta_{ 0 } 
        }
        { q_0 e^{ \alpha_{ 0 } r } }
        + 
        \tfrac{ \beta_{ 1 } }
        { q_1 e^{ \alpha_{ 1 } r } }
        \Big)
        \,
        dr
        +
        \sum_{ i = 0 }^1 
        \tfrac{V_i(x_2)}{q_i e^{\alpha_i s_1} } 
      \right)
\!.
\end{split}     \end{equation}
\end{lemma}
\begin{proof}[Proof of Lemma~\ref{lem:Hoelder}]
  Lemma 3.8 in~\cite{HHM2021}
  (applied  in the case $s_1<T$
  with $T\defeq T-s_1$,
  $(\mathbb{F}_t)_{t\in[0,T-s_1]}\defeq (\mathbb{F}_{s_1+t})_{t\in[0,T-s_1]}$,
  $(W_t)_{t\in[0,T-s_1]}\defeq (W_{s_1+t}-W_{s_1})_{t\in[0,T-s_1]}$,
  $\mu\defeq ([0,T-s_1]\times\mathcal{O}\ni (r,x)\mapsto \mu(x)\in H)$,
  $\sigma\defeq ([0,T-s_1]\times\mathcal{O}\ni (r,x)\mapsto \sigma(x)\in \HS(U,H))$,
  $\tau\defeq T-s_1$,
  $(X_t)_{t\in[0,T-s_1]}\defeq (X_{s_1,s_1+t}^{x_1})_{t\in[0,T-s_1]}$,
  $(Y_t)_{t\in[0,T-s_1]}\defeq (X_{s_1,s_1+t}^{x_2})_{t\in[0,T-s_1]}$,
  $\beta_0\defeq e^{-\alpha_0 s_1}\beta_0$, $\beta_1=e^{-\alpha_1 s_1}\beta_1$,
  $V_0\defeq (O\ni x\mapsto e^{-\alpha_0 s_1}V_0(x)\in[0,\infty))$,
  $V_1\defeq (O\ni x\mapsto e^{-\alpha_1 s_1}V_1(x)\in[0,\infty))$,
  $\bar{V}\defeq ([0,T-s_1]\times \mathcal{O}\ni(t,x)\mapsto e^{-\alpha_1 s_1}\bar{V}(s_1+t,x)\in\R)$,
  $\constFun\defeq ([0,T-s_1]\ni t\mapsto\constFun(s_1+t)\in[0,\infty))$,
  $p\defeq p(1+\theta)$,
  $q\defeq p(1+1/\theta)$,
  $q_0\defeq q_0$,
  $q_1\defeq q_1$,
  $t_1\defeq t_1-s_1$,
  $t_2\defeq t_2-s_1$
  in the notation of
  Lemma 3.8 in~\cite{HHM2021})
  implies that
\begin{equation}  \begin{split}\label{eq:local.Hoelder.fixed.s}
  &\|X_{s_1,t_1}^{x_1}-X_{s_1,t_2}^{x_2}\|_{L^{p}(\P;H)}
  \\&
\leq
    \sqrt{|t_1-t_2|}c e^{\alpha_0\gamma (T-s_1)}\Big|p\gamma+e^{-\alpha_0 s_1}V_0(x_1)
    +\int_0^{T-s_1}\tfrac{e^{-\alpha_0 s_1}\beta_0}{e^{\alpha_0u}}\,du\Big|^{\gamma}
    \Big(\sqrt{T-s_1}+p\Big)
\\&\;\;
+\|x_1-x_2\|_H
  \exp\!\left(
        \int_0^{T-s_1}  
        \Big(
        \constFun(s_1+r)
        +
        \tfrac{
          e^{-\alpha_{0}s_1}\beta_{ 0 } 
        }
        { q_0 e^{ \alpha_{ 0 } r } }
        + 
        \tfrac{ e^{-\alpha_{1}s_1}\beta_{ 1 } }
        { q_1 e^{ \alpha_{ 1 } r } }
        \Big)
        \,
        dr
        +
        \sum_{ i = 0 }^1 
        \tfrac{V_i(x_1)+V_i(x_2)}{2q_ie^{\alpha_i s_1} } 
      \right)
  \! .
\end{split}     \end{equation}
  The fact that $\gamma \geq \tfrac{1}{p}\geq \tfrac{1}{p(1+\theta)}$,
  Lemma 3.7 in~\cite{HHM2021}
  (applied in the case $s_2>s_1$
  with $T\defeq s_2-s_1$,
  $(\mathbb{F}_t)_{t\in[0,s_2-s_1]}\defeq (\mathbb{F}_{s_1+t})_{t\in[0,s_2-s_1]}$,
  $(W_t)_{t\in[0,s_2-s_1]}\defeq (W_{s_1+t}-W_{s_1})_{t\in[0,s_2-s_1]}$,
  $\mu\defeq ([0,s_2-s_1]\times\mathcal{O}\ni(t,x)\mapsto\mu(x)\in H)$,
  $\sigma\defeq ([0,s_2-s_1]\times\mathcal{O}\ni(t,x)\mapsto\sigma(x)\in \HS(U,H))$,
  $\tau\defeq s_2-s_1$,
  $(X_t)_{t\in[0,s_2-s_1]}\defeq (X_{s_1,s_1+t}^{x_2})_{t\in[0,s_2-s_1]}$,
  $(Y_t)_{t\in[0,s_2-s_1]}\defeq (X_{s_1,s_1+t}^{x_2})_{t\in[0,s_2-s_1]}$,
  $\beta_0\defeq e^{-\alpha_0 s_1}\beta_0$,
  $\beta_1\defeq e^{-\alpha_1 s_1}\beta_1$,
  $V_0\defeq (O\ni x\mapsto e^{-\alpha_0 s_1}V_0(x)\in[0,\infty))$,
  $V_1\defeq (O\ni x\mapsto e^{-\alpha_1 s_1}V_1(x)\in[0,\infty))$,
  $\bar{V}\defeq ([0,s_2-s_1]\times \mathcal{O}\ni(t,x)\mapsto e^{-\alpha_1s_1}\bar{V}(s_1+t,x)\in[0,\infty))$,
  $\phi\defeq ([0,s_2-s_1]\ni t\mapsto\phi(s_1+t)\in[0,\infty))$,
  $p\defeq p(1+\theta)$,
  $q\defeq p(1+1/\theta)$,
  $q_0\defeq q_0$,
  $q_1\defeq q_1$,
  $r\defeq p(1+\theta)$,
  $s\defeq 0$
  in the notation of
  Lemma 3.7 in~\cite{HHM2021})
  yield that
 \begin{equation}  \begin{split}\label{eq:local.Hoelder.variable.s}
  &\|X_{s_1,s_2}^{x_2}-x_2\|_{L^{p(1+\theta)}(\P;H)}
  =\|X_{s_1,s_2}^{x_2}-X_{s_1,s_1}^{x_2}\|_{L^{p(1+\theta)}(\P;H)}
  \\&
\leq
c e^{\alpha_0 \gamma(s_2-s_1)}\Big|p(1+\theta)\gamma+e^{-\alpha_0 s_1}V_0(x_2)
+\int_0^{s_2-s_1}\tfrac{e^{-\alpha_0 s_1}\beta_0}{e^{\alpha_0 u}}\,du\Big|^{\gamma}
\big(\sqrt{s_2-s_1}+p(1+\theta)\big)\sqrt{s_2-s_1}.
\end{split}     \end{equation}
Moreover, the fact that for all $t\in[0,T-s_2]$
it holds a.s.~that
\begin{equation}  \begin{split}
  X_{s_1,s_2+t}^{x_2}
  =
  X_{s_1,s_2}^{x_2}+\int_0^t\mu(X_{s_1,s_2+r}^{x_2})\,dr
  +\int_0^t\sigma(X_{s_1,s_2+r}^{x_2})\,d(W_{s_2+r}-W_{s_2})
\end{split}     \end{equation}
and  that
\begin{equation}  \begin{split}
  X_{s_2,s_2+t}^{x_2}
  =
  x_2+\int_0^t\mu(X_{s_2,s_2+r}^{x_2})\,dr
  +\int_0^t\sigma(X_{s_2,s_2+r}^{x_2})\,d(W_{s_2+r}-W_{s_2}),
\end{split}     \end{equation}
  Lemma 3.8 in~\cite{HHM2021}
  (applied  in the case $s_2<T$
  with $T\defeq T-s_2$,
  $(\mathbb{F}_t)_{t\in[0,T-s_2]}\defeq (\mathbb{F}_{s_2+t})_{t\in[0,T-s_2]}$,
  $(W_t)_{t\in[0,T-s_2]}\defeq (W_{s_2+t}-W_{s_2})_{t\in[0,T-s_2]}$,
  $\mu\defeq ([0,T-s_2]\times\mathcal{O}\ni (r,x)\mapsto \mu(x)\in H)$,
  $\sigma\defeq ([0,T-s_2]\times\mathcal{O}\ni (r,x)\mapsto \sigma(x)\in \HS(U,H))$,
  $\tau\defeq T-s_2$,
  $(X_t)_{t\in[0,T-s_2]}\defeq (X_{s_1,s_2+t}^{x_2})_{t\in[0,T-s_2]}$,
  $(Y_t)_{t\in[0,T-s_2]}\defeq (X_{s_2,s_2+t}^{x_2})_{t\in[0,T-s_2]}$,
  $\beta_0\defeq e^{-\alpha_0 s_2}\beta_0$,
  $\beta_1\defeq e^{-\alpha_1 s_2}\beta_1$,
  $V_0\defeq (O\ni x\mapsto e^{-\alpha_0 s_2}V_0(x)\in[0,\infty))$,
  $V_1\defeq (O\ni x\mapsto e^{-\alpha_1 s_2}V_1(x)\in[0,\infty))$,
  $\bar{V}\defeq ([0,T-s_2]\times \mathcal{O}\ni(t,x)\mapsto e^{-\alpha_1 s_2}\bar{V}(s_2+t,x)\in\R)$,
  $p\defeq p(1+\theta)$,
  $q\defeq p(1+1/\theta)$,
  $q_0\defeq q_0$,
  $q_1\defeq q_1$,
  $\phi\defeq ([0,T-s_2]\ni t\mapsto\phi(s_2+t)\in[0,\infty))$,
  $t_1\defeq t_2-s_2$,
  $t_2\defeq t_2-s_2$
  in the notation of
  Lemma 3.8 in~\cite{HHM2021}),
  \eqref{eq:local.Hoelder.variable.s},
  nonnegativity of $\bar{V}$,
  and Lemma~\ref{l:exp_mom}
  (applied for every $i\in\{0,1\}$ with
  $T\defeq s_2$,
  $s\defeq s_1$,
$X\defeq (X^{x_2}_{s_1,t})_{t\in[s_1,s_2]}$,
$\alpha\defeq \alpha_i$,
$\beta\defeq \beta_i$,
$V\defeq V_i$,
$s\defeq 0$,
$\bar{V}\defeq 0$
in the notation of Lemma~\ref{l:exp_mom})
  imply that
\begin{equation}  \begin{split}\label{eq:s-diff}
  &\|X_{s_1,t_2}^{x_2}-X_{s_2,t_2}^{x_2}\|_{L^{p}(\P;H)}
  \\&
\leq
\Big\|X_{s_1,s_2}^{x_2}-x_2\Big\|_{L^{p(1+\theta)}(\P;H)}
 \left(\prod_{i=0}^1\left(\left(\E\Big[\exp\Big(
 \tfrac{V_i(X_{s_1,s_2}^{x_2})}{e^{\alpha_i s_2}}\Big)
 \Big]\right)^{\frac{1}{2q_i}}\right)\right)
 \\&\quad\cdot \exp\!\left(
        \int_0^{T-s_2}  
        \Big(
        \constFun(s_2+r)
        +
        \tfrac{
          e^{-\alpha_{0}s_2}\beta_{ 0 } 
        }
        { q_0 e^{ \alpha_{ 0 } r } }
        + 
        \tfrac{ e^{-\alpha_{1}s_2}\beta_{ 1 } }
        { q_1 e^{ \alpha_{ 1 } r } }
        \Big)
        \,
        dr
        +
        \sum_{ i = 0 }^1 
        \tfrac{V_i(x_2)}{2q_ie^{\alpha_i s_2} } 
      \right)
\\&
\leq
c e^{\alpha_0 \gamma|s_2-s_1|}\Big|p(1+\theta)\gamma+e^{-\alpha_0 s_1}V_0(x_2)
+\textstyle\int_{s_1}^T\tfrac{\beta_0}{e^{\alpha_0u}}\,du\Big|^{\gamma}
\big(\sqrt{|s_2-s_1|}+p(1+\theta)\big)\sqrt{|s_2-s_1|}
\\&
\quad\cdot
\left(\prod_{i=0}^1
\left(
  \exp\!\left(
        \tfrac{V_i(x_2)}{e^{\alpha_i s_1} } 
        +
        \int_{s_1}^{s_2}  
        \tfrac{
          \beta_{ i } 
        }
        { e^{ \alpha_{ i } r } }
        \,
        dr
      \right)
  \right)^{\frac{1}{2q_i}}
  \right)
 \\&\quad\cdot \exp\!\left(
        \int_0^{T-s_2}  
        \Big(
        \constFun(s_2+r)
        +
        \tfrac{
          e^{-\alpha_{0}s_2}\beta_{ 0 } 
        }
        { q_0 e^{ \alpha_{ 0 } r } }
        + 
        \tfrac{ e^{-\alpha_{1}s_2}\beta_{ 1 } }
        { q_1 e^{ \alpha_{ 1 } r } }
        \Big)
        \,
        dr
        +
        \sum_{ i = 0 }^1 
        \tfrac{V_i(x_2)}{2q_ie^{\alpha_i s_2} } 
      \right)
 \end{split}     \end{equation}
Finally, the triangle inequality, \eqref{eq:local.Hoelder.fixed.s},
\eqref{eq:s-diff},
and
nonnegativity of $\phi,\alpha_0,\alpha_1,\beta_0,\beta_1,V_0,V_1$
yield that
\begin{equation}  \begin{split}
  &\|X_{s_1,t_1}^{x_1}-X_{s_2,t_2}^{x_2}\|_{L^{p}(\P;H)}
  \leq
  \|X_{s_1,t_1}^{x_1}-X_{s_1,t_2}^{x_2}\|_{L^{p}(\P;H)}
  +
  \|X_{s_1,t_2}^{x_2}-X_{s_2,t_2}^{x_2}\|_{L^{p}(\P;H)}
  \\&
  \leq
    \sqrt{|t_1-t_2|}c e^{\alpha_0\gamma T}\Big|p\gamma+e^{-\alpha_0 s_1}V_0(x_1)
    +\textstyle\int_{s_1}^T\tfrac{\beta_0}{e^{\alpha_0u}}\,du\Big|^{\gamma}
    \big(\sqrt{T}+p\big)
\\&\quad
+\|x_1-x_2\|_H
  \exp\!\left(
        \int_{s_1}^{T}  
        \Big(
        \constFun(r)
        +
        \tfrac{
          \beta_{ 0 } 
        }
        { q_0 e^{ \alpha_{ 0 } r } }
        + 
        \tfrac{ \beta_{ 1 } }
        { q_1 e^{ \alpha_{ 1 } r } }
        \Big)
        \,
        dr
        +
        \sum_{ i = 0 }^1 
        \tfrac{V_i(x_1)+V_i(x_2)}{2q_ie^{\alpha_i s_1} } 
      \right)
  \\&\quad+
c e^{\alpha_0 \gamma|s_2-s_1|}\Big|p(1+\theta)\gamma+e^{-\alpha_0 s_1}V_0(x_2)
+\textstyle\int_{s_1}^T\tfrac{\beta_0}{e^{\alpha_0u}}\,du\Big|^{\gamma}
\big(\sqrt{T}+p(1+\theta)\big)\sqrt{|s_2-s_1|}
\\&
\qquad\quad\cdot
  \exp\!\left(
        \int_{s_1}^{T}  
        \Big(
        \constFun(r)
        +
        \tfrac{
          \beta_{ 0 } 
        }
        { q_0 e^{ \alpha_{ 0 } r } }
        + 
        \tfrac{ \beta_{ 1 } }
        { q_1 e^{ \alpha_{ 1 } r } }
        \Big)
        \,
        dr
        +
        \sum_{ i = 0 }^1 
        \tfrac{V_i(x_2)}{q_i e^{\alpha_i s_1} } 
      \right)
  \! .
\end{split}     \end{equation}
This completes the proof of Lemma~\ref{lem:Hoelder}.
\end{proof}

\subsection{Moment estimates for the first derivative process}
The following lemma, Lemma~\ref{lem:moments:Zv}, provides
a moment estimate for spatial derivatives of solutions of SDEs.
\begin{lemma}[Moment estimates for the first derivative process] \label{lem:moments:Zv}
Assume Setting~\ref{s:exists:C0},
let $D\subseteq\mathcal{O}$ be an open set,
let $s\in[0,T]$, $t\in[s,T]$,
let $Y\colon\Omega\to D$, $Z\colon\Omega\to H$ be $\mathbb{F}_s$/$\mathcal{B}(H)$-measurable,
assume that $\sigma(\{X_{s,t}^x\colon x\in D\})$ and $\mathbb{F}_s$ are independent,
and
assume for all $\omega\in\Omega$ that $(D\ni x\mapsto X_{s,t}^x(\omega)\in H)\in C^1(D,H)$.
Then it holds that
\begin{align} \label{lem:moments:Z:eq2}
\begin{split}
& \left\| \tfrac{\partial}{\partial x}X_{s,t}^Y Z\right\|_{L^{p}(\P;H)} 
	 \leq  
\left\|Z
\exp\!\left(
      \int_s^{t} \!
        \Big(
        \constFun(r)
        +
        \sum_{i=0}^1
         \tfrac{
          \beta_{ i }
        }{
          q_i
           e^{ \alpha_{ i } r }
        }
        \Big)
        \,
        d r
        +
      \sum_{ i = 0 }^1
        \tfrac{
          V_{ i }( Y )
        }{
            q_ie^{ \alpha_{i} s }
        }
    \right)
 \right\|_{L^{p}(\P;H)}.
\end{split}
\end{align} 
\end{lemma}
\begin{proof}[Proof of Lemma~\ref{lem:moments:Zv}]
 Independence of $\sigma(\{X_{s,t}^x\colon x\in D\})$ and $\mathbb{F}_s$,
 a disintegration formula (e.g.~\cite[Lemma 2.3]{HJK+18}),
 the fact that
 for all $\omega\in\Omega$ it holds that $(D\ni x\mapsto X_{s,t}^x(\omega)\in H)\in C^1(D,H)$,
 Fatou's lemma (e.g.\ Lemma 3.10 in~\cite{HutzenthalerJentzen2015Memoires}),
 and Lemma~\ref{lem:Hoelder}
 (applied in the case $t>0$
  for every $y\in D$, $z\in H$, $h\in\{r\in\R\setminus\{0\}\colon y+zr\in D\}$
 with $T\defeq t$,
 $s_1\defeq s$, $s_2\defeq s$, $t_1\defeq t$, $t_2\defeq t$, $x_1\defeq y+zh$, $x_2\defeq y$
 in the notation of
 Lemma~\ref{lem:Hoelder})
 yield that
\begin{align}
\begin{split}
  & \left\| \tfrac{\partial}{\partial x}X_{s,t}^Y Z\right\|_{L^{p}(\P;H)}^{p}
  =\int_{D\times H} \E\Big[\Big\|\tfrac{\partial}{\partial x}X_{s,t}^y z\Big\|_H^{p}\Big]\P\Big((Y,Z)\in d(y,z)\Big)
  \\&
  =\int_{D\times H} \E\Big[\Big\|\liminf_{\{r\in (\R\setminus\{0\})\colon y+zr\in D\}\ni h\to0}\tfrac{\left(X_{s,t}^{y+zh}-X_{s,t}^{y}\right)}{h} \Big\|_H^{p}\Big]\P\Big((Y,Z)\in d(y,z)\Big)
  \\&
  \leq\int_{D\times H} \liminf_{\{r\in (\R\setminus\{0\})\colon y+zr\in D\}\ni h\to0}
   \E\Big[\Big\|\tfrac{\left(X_{s,t}^{y+zh}-X_{s,t}^{y}\right)}{h} \Big\|_H^{p}\Big]\P\Big((Y,Z)\in d(y,z)\Big)
  \\&
  \leq\int_{D\times H}\left( \liminf_{
  \Big\{\substack{r\in (\R\setminus\{0\})\colon\\
     y+zr\in D}\!\!\Big\}\ni h\to0}
  \tfrac{\|y+zh-y\|_H}{h}
  \exp\!\left(
      \int_s^{t} \!
      \Big(
        \constFun(r)
        +
        \sum_{i=0}^1
         \tfrac{
          \beta_{ i }
        }{
          q_ie^{ \alpha_{ i } r }
        }
        \Big)
        \,
        d r
        +
      \sum_{ i = 0 }^1
        \tfrac{
          V_{ i }( y + z h ) + V_{ i }( y )
        }{
           2q_ie^{ \alpha_{i} s }
        }
    \right)
    \right)^{p}
\\&
\qquad\qquad\qquad\qquad
\qquad\qquad\qquad\qquad
\qquad\qquad\qquad\qquad
\qquad\qquad
\P\Big((Y,Z)\in d(y,z)\Big)
\\&
=
\int_{D\times H}
\left(
\|z\|_H
\exp\!\left(
      \int_s^{t} \!
        \Big(
        \constFun(r)
        +
        \sum_{i=0}^1
         \tfrac{
          \beta_{ i }
        }{
          q_ie^{ \alpha_{ i } r }
        }
        \Big)
        \,
        d r
        +
      \sum_{ i = 0 }^1
        \tfrac{
          V_{ i }( y )
        }{
            q_ie^{ \alpha_{i} s }
        }
    \right)
    \right)^{p}
\P\Big((Y,Z)\in d(y,z)\Big)
\\&
=
\left\|Z
\exp\!\left(
      \int_s^{t} \!
        \Big(
        \constFun(r)
        +
        \sum_{i=0}^1
         \tfrac{
          \beta_{ i }
        }{
          q_ie^{ \alpha_{ i } r }
        }
        \Big)
        \,
        d r
        +
      \sum_{ i = 0 }^1
        \tfrac{
          V_{ i }( Y )
        }{
            q_ie^{ \alpha_{i} s }
        }
    \right)
 \right\|_{L^{p}(\P;H)}^{p}.
\end{split}
\end{align}
This implies \eqref{lem:moments:Z:eq2}
and finishes
the proof of Lemma~\ref{lem:moments:Zv}.
\end{proof}

\subsection{Existence of a $C^0$-solution}
The following theorem establishes existence of continuous
solutions of SDEs.
Theorem \ref{thm:exists:C0}
improves existing results
in~\cite{CoxHutzenthalerJentzen2013v3}
and also existing results on strong completeness, e.g., in 
\cite{Zhang2010,Attanasio2010,FlandoliGubinelliPriola2010,FangImkellerZhang2007,FangZhang2005,Schmalfuss1997,SchenkHoppe1996Deterministic,Li1994}.

\begin{theorem}[Existence of a $C^0$-solution] \label{thm:exists:C0}
Assume Setting~\ref{s:exists:C0},
assume that $V_0$, $V_1$ are bounded on every bounded subset
of $\mathcal{O}$,
and
assume that $\dim(H)<\infty$
and
$p\in(2\dim(H)+4,\infty)$.
Then
there exists a measurable function
$
  \mathcal{X} \colon \Delta_T \times \overline{\mathcal{O}} \times \Omega \to \overline{\mathcal{O}}
$
such that
\begin{enumerate}[(i)]
  \item  
for all 
$
  x \in \mathcal{O}, 
$ 
$ 
	s\in [0,T]
$ 
it holds a.s.~that 
$
  (\mathcal{X}^x_{s, t })_{t \in [s,T]} 
  = 
  (X^x_{s,t})_{t \in [s,T]}
$ 
and
\item
for every $ \omega \in \Omega $
it holds
that
$
  \mathcal{X}(\omega) \in C( \Delta_T \times \overline{\mathcal{O}}, \overline{\mathcal{O}})
$.
\end{enumerate}
\end{theorem}
\begin{proof}[Proof of Theorem~\ref{thm:exists:C0}]
Throughout this proof let $K_n\subseteq \R\times\R\times H$,
$n\in \N$, be the sets which satisfy for all $n\in \N$ that
$K_n=\{(s,t,x)\in \Delta_T\times\mathcal{O} \colon s^2+t^2+\|x\|_H^2 \leq n^2 \}$.
Lemma~\ref{lem:Hoelder},
the fact that $\int_0^T \constFun(r) \, dr <\infty$,
and
$\forall n\in\N\colon 
\sup\big(\big\{V_0(x)+V_1(x)\colon(s,t,x)\in K_n\big\}\cup\{0\}\big)<\infty$
yield for all $n\in\N$ that
\begin{equation}\label{eq:C0.ass.Hoelder} 
  \sup\bigg(\bigg\{
  \tfrac{ 
  \big(\E\big[  \|X_{s_1,t_1}^{x_1} - X_{s_2,t_2}^{x_2}   \|^{p}_{H}\big]\big)^{\frac{1}{p}}
  }
  {\left( |s_1-s_2|^2+|t_1-t_2|^2+\|x_1-x_2\|_H^{2}  \right)^{\frac{1}{4}}}
  \colon
  \substack{
  (s_1,t_1,x_1), (s_2,t_2,x_2) \in  K_n \colon
  \\
  (s_1,t_1,x_1)\neq (s_2,t_2,x_2)
  }
  \bigg\}\cup\{0\}\bigg)
  <\infty.
\end{equation}
In particular this implies for all $n\in\N$ that
\begin{equation}  \begin{split}
  &\sup\bigg(\bigg\{
  \Big(\E\Big[  \|X_{s,t}^{x}   \|^{p}_{H}\Big]\Big)^{\frac{1}{p}}
  \colon (s,t,x) \in  K_n
  \bigg\}\cup\{0\}\bigg)
  \\&
  \leq
  \sup\bigg(\bigg\{
  \frac{ 
  \big(\E\big[  \|X_{s,t}^{x} - X_{s,s}^{x}   \|^{p}_{H}\big]\big)^{\frac{1}{p}}
  }
  {\left( |s-s|^2+|t-s|^2+\|x-x\|_H^{2}  \right)^{\frac{1}{4}}}
  \sqrt{T}+\|x\|_H
  \colon (s,t,x) \in  K_n
  \bigg\}\cup\{0\}\bigg)
  <\infty.
\end{split}     \end{equation}
This, \eqref{eq:C0.ass.Hoelder},
Proposition~\ref{p:KolChen}
(applied with 
$H \defeq  \R\times\R \times H$,
$D\defeq \Delta_T \times \mathcal{O}$,
$E\defeq H$,
$F\defeq \overline{\mathcal{O}}$,
$\alpha\defeq \nicefrac12$,
$X\defeq \left( \Delta_T \times \mathcal{O} \ni (s,t,x) \mapsto X_{s,t}^x \in \overline{\mathcal{O}} \right)$
in the notation of Proposition~\ref{p:KolChen}),
and path continuity of $X_{s,\cdot}^{x}$, $s\in[0,T]$, $x\in\mathcal{O}$
establish the existence of a
measurable function
$
  \mathcal{X} \colon \Delta_T  \times \overline{\mathcal{O}} \times \Omega \to \overline{\mathcal{O}}
$
which satisfies that for all $\omega \in \Omega$ it holds that 
$\mathcal{X}(\omega) \in C(\Delta_T \times \overline{\mathcal{O}} , \overline{\mathcal{O}})$
and which satisfies
that
for all $x\in\mathcal{O}$, $s\in[0,T]$
it holds a.s.~that
$(\mathcal{X}_{s,t}^x)_{t\in[s,T]}=(X_{s,t}^x)_{t\in[s,T]}$.
This 
completes the proof of Theorem~\ref{thm:exists:C0}.
\end{proof}

\section{Existence of a $C^1$-solution}\label{sec:4}
In this section we prove 
a strong local H\"older estimate for the first
derivative process in Lemma \ref{2l:local.Lip.C2} below,
a moment estimate for the second derivative process
in Lemma \ref{lem:moments:ZZv} below,
and
establish existence of a continuously differentiable solution
under suitable assumptions
in Theorem \ref{thm:C1} below.
First, we introduce the setting for these results.
\begin{sett} \label{sett:exists:C1}
Let $( H, \left< \cdot , \cdot \right>_H, \left\| \cdot \right\|_H )$ and $( U, \left< \cdot , \cdot \right>_U, \left\| \cdot \right\|_U )$
be separable $\R$-Hilbert spaces,
let
$ T \in (0,\infty) $,
let $(\Omega, \F, \P, (\mathbb{F}_{t})_{t\in [0,T]})$ be a filtered probability space
satisfying the usual conditions,
let 
$
  (W_t)_{t\in[0,T]}
$
  be an $\textup{Id}_U$-cylindrical $(\mathbb{F}_t)_{t\in[0,T]}$-Wiener process,
let $\Delta_T =\{(s,t)\in [0,T]^2 \colon s \leq t\}$,
let $O \subseteq H$ be an open set,
let $\mathcal{O}\subseteq O$ be a convex set,
let $ \mu \in C^1(O,H)$, 
$ \sigma \in C^1(O, \HS(U,H))$,
for all $s\in [0,T]$, $x\in \mathcal{O}$
let $ X^x_{s,\cdot} \colon [s,T] \times \Omega \to \mathcal{O} $
be an $(\mathbb{F}_t)_{t\in[s,T]}$-adapted stochastic process
with continuous sample paths
which satisfies that for all $t\in [s,T]$ it holds a.s.~that
\begin{align} \label{eq:def:X}
  X^x_{s,t} = 
  x
  + \int_s^{ t } \mu(X^x_{s,r} ) \, dr
  +
  \int_s^{ t } \sigma(X^x_{s,r} ) \, dW_r,
\end{align}
let $ \alpha_0,\,\alpha_1,\beta_0,\,\beta_1,c \in [0,\infty)$,
$ 
  V_0 
$,
$
  V_1  \in C^{ 2 }( O , [0,\infty) ) 
$,
let
$ 
  \bar{V} \colon [0,T] \times \mathcal{O} \to [0,\infty)
$
be a measurable function,
assume
for all
$ i \in \{ 0, 1 \} $,
$t\in[0,T]$,
$x\in\mathcal{O}$
that 
$\P\big(	\int_t^T  \bar{V}(r, X_{t,r}^x)  \,dr <\infty\big)=1$ and
\begin{equation}
\begin{split}
  &\Big\langle
  \mu( x )
  ,
  (\nabla V_i)(x)
  \Big\rangle_H
  +
  \tfrac{ 1 }{ 2 }
  \operatorname{trace}\!\Big(
    \sigma(x) [\sigma(x)]^* 
    ( \operatorname{Hess} V_i )( x )
  \Big)
  \\&
  +
  \tfrac{ 
    1
  }{ 
    2 
    e^{ 
      \alpha_{ i } {t} 
    }
  }
    \|
      \sigma( x )^* ( \nabla V_{ i } )( x )
    \|_U^2
  +
  \mathbbm{1}_{
    \{ 1 \}
  }(i)
  \cdot
  \bar{V}(t,x)
\leq
  \alpha_{ i } V_{ i }(x)
  +
  \beta_{ i },
\end{split}
\end{equation}
let 
$
  \constFun \colon [0,T] \to [0,\infty)
$
be a measurable function which satisfies that
$\int_0^T \constFun(r) \, dr <\infty$,
let $p\in[2,\infty)$, $\theta,\delta\in[0,\infty)$, $q_0,q_1\in(0,\infty)$,
$\gamma\in[\tfrac{1}{p},\infty)$
satisfy that $\tfrac{\theta}{2p(1+\theta)^2(1+\delta)}
=\tfrac{1}{q_0}+\tfrac{1}{q_1}$,
assume that for all $t\in[0,T]$, $x,y\in\mathcal{O}$, $v\in H\setminus\{0\}$
it holds that
\begin{equation}  \begin{split}\label{eq:Lip.ass.strong}
  &\Big\langle v,\smallint_0^1\mu'(\lambda x+(1-\lambda)y)\,d\lambda \,\,v+\delta v
  \Big\rangle_H
  +\tfrac{1+\delta}{2}
  \Big\|\smallint_0^1\sigma'(\lambda x+(1-\lambda)y)\,d\lambda \,\,v
  \Big\|_{\HS(U,H)}^2
  \\&
  \tfrac{\big(p(1+\theta)^2(1+\delta)-1\big)\big\|
  \big\langle v,\smallint_0^1\sigma'(\lambda x+(1-\lambda)y)\,d\lambda\,\, v
  \big\rangle_H\big\|_{\HS(U,\R)}^2}{\|v\|_H^2}
  \leq \|v\|^2_H\cdot\Big(\phi(t)
  +\tfrac{V_0(x)+V_0(y)}{2q_0T e^{\alpha_0t}}
  +\tfrac{\bar{V}(t,x)+\bar{V}(t,y)}{2q_1e^{\alpha_1t}}
  \Big),
\end{split}     \end{equation}
assume for all $x\in\mathcal{O}$ that
\begin{align}
\begin{split}
 \max\left\{ \|\mu(x)\|_H , \|\sigma(x)\|_{\HS(U,H)} \right\}
 \leq  c(1+V_0(x))^{\gamma},
\end{split}
\end{align}
assume for all $x,y\in\mathcal{O}$ that
\begin{align} \label{eq:growth.Dmue}
\begin{split}
 &\max\left\{
    \left\|\smallint_0^1\mu'\big(\lambda x+(1-\lambda)y\big)\,d\lambda 
    \right\|_{L(H,H)} ,
    \left\|\smallint_0^1\sigma'\big(\lambda x+(1-\lambda)y\big)\,d\lambda 
    \right\|_{L(H,\HS(U,H))}
 \right\}
 \\&
 \leq  c\left(2+V_0(x)+V_0(y)\right)^{\gamma},
\end{split}
\end{align}
assume for all $x_1,x_2,x_3,x_4\in\mathcal{O}$
that
\begin{equation}  \begin{split}\label{eq:Dmu.localLip}
  &\max\Big\{\big\|\smallint_0^1
    \mu'\big(\lambda x_1+(1-\lambda)x_2\big)
      -\mu'\big(\lambda x_3+(1-\lambda)x_4\big)
    \,d\lambda
  \big\|_{L(H,H)},
  \\&\qquad\quad
  \big\|\smallint_0^1\sigma'\big(\lambda x_1+(1-\lambda)x_2\big)-\sigma'\big(\lambda x_3+(1-\lambda)x_4\big)
    \,d\lambda
  \big\|_{L(H,\HS(U,H))}
  \Big\}
  \\&
  \leq c\smallint_0^1\lambda\|x_1-x_3\|_H+(1-\lambda)\|x_2-x_4\|_H\,d\lambda
  \,\big(4+\smallsum_{j=1}^4 V_0(x_i)\big)^{\gamma},
\end{split}     \end{equation}
and
for all $(s,t)\in\Delta_T$,
$x\in \mathcal{O}$, $v\in H$, $h\in \R\setminus\{0\}$ with $ x + v h \in \mathcal{O}$
let
$
	D_{s,t}^{x,h}(v) \colon \Omega \to H
$
be the function which satisfies that
\begin{align} \label{eq:def:Zvy}
D_{s,t}^{x,h}(v)= 
\frac{X_{s,t}^{x+h v} -X_{s,t}^{x}}{h}.
\end{align}
\end{sett}

\begin{lemma}\label{l:C1.implies.C0}
  Assume Setting~\ref{sett:exists:C1}
  and let $t\in[0,T]$, $x,y\in \mathcal{O}$.
  Then it holds that
\begin{equation}  \begin{split}\label{eq:ass.41}
  &\big\langle x-y,\mu(x)-\mu(y)\big\rangle_H+\tfrac{1}{2}\big\|\sigma(x)-\sigma(y)\big\|^2_{\HS(U,H)}
  +\tfrac{2p(1+\theta)^2(1+\delta)-2}{2}
  \tfrac{
    \|\langle x-y,\sigma(x)-\sigma(y)\rangle_H\|_{\HS(U,\R)}^2
    }{\|x-y\|_H^2}
  \\&
  \leq \|x-y\|^2_H\cdot\Big(\phi(t)
  +\tfrac{V_0(x)+V_0(y)}{2q_0Te^{\alpha_0t}}
  +\tfrac{\bar{V}(t,x)+\bar{V}(t,y)}{2q_1e^{\alpha_1t}}
  \Big).
\end{split}     \end{equation}
\end{lemma}
\begin{proof}[Proof of Lemma~\ref{l:C1.implies.C0}]
  Convexity of $\mathcal{O}$,
  the fundamental theorem of calculus,
  and 
  assumption~\eqref{eq:Lip.ass.strong}
  (applied in the case $x\neq y$ with $v\defeq x-y$
  in the notation of 
  assumption~\eqref{eq:Lip.ass.strong})
  yield that
\begin{equation}  \begin{split}
  &\big\langle x-y,\mu(x)-\mu(y)\big\rangle_H+\tfrac{1}{2}\big\|\sigma(x)-\sigma(y)\big\|^2_{\HS(U,H)}
  +\tfrac{2p(1+\theta)^2(1+\delta)-2}{2}
  \tfrac{
    \|\langle x-y,\sigma(x)-\sigma(y)\rangle_H\|_{\HS(U,\R)}^2
    }{\|x-y\|_H^2}
  \\&
  =\big\langle x-y,\smallint_0^1\mu'\big(\lambda x+(1-\lambda)y\big)\,d\lambda \,(x-y)\big\rangle_H
  +\tfrac{1}{2}\big\|\smallint_0^1\sigma'\big(\lambda x+(1-\lambda)y\big)\,d\lambda \,(x-y)\big\|^2_{\HS(U,H)}
  \\&\qquad
  +(p(1+\theta)^2(1+\delta)-1)
    \tfrac{
    \big\|\big\langle x-y,\smallint_0^1\sigma'(\lambda x+(1-\lambda)y)\,d\lambda \,(x-y)\big\rangle_H\big\|_{\HS(U,\R)}^2}{\|x-y\|_H^2}
  \\&
  \leq \|x-y\|^2_H\cdot\Big(\phi(t)
  +\tfrac{V_0(x)+V_0(y)}{2q_0Te^{\alpha_0t}}
  +\tfrac{\bar{V}(t,x)+\bar{V}(t,y)}{2q_1e^{\alpha_1t}}
  \Big).
\end{split}     \end{equation}
 This completes the proof of Lemma~\ref{l:C1.implies.C0}.
\end{proof}

\begin{lemma}[Difference processes satisfy linear SDEs] \label{lem:Zv:eq}
Assume Setting~\ref{sett:exists:C1}
and let
$x\in \mathcal{O}$, $v\in H$, $h\in \R\setminus\{0\}$, $(s,t)\in\Delta_T$ satisfy
that $x+vh\in \mathcal{O}$.
Then
it holds a.s.~that
\begin{align} \label{eq:eq:Zvy}
\begin{split}
D_{s,t}^{x,h}(v)
	& = 
v +
\int_s^t \int_0^1 \mu' \big( X_{s,r}^{x} + \lambda (X_{s,r}^{x + v h} - X_{s,r}^{x}) \big) D_{s,r}^{x,h}(v) 
\, d\lambda \, d r
	\\
	& \quad
+ \int_s^t \int_0^1 \sigma' \big( X_{s,r}^{x} + \lambda (X_{s,r}^{x + v h} - X_{s,r}^{x}) \big) D_{s,r}^{x,h}(v) 
\, d\lambda \, d W_r.
\end{split}
\end{align}
\end{lemma}
\begin{proof}[Proof of Lemma~\ref{lem:Zv:eq}]
Note that~\eqref{eq:def:Zvy},
\eqref{eq:def:X},
$\mu\in C^1(O,H)$, $\sigma\in C^1(O,\HS(U,H))$,
convexity of $\mathcal{O}$,
and
the fundamental theorem of calculus
imply that it holds a.s.~that
\begin{align}
\begin{split}
D_{s,t}^{x,h}(v)
	& = 
\tfrac{x+vh-x}{h} + 
\int_s^t \tfrac{ \mu(X_{s,r}^{x + v h}) - \mu(X_{s,r}^{x}) }{h} \, dr
+ \int_s^t \tfrac{ \sigma(X_{s,r}^{x + v h}) - \sigma(X_{s,r}^{x}) }{h} \, d W_r
\\
	& = 
v +
\int_s^t \int_0^1 \mu' \big( X_{s,r}^{x} + \lambda (X_{s,r}^{x + v h} - X_{s,r}^{x}) \big) D_{s,r}^{x,h}(v) 
\, d\lambda \, d r
	\\
	& \quad
+ \int_s^t \int_0^1 \sigma' \big( X_{s,r}^{x} + \lambda (X_{s,r}^{x + v h} - X_{s,r}^{x}) \big) D_{s,r}^{x,h}(v) 
\, d\lambda \, d W_r
.
\end{split}
\end{align}
This proves~\eqref{eq:eq:Zvy}
and finishes
the proof of Lemma~\ref{lem:Zv:eq}.
\end{proof}

%


\subsection{Strong local H\"older estimate}\label{sec:4.1}
The following lemma proves strong local H\"older continuity
of the difference quotients.

\begin{lemma}[Strong local H\"older estimate for difference quotients]\label{l:local.Lip.C1}
  Assume Setting~\ref{sett:exists:C1} and let 
  $s_1,s_2,t_1,t_2\in[0,T]$, $x_1,x_2\in\mathcal{O}$, $v_1,v_2\in H$, $h_1,h_2\in\R\setminus\{0\}$ satisfy that
  $s_1\leq t_1$, $s_2\leq t_2$, $x_1+v_1h_1,x_2+v_2h_2\in \mathcal{O}$,
  and $s_1\leq s_2$.
  Then it holds that
  \begin{equation}  \begin{split}
    &\Big\|
     D_{s_1,t_1}^{x_1,h_1}(v_1)-D_{s_2,t_2}^{x_2,h_2}(v_2)
     \Big\|_{L^{p}(\P;H)}
   \\&
   \leq
 \bigg[
  \Big(\Big(\|x_1-x_2\|_H+\|v_1h_1-v_2h_2\|_H\Big)\sqrt{t_2-s_2}
  +\sqrt{|s_1-s_2|}+
  \sqrt{|t_1-t_2|}\Big)
  \|v_1\|_H
  +\|v_1-v_2\|_H
  \bigg]
  \\&\quad
  \cdot
    e^{2\alpha_0 \gamma T}
    \left(\tfrac{4p(1+\theta)^2(1+\delta)(1+\gamma)}{\min\{\delta,1\}}
    +2\sqrt{T}
    +\smallint_0^T \tfrac{2\beta_0}{e^{\alpha_0 r}}\,dr
    +2\max_{\iota\in\{0,1\},j\in\{1,2\}}
    V_0(x_j+\iota v_jh_j)
  \right)^{2\gamma+2}
  \\&\quad
  \cdot
  (1+c)^2\max_{\iota\in\{0,1\},j\in\{1,2\}}
  \exp\!\left(
        3\int_{0}^{T}  
        \Big(
        \constFun(r)
        +
        \sum_{i=0}^1
        \tfrac{
          \beta_{ i } 
        }{
          q_{i} e^{ \alpha_{ i } r }
        }
        \Big)
        \,
        dr
        +
        3\sum_{ i = 0 }^1 
        \tfrac{V_i(x_j+\iota v_jh_j)}{ q_{i}e^{\alpha_i s_1} } 
      \right).
  \end{split}     \end{equation}
\end{lemma}
\begin{proof}[Proof of Lemma~\ref{l:local.Lip.C1}]
  Throughout this proof let $Y,a,\zeta^{\mu}\colon[0,T-s_2]\times\Omega\to H$, $b,\zeta^{\sigma}\colon[0,T-s_2]\times\Omega\to\HS(U,H)$,
  and $\eta\colon[0,T-s_2]\times\Omega\to L(H,\HS(U,H))$
  be the functions which satisfy for all $r\in[0,T-s_2]$ that
  \begin{align}
    \label{eq:ar}
    a_r&=
     \int_0^1
     \Big(
     \mu' \big( X_{s_1,s_2+r}^{x_1} + \lambda (X_{s_1,s_2+r}^{x_1 + v_1 h_1} - X_{s_1,s_2+r}^{x_1}) \big) D_{s_1,s_2+r}^{x_1,h_1}(v_1)
     \\&\qquad\qquad\nonumber
     -
     \mu' \big( X_{s_2,s_2+r}^{x_2} + \lambda (X_{s_2,s_2+r}^{x_2 + v_2 h_2} - X_{s_2,s_2+r}^{x_2}) \big) D_{s_2,s_2+r}^{x_2,h_2}(v_2)
     \Big)
     \,d\lambda,
  \\
    \label{eq:br}
    b_r&=
     \int_0^1
     \Big(
     \sigma' \big( X_{s_1,s_2+r}^{x_1} + \lambda (X_{s_1,s_2+r}^{x_1 + v_1 h_1} - X_{s_1,s_2+r}^{x_1}) \big) D_{s_1,s_2+r}^{x_1,h_1}(v_1)
     \\&\qquad\qquad\nonumber
     -
     \sigma' \big( X_{s_2,s_2+r}^{x_2} + \lambda (X_{s_2,s_2+r}^{x_2 + v_2 h_2} - X_{s_2,s_2+r}^{x_2}) \big) D_{s_2,s_2+r}^{x_2,h_2}(v_2)
     \Big)
     \,d\lambda,
  \\
    \label{eq:Yr}
     Y_r&=D_{s_1,s_2+r}^{x_1,h_1}(v_1)-D_{s_2,s_2+r}^{x_2,h_2}(v_2),
   \\
     \eta_r&=\label{eq:etar}
     \int_0^1
     \sigma' \big( X_{s_2,s_2+r}^{x_2} + \lambda (X_{s_2,s_2+r}^{x_2 + v_2 h_2} - X_{s_2,s_2+r}^{x_2}) \big)\,d\lambda,
   \\
     \zeta_r^{\mu}&=\label{eq:zeta.mue}
     \int_0^1
     \Big(
     \mu' \big(  \lambda X_{s_1,s_2+r}^{x_1 + v_1 h_1} +(1-\lambda) X_{s_1,s_2+r}^{x_1} \big)
     \\&\qquad\qquad\nonumber
     -
     \mu' \big(  \lambda X_{s_2,s_2+r}^{x_2 + v_2 h_2} +(1-\lambda) X_{s_2,s_2+r}^{x_2} \big)
     \Big)
     \,d\lambda D_{s_1,s_2+r}^{x_1,h_1}(v_1),
   \\
     \zeta_r^{\sigma}&=\label{eq:zeta.sigma}
     \int_0^1
     \Big(
     \sigma' \big(  \lambda X_{s_1,s_2+r}^{x_1 + v_1 h_1} +(1-\lambda) X_{s_1,s_2+r}^{x_1} \big)
     \\&\qquad\qquad\nonumber
     -
     \sigma' \big(  \lambda X_{s_2,s_2+r}^{x_2 + v_2 h_2} +(1-\lambda) X_{s_2,s_2+r}^{x_2} \big)
     \Big)
     \,d\lambda D_{s_1,s_2+r}^{x_1,h_1}(v_1).
  \end{align}
  Note that $Y,a,b,\eta,\zeta^{\mu},\zeta^{\sigma}$
  are $(\mathbb{F}_{s_2+r})_{r\in[0,T-s_2]}$-adapted stochastic processes
  with continuous sample paths 
  and, therefore, measurable.
  Lemma~\ref{lem:Zv:eq} implies that for all $t\in[0,T-s_2]$ it holds a.s.~that
  \begin{equation}  \begin{split}\label{eq:D.Y}
     &Y_t
     =D_{s_1,s_2+t}^{x_1,h_1}(v_1)-D_{s_2,s_2+t}^{x_2,h_2}(v_2)
     =
     D_{s_1,s_2}^{x_1,h_1}(v_1)-v_2
     +
     \int_{s_2}^{s_2+t}a_{r-s_2}\,dr
     +
     \int_{s_2}^{s_2+t}b_{r-s_2}\,dW_r
     \\&
     =
     D_{s_1,s_2}^{x_1,h_1}(v_1)-v_2
     +
     \int_{0}^{t}a_{r}\,dr
     +
     \int_{0}^{t}b_{r}\,d(W_{s_2+r}-W_{s_2}).
  \end{split}     \end{equation}
  We consider the one-sided affine-linear growth condition for the It\^o process $Y$.
  Equations~\eqref{eq:ar}, \eqref{eq:Yr},
  and \eqref{eq:zeta.mue} imply for all $r\in[0,T-s_2]$ that
  \begin{equation}  \begin{split}\label{eq:for.a}
    a_r
    &=
     \int_0^1
     \mu' ( X_{s_1,s_2+r}^{x_1} + \lambda (X_{s_1,s_2+r}^{x_1 + v_1 h_1} - X_{s_1,s_2+r}^{x_1}) ) D_{s_1,s_2+r}^{x_1,h_1}(v_1)
     \\&\quad
     -
     \mu' ( X_{s_2,s_2+r}^{x_2} + \lambda (X_{s_2,s_2+r}^{x_2 + v_2 h_2} - X_{s_2,s_2+r}^{x_2}) ) D_{s_2,s_2+r}^{x_2,h_2}(v_2)
     \,d\lambda
   \\&
     =\smallint_0^1
     \mu' ( X_{s_2,s_2+r}^{x_2} + \lambda (X_{s_2,s_2+r}^{x_2 + v_2 h_2} - X_{s_2,s_2+r}^{x_2}) ) 
     \,d\lambda\, Y_r
   \\&\quad
   +
     \smallint_0^1\mu' (\lambda X_{s_1,s_2+r}^{x_1 + v_1 h_1} +(1-\lambda) X_{s_1,s_2+r}^{x_1} )
     -
     \mu' (  \lambda X_{s_2,s_2+r}^{x_2 + v_2 h_2} +(1-\lambda) X_{s_2,s_2+r}^{x_2} )
     \,d\lambda \,D_{s_1,s_2+r}^{x_1,h_1}(v_1)
   \\&
     =\smallint_0^1
     \mu' ( X_{s_2,s_2+r}^{x_2} + \lambda (X_{s_2,s_2+r}^{x_2 + v_2 h_2} - X_{s_2,s_2+r}^{x_2}) ) 
     \,d\lambda\, Y_r+\zeta_r^{\mu}.
  \end{split}     \end{equation}
  Analogously, equations~\eqref{eq:br}, \eqref{eq:Yr}, \eqref{eq:etar}, and~\eqref{eq:zeta.sigma}
  imply for all $r\in[0,T-s_2]$ that
  \begin{equation}  \begin{split}\label{eq:for.b}
    b_r&=
     \smallint_0^1
     \sigma' \big( X_{s_2,s_2+r}^{x_2} + \lambda (X_{s_2,s_2+r}^{x_2 + v_2 h_2} - X_{s_2,s_2+r}^{x_2}) \big) 
     \,d\lambda\, Y_r+\zeta_r^{\sigma}
   =
    \eta_rY_r+\zeta^{\sigma}_r.
  \end{split}     \end{equation}
  Equation~\eqref{eq:for.a},
  the Cauchy-Schwarz inequality, and Young's inequality
  yield for all $r\in[0,T-s_2]$ that
  \begin{equation}  \begin{split}\label{eq:estimate.a}
    \langle Y_r,a_r\rangle_H
    &\leq \left\langle Y_r,
     \smallint_0^1
     \mu' \big( X_{s_2,s_2+r}^{x_2} + \lambda (X_{s_2,s_2+r}^{x_2 + v_2 h_2} - X_{s_2,s_2+r}^{x_2}) \big) 
     \,d\lambda \,Y_r
     \right\rangle_H
    +\big\| Y_r\big\|_H
    \big\| \zeta_r^{\mu}\big\|_H
   \\&
    \leq\left\langle Y_r,
     \smallint_0^1
     \mu' \big( X_{s_2,s_2+r}^{x_2} + \lambda (X_{s_2,s_2+r}^{x_2 + v_2 h_2} - X_{s_2,s_2+r}^{x_2}) \big) 
     \,d\lambda \,Y_r+\delta Y_r
     \right\rangle _H
   +\tfrac{1}{4\delta}
   \big\|\zeta_r^{\mu}   \big\|_{H}^2.
  \end{split}     \end{equation}
  Similarly, equation~\eqref{eq:for.b}, the Cauchy-Schwarz inequality, and Young's inequality
  imply for all $r\in[0,T-s_2]$ that
  \begin{equation}  \begin{split}\label{eq:estimate.b}
    &\tfrac{1}{2}\|b_r\|_{\HS(U,H)}^2
    +\tfrac{p(1+\theta)-2}{2}\tfrac{\left\|
    \langle 
     Y_r, b_r
    \rangle_H
    \right\|_{\HS(U,\R)}^2
    }{
      \|Y_r\|_H^2
    }
    =
    \tfrac{1}{2}\|\eta_rY_r+\zeta^{\sigma}_r\|_{\HS(U,H)}^2
    +\tfrac{p(1+\theta)-2}{2}\tfrac{\left\|
    \langle 
     Y_r, \eta_rY_r+\zeta^{\sigma}_r
    \rangle_H
    \right\|_{\HS(U,\R)}^2
    }{
      \|Y_r\|_H^2
    }
    \\&
    =
    \tfrac{1}{2}\|\eta_rY_r\|_{\HS(U,H)}^2
    +\langle \eta_rY_r,\zeta^{\sigma}_r\rangle_{\HS(U,H)}
    +
    \tfrac{1}{2}\|\zeta^{\sigma}_r\|_{\HS(U,H)}^2
    \\&\qquad
    +\tfrac{p(1+\theta)-2}{2}\tfrac{\left\|
    \langle 
     Y_r, \eta_rY_r
    \rangle_H
    \right\|_{\HS(U,\R)}^2
    }{
      \|Y_r\|_H^2
    }
    +\tfrac{p(1+\theta)-2}{2}\tfrac{
    2\left\langle \langle Y_r,\eta_rY_r\rangle_H,\langle Y_r,\zeta^{\sigma}_r\rangle_H
    \right\rangle_{\HS(U,\R)}
    }{
      \|Y_r\|_H^2
    }
    +\tfrac{p(1+\theta)-2}{2}\tfrac{\left\|
    \langle 
     Y_r, \zeta^{\sigma}_r
    \rangle_H
    \right\|_{\HS(U,\R)}^2
    }{
      \|Y_r\|_H^2
    }
    \\&
    \leq
    \tfrac{1}{2}\|\eta_rY_r\|_{\HS(U,H)}^2
    +\| \eta_rY_r\|_{\HS(U,H)}\left\|\zeta^{\sigma}_r\right\|_{\HS(U,H)}
    +
    \tfrac{1}{2}\|\zeta^{\sigma}_r\|_{\HS(U,H)}^2
    \\&\qquad
    +\tfrac{p(1+\theta)-2}{2}\tfrac{\left\|
    \langle 
     Y_r, \eta_rY_r
    \rangle_H
    \right\|_{\HS(U,\R)}^2
    }{
      \|Y_r\|_H^2
    }
    +\tfrac{p(1+\theta)-2}{2}\tfrac{
    2\left\| \langle Y_r,\eta_rY_r\rangle_H\right\|_{\HS(U,\R)}
    \left\|\langle Y_r,\zeta^{\sigma}_r\rangle_H\right\|_{\HS(U,\R)}
    }{
      \|Y_r\|_H^2
    }
    +\tfrac{p(1+\theta)-2}{2}\tfrac{\left\|
    \langle 
     Y_r, \zeta^{\sigma}_r
    \rangle_H
    \right\|_{\HS(U,\R)}^2
    }{
      \|Y_r\|_H^2
    }
    \\&
    \leq
    \tfrac{1+\delta}{2}\|\eta_rY_r\|_{\HS(U,H)}^2
    +\tfrac{(p(1+\theta)-2)(1+\delta)}{2}\tfrac{\left\|
    \langle 
     Y_r, \eta_rY_r
    \rangle_H
    \right\|_{\HS(U,\R)}^2
    }{
      \|Y_r\|_H^2
    }
    +
    (\tfrac{1+\delta^{-1}}{2}+\tfrac{(p(1+\theta)-2)(1+\delta^{-1})}{2})\|\zeta^{\sigma}_r\|_{\HS(U,H)}^2
    \\&
    =\tfrac{1+\delta}{2}\|\eta_rY_r\|_{\HS(U,H)}^2
    +\tfrac{(p(1+\theta)-2)(1+\delta)}{2}\tfrac{\left\|
    \langle 
     Y_r, \eta_rY_r
    \rangle_H
    \right\|_{\HS(U,\R)}^2
    }{
      \|Y_r\|_H^2
    }
    +
    \tfrac{(p(1+\theta)-1)(1+\delta^{-1})}{2}\|\zeta^{\sigma}_r\|_{\HS(U,H)}^2.
  \end{split}     \end{equation}
  Next, \eqref{eq:estimate.a}, \eqref{eq:estimate.b}, \eqref{eq:etar},
  the fact that $(p(1+\theta)-2)(1+\delta)\leq 2p(1+\theta)^2(1+\delta)-2
  $,
  and the hypothesis~\eqref{eq:Lip.ass.strong}
  (applied for every $r\in[0,T-s_2]$ with $t\defeq s_2+r$,
  $x\defeq X_{s_2,s_2+r}^{x_2+v_2h_2}$,
  $y\defeq X_{s_2,s_2+r}^{x_2}$,
  $v\defeq Y_r$ in the notation of 
  hypothesis~\eqref{eq:Lip.ass.strong})
  imply for all $r\in[0,T-s_2]$ that
  \begin{equation}  \begin{split}
    &\langle 
     Y_r, a_r
    \rangle_H
    +\tfrac{1}{2}\|b_r\|_{\HS(U,H)}^2
    +\tfrac{p(1+\theta)-2}{2}\tfrac{\left\|
    \langle 
     Y_r, b_r
    \rangle_H
    \right\|_{\HS(U,\R)}^2
    }{
      \|Y_r\|_H^2
    }
    \\&
    \leq\left\langle Y_r,
     \smallint_0^1
     \mu' \big( X_{s_2,s_2+r}^{x_2} + \lambda (X_{s_2,s_2+r}^{x_2 + v_2 h_2} - X_{s_2,s_2+r}^{x_2}) \big) 
     \,d\lambda \,Y_r+\delta Y_r
     \right\rangle _H
    +\tfrac{1+\delta}{2}\|\eta_rY_r\|_{\HS(U,H)}^2
   \\&
   \qquad
    +\tfrac{(p(1+\theta)-2)(1+\delta)}{2}\tfrac{\left\|
    \langle 
     Y_r, \eta_rY_r
    \rangle_H
    \right\|_{\HS(U,\R)}^2
    }{
      \|Y_r\|_H^2
    }
    +
    \tfrac{(p(1+\theta)-1)(1+\delta^{-1})}{2}\|\zeta^{\sigma}_r\|_{\HS(U,H)}^2
  + \tfrac{1}{4\delta}
   \|\zeta_r^{\mu}   \|_{H}^2
    \\&
    \leq \|Y_r\|_H^2\left(\phi(s_2+r)
  +\tfrac{V_0(X_{s_2,s_2+r}^{x_2+v_2h_2})+V_0(X_{s_2,s_2+r}^{x_2})}{2q_0Te^{\alpha_0(s_2+r)}}
  +\tfrac{\bar{V}(s_2+r,X_{s_2,s_2+r}^{x_2+v_2h_2})+\bar{V}(s_2+r,X_{s_2,s_2+r}^{x_2})}{2q_1e^{\alpha_1(s_2+r)}}
  \right)
  \\&\qquad
    +
    \tfrac{(p(1+\theta)-1)(1+\delta^{-1})}{2}\|\zeta^{\sigma}_r\|_{\HS(U,H)}^2
  + \tfrac{1}{4\delta}
   \|\zeta_r^{\mu}   \|_{H}^2.
  \end{split}     \end{equation}
  This, \eqref{eq:D.Y}, nonnegativity of $\phi,V_0,\bar{V}$, Proposition~\ref{prop:moments:hilbert}
  (applied in the case $s_2<T$ for every $z\in\R^d$ with $p\defeq p(1+\theta)$,
  $T\defeq T-s_2$,
  $\P\defeq \P(\cdot|D_{s_1,s_2}^{x_1,h_1}(v_1)=z)$, 
  $(\mathbb{F}_t)_{t\in[0,T-s_2]}\defeq (\mathbb{F}_{s_2+t})_{t\in[0,T-s_2]}$,
  $W\defeq (W_{s_2+r}-W_{s_2})_{r\in[0,T-s_2]}$, $X\defeq Y$,
  \begin{align}
    \alpha&\defeq \Big(\phi(s_2+r)
    +\tfrac{V_0(X_{s_2,s_2+r}^{x_2+v_2h_2})+V_0(X_{s_2,s_2+r}^{x_2})}{2q_0Te^{\alpha_0(s_2+r)}}
    +\tfrac{\bar{V}(s_2+r,X_{s_2,s_2+r}^{x_2+v_2h_2})+\bar{V}(s_2+r,X_{s_2,s_2+r}^{x_2})}{2q_1e^{\alpha_1(s_2+r)}}\Big)_{r\in[0,T-s_2]},
  \\
  \beta&\defeq 
    \Big(\left((p(1+\theta)-1)(1+\delta^{-1})\|\zeta^{\sigma}_r\|_{\HS(U,H)}^2
  + \tfrac{1}{2\delta}
   \|\zeta_r^{\mu}\|_H^2\right)^{\frac{1}{2}}\Big)_{r\in[0,T-s_2]},
  \end{align}
  $q_1\defeq p$, $q_2\defeq p(1+\theta)/\theta$
  in the notation of
  Proposition~\ref{prop:moments:hilbert}),
  the triangle inequality,
  the fact that $\tfrac{\theta}{p(1+\theta)}\geq \tfrac{1}{q_0}+\tfrac{1}{q_1}$,
  and Lemma~\ref{lem:multiple_exp:guess}
  (applied in the case $t+s_2>0$ with
  $T\defeq t+s_2$,
  $s\defeq s_2$,
  $X^{1}\defeq X_{s_2,\cdot}^{x_2+v_2h_2}$,
  $X^{2}\defeq X_{s_2,\cdot}^{x_2+v_2h_2}$,
  $X^{3}\defeq X_{s_2,\cdot}^{x_2}$,
  $X^{4}\defeq X_{s_2,\cdot}^{x_2}$,
  $q\defeq 2p(1+\theta)^2(1+\delta)/\theta$
  in the notation of
  Lemma~\ref{lem:multiple_exp:guess})
  yield that for all $t\in[0,T-s_2]$
  it holds a.s.\ that
  \begin{equation}  \begin{split}
    &\Big\|
     D_{s_1,s_2+t}^{x_1,h_1}(v_1)-D_{s_2,s_2+t}^{x_2,h_2}(v_2)
     \Big\|_{L^{p}(\P(\cdot|\mathbb{F}_{s_2});H)}
     =
    \Big\|
     Y_t
     \Big\|_{L^{p}(\P(\cdot|D_{s_1,s_2}^{x_1,h_1}(v_1));H)}
     \\&
     \leq
     \left(\|Y_0\|_{L^{p(1+\theta)}(\P(\cdot|\mathbb{F}_{s_2});H)}^2
     +\int_0^t
     \Big\|
     \Big(
     \tfrac{(p(1+\theta)-1)(1+\delta)}{\delta}\|\zeta_s^{\sigma}\|_{\HS(U,H)}^2
     +\tfrac{1}{2\delta}\|\zeta_s^{\mu}\|_{H}^2\Big)^{\frac{1}{2}}\Big\|_{L^{p(1+\theta)}(\P(\cdot|\mathbb{F}_{s_2});\R)}^2
     \,ds
     \right)^{\!\frac{1}{2}}
  \\&
  \cdot
   \Big\|\exp\Big(\smallint_{s_2}^{t+s_2}
   \Big(
  \constFun(r)
  +\tfrac{V_0(X_{s_2,r}^{x_2+v_2h_2})+V_0(X_{s_2,r}^{x_2})}{2q_0Te^{\alpha_0(r)}}
  +\tfrac{\bar{V}(r,X_{s_2,r}^{x_2+v_2h_2})+\bar{V}(r,X_{s_2,r}^{x_2})}{2q_1e^{\alpha_1(r)}}
  \Big)
   \,dr\Big)\Big\|_{L^{p(1+\theta)/\theta}(\P;\R)}
     \\&
\leq
     \Big\|
     D_{s_1,s_2}^{x_1,h_1}(v_1)-v_2
     \Big\|_H
     +\left(\int_0^t
     \Big(\tfrac{(p(1+\theta)-1)(1+\delta)}{\delta}\|\zeta_s^{\sigma}\|_{L^{p(1+\theta)}(\P(\cdot|\mathbb{F}_{s_2});\HS(U,H))}^2
     +\tfrac{1}{2\delta}\|\zeta_s^{\mu}\|_{L^{p(1+\theta)}(\P(\cdot|\mathbb{F}_{s_2});H)}^2
     \Big)\,ds
     \right)^{\!\frac{1}{2}}
  \\&\qquad
  \cdot
   \exp\Big(\smallint_{s_2}^{t+s_2}
   \Big(
  \constFun(r)
        +
        \tfrac{
          \beta_{ 0 }
        }
        { q_0 e^{ \alpha_{ 0 } r } }
        + 
        \tfrac{ \beta_{ 1 } }
        { q_1 e^{ \alpha_{ 1 } r } }
    \Big)
        \,
        dr
        +\sum_{i=0}^1 \tfrac{V_i(x_2+v_2h_2)+V_i(x_2)}{2q_i e^{\alpha_is_2}}
   \Big).
  \end{split}     \end{equation}
  This, the fact that $p\geq 2$, and the triangle inequality imply that
  \begin{equation}  \begin{split}\label{eq:deltaD}
    &\Big\|
     D_{s_1,t_2}^{x_1,h_1}(v_1)-D_{s_2,t_2}^{x_2,h_2}(v_2)
     \Big\|_{L^{p}(\P;H)}
     \\&
    =\Big\|\Big\|
     D_{s_1,s_2+t_2-s_2}^{x_1,h_1}(v_1)-D_{s_2,s_2+t_2-s_2}^{x_2,h_2}(v_2)
     \Big\|_{L^{p}(\P(\cdot|\mathbb{F}_{s_2});H)}
     \Big\|_{L^{p}(\P;\R)}
     \\&
     \leq
     \bigg(
    \Big\|
     D_{s_1,s_2}^{x_1,h_1}(v_1)-v_2
     \Big\|_{L^{p}(\P;H)}
     \\&\qquad\quad
     +
     \sup_{r\in[0,T-s_2]}\max\big\{
    \big\|\zeta^{\mu}_r\big\|_{L^{p(1+\theta)}(\P;H)},
    \big\|\zeta^{\sigma}_r\big\|_{L^{p(1+\theta)}(\P;\HS(U,H))}
    \big\}
    \sqrt{\tfrac{(t_2-s_2) p(1+\theta)(1+\delta)}{\delta}}
    \bigg)
  \\&
  \cdot
   \exp\Big(\smallint_{s_2}^{t_2}
   \Big(
  \constFun(r)
        +
        \tfrac{
          \beta_{ 0 }
        }{
          q_{0} e^{ \alpha_{ 0 } r }
        }
        + 
        \tfrac{ \beta_{ 1 } }{ 	q_{ 1 } e^{ \alpha_{ 1 } r } }
     \Big)
        \,
        dr
        +\sum_{i=0}^1 \tfrac{V_i(x_2+v_2h_2)+V_i(x_2)}{2q_ie^{\alpha_is_2}}
   \Big).
  \end{split}     \end{equation}
  Moreover, \eqref{eq:zeta.mue}, the triangle inequality, and~\eqref{eq:Dmu.localLip}
  yield for all $r\in[0,T-s_2]$ that
  \begin{equation}  \begin{split}
   &\Big\|\zeta_r^{\mu}   \Big\|_{L^{p(1+\theta)}(\P;H)}
   \leq\Big\|
     \Big\|
     \smallint_0^1
     \Big(
     \mu' \Big(  \lambda X_{s_1,s_2+r}^{x_1 + v_1 h_1} +(1-\lambda) X_{s_1,s_2+r}^{x_1} \Big)
     \\&\qquad\qquad\qquad
     -
     \mu' \Big(  \lambda X_{s_2,s_2+r}^{x_2 + v_2 h_2} +(1-\lambda) X_{s_2,s_2+r}^{x_2} \Big)
     \Big)
     \,d\lambda
     \Big\|_{L(H,H)}
     \Big\|D_{s_1,s_2+r}^{x_1,h_1}(v_1)\Big\|_H
   \Big\|_{L^{p(1+\theta)}(\P;\R)}
   \\&
   \leq\Big\|
     c\smallint_0^1
     \Big(
     \lambda\big\|X_{s_1,s_2+r}^{x_1 + v_1 h_1}-X_{s_2,s_2+r}^{x_2 + v_2 h_2}\big\|_H
     +(1-\lambda)\big\|X_{s_1,s_2+r}^{x_1}- X_{s_2,s_2+r}^{x_2}\big\|_H
     \Big)
     \,d\lambda\,
     \big\|D_{s_1,s_2+r}^{x_1,h_1}(v_1)\big\|_H
     \\&\
     \qquad
     \Big(4+V_0(X_{s_1,s_2+r}^{x_1+v_1h_1})
     +V_0(X_{s_1,s_2+r}^{x_1})
     +V_0(X_{s_2,s_2+r}^{x_2+v_2h_2})
     +V_0(X_{s_2,s_2+r}^{x_2})
     \Big)^{\gamma}
   \Big\|_{L^{p(1+\theta)}(\P;H)}.
  \end{split}     \end{equation}
  This, H\"older's inequality
  (applied with
  $\tfrac{1}{p(1+\theta)}=2\tfrac{1}{2p(1+\theta)(1+\delta)}+\tfrac{\delta}{p(1+\theta)(1+\delta)}$),
  and
  the triangle inequality
  show for all $r\in[0,T-s_2]$ that
  \begin{equation}  \begin{split}\label{eq:estimate.zetamue.1}
   &\Big\|\zeta_r^{\mu}   \Big\|_{L^{p(1+\theta)}(\P;H)}
   \leq
     \Big\|\Big\|D_{s_1,s_2+r}^{x_1,h_1}(v_1)\Big\|_H
     \Big\|_{L^{2p(1+\theta)(1+\delta)}(\P;\R)}
     \\&\qquad\cdot
   \Big\|
     \smallint_0^1
     c\big\|\lambda\big(X_{s_1,s_2+r}^{x_1 + v_1 h_1}-X_{s_2,s_2+r}^{x_2 + v_2 h_2}\big)
     +(1-\lambda)\big(X_{s_1,s_2+r}^{x_1}- X_{s_2,s_2+r}^{x_2}\big)\big\|_H
     \,d\lambda
     \Big\|_{L^{2p(1+\theta)(1+\delta)}(\P;\R)}
     \\&
     \qquad\cdot
     \Big\|
     \Big(4
     +V_0(X_{s_1,s_2+r}^{x_1+v_1h_1})
     +V_0(X_{s_1,s_2+r}^{x_1})
     +V_0(X_{s_2,s_2+r}^{x_2+v_2h_2})
     +V_0(X_{s_2,s_2+r}^{x_2})
     \Big)^{\gamma}
   \Big\|_{L^{\frac{p(1+\theta)(1+\delta)}{\delta }}(\P;\R)}
   \\&
   \leq
     \Big\|D_{s_1,s_2+r}^{x_1,h_1}(v_1)
     \Big\|_{L^{2p(1+\theta)(1+\delta)}(\P;H)}
     \cdot
     c\max_{\iota\in\{0,1\}}
   \Big\|
     X_{s_1,s_2+r}^{x_1 + \iota v_1 h_1}-X_{s_2,s_2+r}^{x_2 + \iota v_2 h_2}
    \Big\|_{L^{2p(1+\theta)(1+\delta)}(\P;H)}
  \\&
     \qquad\cdot
     4^{\gamma}\max_{i\in\{1,2\}}\max_{z\in\{x_i,x_i+v_ih_i\}}
     \Big\|1+
     V_0(X_{s_i,s_2+r}^{z})
   \Big\|_{L^{\gamma p(1+\theta)(1+\delta)/\delta}(\P;\R)}^{\gamma}.
  \end{split}     \end{equation}
  Equation \eqref{eq:def:Zvy},
  Lemma~\ref{l:C1.implies.C0}, and
  Lemma~\ref{lem:Hoelder}
  (applied for all $r\in[s_1,T]$
  with 
  $T\defeq t_2$,
  $p\defeq 2p(1+\theta)(1+\delta)$,
  $s_1\defeq s_1$, $s_2\defeq s_1$, $t_1\defeq s_2+r$, $t_2\defeq s_2+r$,
  $x_1\defeq x_1+v_1h_1$, $x_2\defeq x_1$
  in the notation of
  Lemma~\ref{lem:Hoelder})
  yield for all $r\in[0,t_2-s_2]$ that
  \begin{equation}  \begin{split}\label{eq:estimate.D}
     &\Big\|D_{s_1,s_2+r}^{x_1,h_1}(v_1)
     \Big\|_{L^{2p(1+\theta)(1+\delta)}(\P;H)}
     =
     \tfrac{1}{|h_1|}
     \Big\|X_{s_1,s_2+r}^{x_1+v_1h_1}-X_{s_1,s_2+r}^{x_1}
     \Big\|_{L^{2p(1+\theta)(1+\delta)}(\P;H)}
     \\&
     \leq
\|v_1\|_H
  \exp\!\left(
        \int_{s_1}^{t_2}  
        \Big(
        \constFun(t)
        +
        \sum_{i=0}^1
        \tfrac{
          \beta_{ i } 
        }{
           q_i e^{ \alpha_{ i } t }
        }
        \Big)
        \,
        dt
        +
        \sum_{ i = 0 }^1 
        \tfrac{V_i(x_1+v_1h_1)+V_i(x_1)}{2q_i e^{\alpha_i s_1}}
      \right).
  \end{split}     \end{equation}
  Moreover,
  \eqref{eq:estimate.zetamue.1},
  \eqref{eq:estimate.D},
  \eqref{eq:ass.41},
  Lemma~\ref{lem:Hoelder}
  (applied for all $r\in[0,t_2-s_2]$, $\iota\in\{0,1\}$
  with 
  $p\defeq 2p(1+\theta)(1+\delta)$,
  $s_1\defeq s_1$, $s_2\defeq s_2$, $t_1\defeq s_2+r$, $t_2\defeq s_2+r$,
  $x_1\defeq x_1+\iota v_1h_1$,
  $x_2\defeq x_2+\iota v_2h_2$
  in the notation of
  Lemma~\ref{lem:Hoelder}),
  and
  Lemma~\ref{l:exp_mom.moments}
  yield for all $r\in[0,t_2-s_2]$ that
  \begin{equation}  \begin{split}
   &\Big\|\zeta_r^{\mu}   \Big\|_{L^{p(1+\theta)}(\P;H)}
   \\&
  \leq
\|v_1\|_H
  \exp\!\left(
        \int_{s_1}^{t_2}  
        \Big(
        \constFun(t)
        +
        \sum_{i=0}^1
        \tfrac{
          \beta_{ i } 
        }{
           q_i e^{ \alpha_{ i } t }
        }
        \Big)
        \,
        dt
        +
        \sum_{ i = 0 }^1 
        \tfrac{V_i(x_1+v_1h_1)+V_i(x_1)}{2q_i e^{\alpha_i s_1}}
      \right)
\\&
\cdot c\max_{\iota\in\{0,1\}}\Bigg[
  \exp\!\left(
        \int_{s_1}^{t_2}  
        \Big(
        \constFun(t)
        +
        \sum_{i=0}^1
        \tfrac{
          \beta_{ i } 
        }{
           q_i e^{ \alpha_{ i } t }
        }
        \Big)
        \,
        dt
        +
        \sum_{ i = 0 }^1 
        \tfrac{V_i(x_1+\iota v_1h_1)+V_i(x_2+\iota v_2h_2)}{2 q_i e^{\alpha_is_1}}
      \right)
   \\&
\Bigg(
\|x_1+\iota v_1h_1-x_2-\iota v_2h_2\|_H
  +\sqrt{|s_1-s_2|}
 c e^{\alpha_0 \gamma|s_2-s_1|}
\\&\qquad\cdot \Big| 2p(1+\theta)^2(1+\delta)\gamma+V_0(x_2+\iota v_2h_2)
 +\int_{s_1}^{t_2}\tfrac{\beta_0}{e^{\alpha_0 u}}\,du \Big|^{\gamma}
\Big(\sqrt{t_2}+2p(1+\theta)^2(1+\delta)\Big)
      \Bigg)
  \Bigg]
  \\&\cdot
  4^{\gamma}
    e^{\alpha_0 \gamma t_2}\left(\tfrac{\gamma p(1+\theta)(1+\delta)}{\delta}
    +\smallint_{s_1}^{t_2} \tfrac{\beta_0}{e^{\alpha_0 t}}\,dt
    +\max_{z\in\{x_1,x_2,x_1+v_1h_1,x_2+v_2h_2\}}
    V_0(z)
  \right)^{\gamma}
  \end{split}     \end{equation}
  and
  \begin{equation}  \begin{split}\label{eq:estimate.mue}
   &\Big\|\zeta_r^{\mu}   \Big\|_{L^{p(1+\theta)}(\P;H)}
  \leq
  \left(\|x_1-x_2\|_H+\|v_1h_1-v_2h_2\|_H+\sqrt{|s_1-s_2|}\right)
  \\&
  \cdot
\|v_1\|_H\,
  c\max_{\iota\in\{0,1\},j\in\{1,2\}}
  \exp\!\left(
        2\int_{s_1}^{t_2}  
        \Big(
        \constFun(t)
        +
        \sum_{i=0}^1
        \tfrac{
          \beta_{ i } 
        }{
          q_{i} e^{ \alpha_{ i } t }
        }
        \Big)
        \,
        dt
        +
        2\sum_{ i = 0 }^1 
        \tfrac{V_i(x_j+\iota v_jh_j)}{ q_{i} e^{\alpha_i s_1}} 
      \right)
  \\&\cdot
  4^{\gamma}
    e^{2\alpha_0 \gamma t_2}\left(\tfrac{2(\gamma+1) p(1+\theta)^2(1+\delta)}{\min\{\delta,1\}}+\sqrt{t_2}
    +\smallint_{s_1}^{t_2} \tfrac{\beta_0}{e^{\alpha_0 t}}\,dt
    +\max_{\iota\in\{0,1\},j\in\{1,2\}}
    V_0(x_j+\iota v_jh_j)
  \right)^{2\gamma+1}(1+c).
  \end{split}     \end{equation}
  An analogous argumentation shows for all $r\in[0,t_2-s_2]$ that
  \begin{equation}  \begin{split}\label{eq:estimate.sigma}
   &\Big\|\zeta_r^{\sigma}   \Big\|_{L^{p(1+\theta)}(\P;\HS(U,H))}
  \leq
  \left(\|x_1-x_2\|_H+\|v_1h_1-v_2h_2\|_H+\sqrt{|s_1-s_2|}\right)
  \\&
  \cdot
\|v_1\|_H\,
  c\max_{\iota\in\{0,1\},j\in\{1,2\}}
  \exp\!\left(
        2\int_{{s_1}}^{t_2}  
        \Big(
        \constFun(t)
        +
        \sum_{i=0}^1
        \tfrac{
          \beta_{ i } 
        }{
          q_{i} e^{ \alpha_{ i } t }
        }
        \Big)
        \,
        dt
        +
        2\sum_{ i = 0 }^1 
        \tfrac{V_i(x_j+\iota v_jh_j)}{ q_{i} } 
      \right)
  \\&\cdot
  4^{\gamma}
    e^{2\alpha_0 \gamma t_2}\left(\tfrac{2(\gamma+1) p(1+\theta)^2(1+\delta)}{\min\{\delta,1\}}+\sqrt{t_2}
    +\smallint_{s_1}^{t_2} \tfrac{\beta_0}{e^{\alpha_0 t}}\,dt
    +\max_{\iota\in\{0,1\},j\in\{1,2\}}
    V_0(x_j+\iota v_jh_j)
  \right)^{2\gamma+1}(1+c).
 \end{split}     \end{equation}
  Next, we derive a temporal regularity estimate.
  Lemma~\ref{lem:Zv:eq},
  the triangle inequality,
  the Burkholder-Davis-Gundy type inequality in \cite[Lemma 7.7]{dz92}, 
  H\"older's inequality (applied with $\tfrac{1}{p}=\tfrac{1}{2p}+\tfrac{1}{2p}$),
  and~\eqref{eq:growth.Dmue}
  prove for all $u_1\in[s_1,T]$, $u_2\in[u_1,T]$ that
  \begin{equation}  \begin{split}
    &\left\|D_{s_1,u_2}^{x_1,h_1}(v_1)-D_{s_1,u_1}^{x_1,h_1}(v_1)
    \right\|_{L^{p}(\P;H)}
    \\&
    \leq\int_{u_1}^{u_2}\left\|
    \left\|
      \int_0^1 \mu' \Big( \lambda X_{s_1,r}^{x_1 + v_1 h_1} 
      +(1-\lambda)X_{s_1,r}^{x_1} \Big)\,d\lambda\right\|_{L(H,H)}
      \|D_{s_1,r}^{x_1,h_1}(v_1) \|_H
      \right\|_{L^{p}(\P;\R)}
      \,dr
    \\&\quad
    +\left(\tfrac{p(p-1)}{2}\int_{u_1}^{u_2}\left\|
    \left\|
      \int_0^1 \sigma' \Big( \lambda X_{s_1,r}^{x_1 + v_1 h_1} 
      +(1-\lambda)X_{s_1,r}^{x_1} \Big)\,d\lambda\right\|_{L(H,\HS(U,H))}
      \|D_{s_1,r}^{x_1,h_1}(v_1) \|_H
      \right\|_{L^{p}(\P;\R)}^2\,dr
    \right)^{\frac{1}{2}}
    \\&
    \leq\int_{u_1}^{u_2}
    \left\|
    c\left(2+V_0(X_{s_1,r}^{x_1 + v_1 h_1})+V_0(X_{s_1,r}^{x_1})\right)^{\gamma}
    \right\|_{L^{2p}(\P;\R)}
    \left\|
      D_{s_1,r}^{x_1,h_1}(v_1)
    \right\|_{L^{2p}(\P;H)}
      \,dr
    \\&\quad
    +\left(\tfrac{p(p-1)}{2}\int_{u_1}^{u_2}
    \left\|
    c\left(2+V_0(X_{s_1,r}^{x_1 + v_1 h_1})+V_0(X_{s_1,r}^{x_1})\right)^{\gamma}
    \right\|_{L^{2p}(\P;\R)}^2
    \left\|
      D_{s_1,r}^{x_1,h_1}(v_1)
    \right\|_{L^{2p}(\P;H)}^2
      \,dr
    \right)^{\frac{1}{2}}.
  \end{split}     \end{equation}
  This, the triangle inequality,
  Lemma~\ref{l:exp_mom.moments}
  (applied for all $r\in[s_1,T]$, $\iota\in\{0,1\}$ with $T\defeq r$,
  $s\defeq s_1$,
  $\mathbb{F}\defeq (\mathbb{F}_{t})_{t\in[s_1,r]}$,
  $W\defeq (W_t)_{t\in[s_1,r]}$,
  $X\defeq (X_{s_1,t}^{x_1+\iota v_1y})_{t\in[s_1,r]}$,
  $\alpha\defeq \alpha_0$,
  $\beta\defeq \beta_0$,
  $V\defeq V_0$,
  $t\defeq r$,
  $p\defeq 2p\gamma$
  in the notation of 
  Lemma~\ref{l:exp_mom.moments}),
  and
  \eqref{eq:estimate.D}
  imply for all $u_1\in[s_1,T]$, $u_2\in[u_1,T]$ that
  \begin{equation}  \begin{split}\label{eq:temporal.estimate}
    &\left\|D_{s_1,u_2}^{x_1,h_1}(v_1)-D_{s_1,u_1}^{x_1,h_1}(v_1)
    \right\|_{L^{p}(\P;H)}
    \\&\leq
    \left(u_2-u_1+p\sqrt{u_2-u_1}\right)
    c
    \\&\quad
    \cdot\sup_{r\in[u_1,u_2]}
    \bigg[
    \bigg(
    \sum_{\iota\in\{0,1\}}
    \left\|
    1+V_0(X_{s_1,r}^{x_1 + \iota v_1 h_1})
    \right\|_{L^{2p\gamma}(\P;\R)}\bigg)^{\gamma}
    \left\|
      D_{s_1,r}^{x_1,h_1}(v_1)
    \right\|_{L^{2p}(\P;H)}
    \bigg]
    \\&\leq
    \sqrt{u_2-u_1}
    \left(\sqrt{T}+p\right)
    c
    e^{\alpha_0\gamma u_2}
    \bigg(
    4p\gamma+\int_{0}^{u_2}\tfrac{2\beta_0}{e^{\alpha_0 r}}\,dr
    +
    \sum_{\iota\in\{0,1\}}
    V_0(x_1+\iota v_1h_1)
    \bigg)^{\gamma}
   \\&\quad
   \cdot
\|v_1\|_H
  \exp\!\left(
        \int_{0}^{T}  
        \Big(
        \constFun(r)
        +
        \sum_{i=0}^1
        \tfrac{
          \beta_{ i } 
        }{
          q_{i} e^{ \alpha_{ i } r }
        }
        \Big)
        \,
        dr
        +
        \sum_{ i = 0 }^1 
        \tfrac{V_i(x_1+v_1h_1)+V_i(x_1)}{2 q_{i}e^{\alpha_i s_1}  } 
      \right)
  \end{split}     \end{equation}
  The triangle inequality and~\eqref{eq:deltaD} yield that
  \begin{equation}  \begin{split}
    &\Big\|
     D_{s_1,t_1}^{x_1,h_1}(v_1)-D_{s_2,t_2}^{x_2,h_2}(v_2)
     \Big\|_{L^{p}(\P;H)}
     \\&
     \leq
   \Big\|
     D_{s_1,t_1}^{x_1,h_1}(v_1)-D_{s_1,t_2}^{x_1,h_1}(v_1)
     \Big\|_{L^{p}(\P;H)}
    +
   \Big\|
     D_{s_1,t_2}^{x_1,h_1}(v_1)-D_{s_2,t_2}^{x_2,h_2}(v_2)
     \Big\|_{L^{p}(\P;H)}
     \\&
     \leq
   \Big\|
     D_{s_1,t_1}^{x_1,h_1}(v_1)-D_{s_1,t_2}^{x_1,h_1}(v_1)
     \Big\|_{L^{p}(\P;H)}
   \\&+
     \bigg(
    \Big\|
     D_{s_1,s_2}^{x_1,h_1}(v_1)-v_1+v_1-v_2
     \Big\|_{L^{p}(\P;H)}
     \\&\qquad\quad
     +
     \sup_{r\in[0,t_2-s_2]}\max\big\{
    \big\|\zeta^{\mu}_r\big\|_{L^{p(1+\theta)}(\P;H)},
    \big\|\zeta^{\sigma}_r\big\|_{L^{p(1+\theta)}(\P;\HS(U,H))}
    \big\}
    \sqrt{\tfrac{(t_2-s_2)p(1+\theta)(1+\delta)}{\delta}}
    \bigg)
  \\&\quad
  \cdot
   \exp\Big(\smallint_{s_2}^{t_2}
   \Big(
  \constFun(r)
        +
        \tfrac{
          \beta_{ 0 }
        }{
          q_{0} e^{ \alpha_{ 0 } r }
        }
        + 
        \tfrac{ \beta_{ 1 } }{ 	q_{ 1 } e^{ \alpha_{ 1 } r } }
   \Big)
        \,
        dr
        +\sum_{i=0}^1 \tfrac{V_i(x_2+v_2h_2)+V_i(x_2)}{2q_ie^{\alpha_is_2}}
   \Big).
  \end{split}     \end{equation}
  This,~\eqref{eq:temporal.estimate},~\eqref{eq:estimate.mue},
  and
  \eqref{eq:estimate.sigma}
  ensure that
  \begin{equation}  \begin{split}
    &\Big\|
     D_{s_1,t_1}^{x_1,h_1}(v_1)-D_{s_2,t_2}^{x_2,h_2}(v_2)
     \Big\|_{L^{p}(\P;H)}
  \\&
     \leq
    \sqrt{|t_2-t_1|}
    \left(\sqrt{T}+p\right)
    c
    e^{\alpha_0\gamma T}
    \bigg(
    4p\gamma+\int_{0}^{T}\tfrac{2\beta_0}{e^{\alpha_0 r}}\,dr
    +
    \sum_{\iota\in\{0,1\}}
    V_0(x_1+\iota v_1h_1)
    \bigg)^{\gamma}
   \\&\quad
   \cdot
\|v_1\|_H
  \exp\!\left(
        \int_{0}^{T}  
        \Big(
        \constFun(r)
        +
        \sum_{i=0}^1
        \tfrac{
          \beta_{ i } 
        }{
          q_{i} e^{ \alpha_{ i } r }
        }
        \Big)
        \,
        dr
        +
        \sum_{ i = 0 }^1 
        \tfrac{V_i(x_1+v_1h_1)+V_i(x_1)}{2q_{i} e^{\alpha_i s_1} } 
      \right)
\\&+
     \Bigg[
     \|v_1-v_2\|_H
     +
    \sqrt{|s_2-s_1|}
    \left(\sqrt{T}+p\right)
    c
    e^{\alpha_0\gamma s_2}
    \Big(
    4p\gamma+\int_{0}^{s_2}\tfrac{2\beta_0}{e^{\alpha_0 r}}\,dr
    +
    \sum_{\iota\in\{0,1\}}
    V_0(x_1+\iota v_1h_1)
    \Big)^{\gamma}
 \\&\qquad\quad
   \cdot
\|v_1\|_H
  \exp\!\left(
        \int_{0}^{s_2}  
        \Big(
        \constFun(r)
        +
        \sum_{i=0}^1
        \tfrac{
          \beta_{ i } 
        }{
          q_{i} e^{ \alpha_{ i } r }
        }
        \Big)
        \,
        dr
        +
        \sum_{ i = 0 }^1 
        \tfrac{V_i(x_1+v_1h_1)+V_i(x_1)}{2 q_{i} e^{\alpha_i s_1} } 
      \right)
   \\&\qquad
     +
  \left(\|x_1-x_2\|_H+\|v_1h_1-v_2h_2\|_H+\sqrt{|s_1-s_2|}\right)
  \\&\quad\qquad
  \cdot
\|v_1\|_H\,
  c\max_{\iota\in\{0,1\},j\in\{1,2\}}
  \exp\!\left(
        2\int_{0}^{T}  
        \Big(
        \constFun(r)
        +
        \sum_{i=0}^1
        \tfrac{
          \beta_{ i } 
        }{
          q_{i} e^{ \alpha_{ i } r }
        }
        \Big)
        \,
        dr
        +
        2\sum_{ i = 0 }^1 
        \tfrac{V_i(x_j+\iota v_jh_j)}{ q_{i} e^{\alpha_i s_1} } 
      \right)
  \\&\quad\qquad\cdot
  4^{\gamma}
    e^{2\alpha_0 \gamma T}\left(\tfrac{2(1+\gamma)p(1+\theta)^2(1+\delta)}
    {\min\{\delta,1\}}
    +\sqrt{T}
    +\smallint_0^T \tfrac{\beta_0}{e^{\alpha_0 r}}\,dr
    +\max_{\iota\in\{0,1\},j\in\{1,2\}}
    V_0(x_j+\iota v_jh_j)
  \right)^{2\gamma+1}
  \\&\quad\qquad\cdot
  (1+c)
  \sqrt{(t_2-s_2) p(1+\theta)(1+\delta^{-1})}
    \Bigg]
  \\&\quad
  \cdot
   \exp\Big(\smallint_{s_2}^{t_2}
   \Big(
  \constFun(r)
        +
        \tfrac{
          \beta_{ 0 }
        }{
          q_{0} e^{ \alpha_{ 0 } r }
        }
        + 
        \tfrac{ \beta_{ 1 } }{ 	q_{ 1 } e^{ \alpha_{ 1 } r } }
   \Big)
        \,
        dr
        +\sum_{i=0}^1 \tfrac{V_i(x_2+v_2h_2)+V_i(x_2)}{2q_ie^{\alpha_is_2}}
   \Big).
  \end{split}     \end{equation}
  Consequently we obtain that
  \begin{equation}  \begin{split}
    &\Big\|
     D_{s_1,t_1}^{x_1,h_1}(v_1)-D_{s_2,t_2}^{x_2,h_2}(v_2)
     \Big\|_{L^{p}(\P;H)}
   \\&
   \leq
 \bigg[
  \Big(\Big(\|x_1-x_2\|_H+\|v_1h_1-v_2h_2\|_H\Big)\sqrt{t_2-s_2}
  +\sqrt{|s_1-s_2|}+
  \sqrt{|t_1-t_2|}\Big)
  \|v_1\|_H
  +\|v_1-v_2\|_H
  \bigg]
  \\&\quad
  \cdot
    e^{2\alpha_0 \gamma T}
    \left(\tfrac{4p(1+\theta)^2(1+\delta)(1+\gamma)}{\min\{\delta,1\}}
    +2\sqrt{T}
    +\smallint_0^T \tfrac{2\beta_0}{e^{\alpha_0 r}}\,dr
    +2\max_{\iota\in\{0,1\},j\in\{1,2\}}
    V_0(x_j+\iota v_jh_j)
  \right)^{2\gamma+2}
  \\&\quad
  \cdot
  (1+c)^2\max_{\iota\in\{0,1\},j\in\{1,2\}}
  \exp\!\left(
        3\int_{0}^{T}  
        \Big(
        \constFun(r)
        +
        \sum_{i=0}^1
        \tfrac{
          \beta_{ i } 
        }{
          q_{i} e^{ \alpha_{ i } r }
        }
        \Big)
        \,
        dr
        +
        3\sum_{ i = 0 }^1 
        \tfrac{V_i(x_j+\iota v_jh_j)}{ q_{i}e^{\alpha_i s_1} } 
      \right).
  \end{split}     \end{equation}
 This completes the proof of Lemma~\ref{l:local.Lip.C1}.
 \end{proof}

\subsection{Moment estimates for the second-order derivative process}
The following lemma, Lemma~\ref{lem:moments:ZZv}, provides
a moment estimate for second-order spatial derivatives of solutions of SDEs.
\begin{lemma}[Moment estimates for the second-order derivative process] \label{lem:moments:ZZv}
Assume Setting~\ref{sett:exists:C1},
let $D\subseteq\mathcal{O}$ be an open set,
let $s\in[0,T]$, $t\in[s,T]$,
let $Y\colon\Omega\to D$, $Z_1,Z_2\colon\Omega\to H$ be $\mathbb{F}_s$/$\mathcal{B}(H)$-measurable,
assume that $\sigma(\{X_{s,t}^x\colon x\in D\})$ and $\mathbb{F}_s$ are independent,
and
assume for all $\omega\in\Omega$ that $(D\ni x\mapsto X_{s,t}^x(\omega)\in H)\in C^2(D,H)$.
Then
\begin{align} \label{lem:moments:ZZ:eq2}
\begin{split}
& \left\| \tfrac{\partial^2}{\partial x^2}X_{s,t}^Y (Z_1,Z_2)\right\|_{L^{p}(\P;H)} 
	 \leq \sqrt{t-s} 
\Bigg\|
  \exp\!\left(
        3\int_{0}^{t}  
        \Big(
        \constFun(r)
        +
        \sum_{i=0}^1
        \tfrac{
          \beta_{ i } 
        }{
          q_{i} e^{ \alpha_{ i } r }
        }
        \,
        \Big)
        dr
        +
        3\sum_{ i = 0 }^1 
        \tfrac{V_i(Y)}{ q_{i} e^{\alpha_i s}} 
      \right)
  \\&\quad
  \cdot
\|Z_1\|_H\|Z_2\|_H
    e^{2\alpha_0 \gamma t}
    \left(\tfrac{4p(1+\theta)^2(1+\delta)(1+\gamma)}{\min\{\delta,1\}}
    +2\sqrt{t}
    +\smallint_0^t \tfrac{2\beta_0}{e^{\alpha_0 r}}\,dr
    +
    2
    V_0(Y)
  \right)^{2\gamma+2}
  (1+c)^2
 \Bigg\|_{L^{p}(\P;\R)}.
\end{split}
\end{align} 
\end{lemma}
\begin{proof}[Proof of Lemma~\ref{lem:moments:ZZv}]
 Lemma~\ref{l:local.Lip.C1}
 (applied in the case $t>0$
  for every $y\in D$, $h,\eps\in(0,\infty)$, $z_1,z_2\in H$ satisfying that $y+hz_2\in D$,
  $y+hz_2+\eps z_1, y+\eps z_1\in \mathcal{O}$
 with $T\defeq t$,
 $s_1\defeq s$, $s_2\defeq s$, $t_1\defeq t$, $t_2\defeq t$, $x_1\defeq y+hz_2$, $x_2\defeq y$, $v_1\defeq z_1$, $v_2\defeq z_1$, $h_1\defeq \eps$, $h_2\defeq \eps$
 in the notation of
 Lemma~\ref{l:local.Lip.C1})
 yields for all $y\in D$, $z_1,z_2\in H$, $h,\eps\in(0,\infty)$ with $y+hz_2\in D$  that
\begin{align}
\begin{split}
   &\E\Big[\Big\|
  \tfrac{D_{s,t}^{y+hz_2,\eps}(z_1)-D_{s,t}^{y,\eps}(z_1)}{h}
   \Big\|_H^{p}\Big]
  \\&
\leq
 \Bigg(
  \|z_2\|_H
    e^{2\alpha_0 \gamma t}
    \left(\tfrac{4p(1+\theta)^2(1+\delta)(1+\gamma)}{\min\{\delta,1\}}
    +2\sqrt{t}
    +\smallint_0^t \tfrac{2\beta_0}{e^{\alpha_0 r}}\,dr
    +
    2\max_{i,j\in\{0,1\}}
    V_0(y+ihz_2+j\eps z_1)
  \right)^{2\gamma+2}
  \\&\quad
  \cdot\sqrt{t-s}
\|z_1\|_H\,
  (1+c)^2
  \exp\!\left(
        3\int_{0}^{t}  
        \Big(
        \constFun(r)
        +
        \sum_{i=0}^1
        \tfrac{
          \beta_{ i } 
        }{
          q_{i} e^{ \alpha_{ i } r }
        }
        \Big)
        \,
        dr
        +
        3
        \sum_{ i = 0 }^1 
        \tfrac{
    \max_{l,j\in\{0,1\}}
        V_i(y+lhz_2+j\eps z_1)}{ q_{i}e^{\alpha_i s} } 
      \right)
    \Bigg)^{p}.
\end{split}
\end{align}
 This, independence of $\sigma(\{X_{s,t}^x\colon x\in D\})$ and $\mathbb{F}_s$,
 a disintegration formula (e.g.~\cite[Lemma 2.3]{HJK+18}),
 the fact that
 for all $\omega\in\Omega$ it holds that $(D\ni x\mapsto X_{s,t}^x(\omega)\in H)\in C^2(D,H)$,
 and
 Fatou's lemma (e.g.\ Lemma 3.10 in~\cite{HutzenthalerJentzen2015Memoires})
 yield that
\begin{align}
\begin{split}
  & \left\| \tfrac{\partial^2}{\partial x^2}X_{s,t}^Y (Z_1,Z_2)\right\|_{L^{p}(\P;H)}^{p}
  =\int_{D\times H\times H} \E\Big[\Big\|\tfrac{\partial^2}{\partial x^2}X_{s,t}^y (z_1,z_2)\Big\|_H^{p}\Big]\P\Big((Y,Z_1,Z_2)\in d(y,z_1,z_2)\Big)
  \\&
  =\int_{D\times H\times H} \E\Big[\Big\|
  \liminf_{\substack{\{r\in (0,\infty)\colon y+rz_2\in D\}\ni h\to 0\\
   (0,\infty)\ni \eps \to 0}
  }
  \tfrac{D_{s,t}^{y+hz_2,\eps}(z_1)-D_{s,t}^{y,\eps}(z_1)}{h}
  \Big\|_H^{p}\Big]\P\Big((Y,Z_1,Z_2)\in d(y,z_1,z_2)\Big)
  \\&
  \leq\int_{D\times H\times H}
  \liminf_{\substack{\{r\in (0,\infty)\colon y+rz_2\in D\}\ni h\to0\\
   (0,\infty)\ni \eps \to 0}
  }
   \E\Big[\Big\|
  \tfrac{D_{s,t}^{y+hz_2,\eps}(z_1)-D_{s,t}^{y,\eps}(z_1)}{h}
   \Big\|_H^{p}\Big]\P\Big((Y,Z_1,Z_2)\in d(y,z_1,z_2)\Big)
  \\&
\leq\int_{D\times H\times H}
    \Bigg(
  \|z_2\|_H
    e^{2\alpha_0 \gamma t}
    \left(\tfrac{4p(1+\theta)^2(1+\delta)(1+\gamma)}{\min\{\delta,1\}}
    +2\sqrt{t}
    +\smallint_0^t \tfrac{2\beta_0}{e^{\alpha_0 r}}\,dr
    +
    2
    V_0(y)
  \right)^{2\gamma+2}
\|z_1\|_H\,
  (1+c)^2
  \\&\quad
  \cdot
  \sqrt{t-s}\exp\!\left(
        3\int_{0}^{t}  
        \Big(
        \constFun(r)
        +
        \sum_{i=0}^1
        \tfrac{
          \beta_{ i } 
        }{
          q_{i} e^{ \alpha_{ i } r }
        }
        \Big)
        \,
        dr
        +
        3\sum_{ i = 0 }^1 
        \tfrac{V_i(y)}{ q_{i} e^{\alpha_i s}} 
      \right)
    \Bigg)^{p}
\P\Big((Y,Z_1,Z_2)\in d(y,z_1,z_2)\Big)
\\&=
\Bigg\|\sqrt{t-s}\|Z_1\|_H\|Z_2\|_H
    e^{2\alpha_0 \gamma t}
    \left(\tfrac{4p(1+\theta)^2(1+\delta)(1+\gamma)}{\min\{\delta,1\}}
    +2\sqrt{t}
    +\smallint_0^t \tfrac{2\beta_0}{e^{\alpha_0 r}}\,dr
    +
    2
    V_0(Y)
  \right)^{2\gamma+2}
  (1+c)^2
  \\&\quad
  \cdot
  \exp\!\left(
        3\int_{0}^{t}  
        \Big(
        \constFun(r)
        +
        \sum_{i=0}^1
        \tfrac{
          \beta_{ i } 
        }{
          q_{i} e^{ \alpha_{ i } r }
        }
        \Big)
        \,
        dr
        +
        3\sum_{ i = 0 }^1 
        \tfrac{V_i(Y)}{ q_{i}e^{\alpha_i s} } 
      \right)
 \Bigg\|_{L^{p}(\P;\R)}^{p}.
\end{split}
\end{align}
This implies \eqref{lem:moments:ZZ:eq2}
and finishes
the proof of Lemma~\ref{lem:moments:ZZv}.
\end{proof}

\subsection{Existence of a $C^1$-solution}
The following theorem establishes existence of continuously
differentiable
solutions of SDEs.

\begin{theorem}[Existence of a $C^1$-solution] \label{thm:C1}
Assume Setting~\ref{sett:exists:C1},
assume that $\dim(H)<\infty$,
assume that $O\subseteq \overline{\mathcal{O}}$,
and
assume that $p\in (2\dim(H)+6,\infty)$.
Then
there exists a measurable
function
$
  \mathcal{X} \colon \Delta_T \times {O} \times \Omega \to \overline{O}
$
such that
\begin{enumerate}[(i)]
  \item  
for all 
$
  x \in \mathcal{O}, 
$ 
$ 
	s\in [0,T]
$ 
it holds a.s.~that 
$
  (\mathcal{X}^x_{s, t })_{t \in [s,T]} 
  = 
  (X^x_{s,t})_{t \in [s,T]}
$, 
and
\item
for every $ \omega \in \Omega $
it holds
that
$
  \mathcal{X}(\omega) \in C^{0,1}( \Delta_T \times {O}, \overline{O})
$.
\end{enumerate}
\end{theorem}
\begin{proof}[Proof of Theorem~\ref{thm:C1}]
Without loss of generality we additionally assume throughout this proof
that $\dim(H)\geq 1$ and that $O\neq\emptyset$.
Throughout this proof
let $o\in O$, $d\in\N$, $\mathbb{H}\subseteq H$ satisfy that $d=\dim(H)$
and that $\mathbb{H}$ is an orthonormal basis of $H$,
let $\mathcal{O}^{\R}$, $O^{\R}$
be the sets which satisfy that
$O^{\R}=\cap_{v\in\mathbb{H}}
\{(x,h)\in
O\times\R\colon x+vh\in O\}$
and
$\mathcal{O}^{\R}=
\cap_{v\in\mathbb{H}}\{(x,h)\in
\mathcal{O}\times(\R\setminus\{0\})\colon x+vh\in \mathcal{O}\}$
and
let $K_n\subseteq \Delta_T\times H\times \R$,
$n\in \N$,
be the sets which satisfy for all $n\in \N$ that
$K_n=\{(s,t,x,h)\in \Delta_T\times \mathcal{O}^{\R} \colon s^2+t^2+\|x\|_H^2+h^2 \leq n^2 \}$.
 Lemma~\ref{l:C1.implies.C0},
 the fact that $p>2d+6\geq 2d+4$,
 and
 Theorem~\ref{thm:exists:C0}
 (applied with $p\defeq 2p(1+\theta)(1+\delta)$ in the notation of 
 Theorem~\ref{thm:exists:C0})
 yield that there exists a
  measurable
  function $\tilde{\mathcal{X}}\colon\Delta_T\times\overline{\mathcal{O}}\times\Omega\to\overline{\mathcal{O}}$
  such that
for all 
$
  x \in \mathcal{O}, 
$ 
$ 
	s\in [0,T]
$ 
it holds a.s.~that 
$
  (\tilde{\mathcal{X}}^x_{s, t })_{t \in [s,T]} 
  = 
  (X^x_{s,t})_{t \in [s,T]}
$, 
and such that
for every $ \omega \in \Omega $
it holds
that
$
  \tilde{\mathcal{X}}(\omega) \in C^0( \Delta_T \times \overline{\mathcal{O}}, \overline{\mathcal{O}})
$.
Next,
Lemma~\ref{l:local.Lip.C1}, the fact that $\int_0^T \constFun(r) \, dr <\infty$,
and
boundedness of the functions $V_0,V_1$ on each of the bounded 
subsets $\{x\in \mathcal{O}\colon \exists s,t,h\in\R \text{ s.t.\ }(s,t,x,h)\in K_n\}\subseteq\mathcal{O}$,
$n\in\N$,
demonstrate for all $n\in\N$, $v\in\mathbb{H}$ that
\begin{equation}\label{eq:C0.ass.Hoelder2} 
  \sup\bigg(\bigg\{
  \tfrac{ 
  \big(\E\big[  \big\|D_{s_1,t_1}^{x_1,h_1}(v) - D_{s_2,t_2}^{x_2,h_2}(v)   \big\|^{p}_{H}\big]\big)^{\frac{1}{p}}
  }
  {\left( |s_1-s_2|^2+|t_1-t_2|^2+\|x_1-x_2\|_H^{2}+|h_1-h_2|^2
  \right)^{\frac{1}{4}}}
  \colon
  \substack{ (s_1,t_1,x_1,h_1), (s_2,t_2,x_2,h_2) \in K_n \colon
  \\(s_1,t_1,x_1,h_1)\neq (s_2,t_2,x_2,h_2)}
  \bigg\}\cup\{0\}\bigg)
  <\infty.
\end{equation}
In particular this implies for all $n\in\N$, $v\in\mathbb{H}$ that
\begin{equation}  \begin{split}
  &\sup
  \Big(\Big\{
  \Big(\E\Big[  \|D_{s,t}^{x,h}(v)   \|^{p}_{H}\Big]\Big)^{\frac{1}{p}}
  \colon
  (s,t,x,h) \in  K_n\Big\}\cup\{0\}\Big)
  \\&
  \leq
  \sup
  \bigg(\bigg\{
  \tfrac{ 
  \big(\E\big[  \|D_{s,t}^{x,h}(v) - D_{s,s}^{x,h}(v)   \|^{p}_{H}\big]
  \big)^{\frac{1}{p}}
  }
  {\left( |s-s|^2+|t-s|^2+\|x-x\|_H^{2}+|h-h|^2  \right)^{\frac{1}{4}}}
  \sqrt{T}+\|v\|_H
  \colon
  (s,t,x,h) \in  K_n,t\neq s\bigg\}\cup\{0\}\bigg)
  <\infty.
\end{split}     \end{equation}
This, \eqref{eq:C0.ass.Hoelder2},
Proposition~\ref{p:KolChen}
(applied for every $v\in\mathbb{H}$ with 
$H \defeq  \R\times\R\times H\times \R$,
$D\defeq \Delta_T\times \mathcal{O}^{\R}$,
$E\defeq H$,
$F\defeq H$,
$p\defeq p$,
$\alpha\defeq \nicefrac12$,
$X\defeq \left( \Delta_T\times \mathcal{O}^{\R} \ni (s,t,x,h)
\mapsto D_{s,t}^{x,h}(v) \in H \right)$
in the notation of Proposition~\ref{p:KolChen}),
and path continuity of $D_{s,\cdot}^{x,h}(v)$, 
$(s,x,h,v)\in[0,T]\times \mathcal{O}^{\R}\times\mathbb{H}$,
establish for every $v\in\mathbb{H}$ the existence of a
measurable function
$
  \mathcal{D}^v \colon \overline{\Delta_T\times \mathcal{O}^{\R}}
  \times \Omega \to H
$
which satisfies that for all $\omega \in \Omega$ it holds that 
$\mathcal{D}^v(\omega) \in C(\overline{\Delta_T\times \mathcal{O}^{\R}} , H)$
and which satisfies
that
for all $(s,t,x,h)\in\Delta_T\times \mathcal{O}^{\R}$
it holds a.s.~that
$(\mathcal{D}_{s,t}^v(x,h))_{t\in[s,T]}=(D_{s,t}^{x,h}(v))_{t\in[s,T]}$.
Note that $O\subseteq\overline{\mathcal{O}}$ and convexity of $\mathcal{O}$ imply that
$O^{\R}\subseteq \overline{\mathcal{O}^{\R}}$.
Let $\mathcal{D}\colon \Delta_T\times
O^{\R}\times H\times\Omega
\to H$ be the function which satisfies for all
$(s,t,x,h)\in\Delta_T\times O^{\R}$, $v\in H$
that
$\mathcal{D}_{s,t}(x,h,v)=\sum_{e\in\mathbb{H}}\langle v,e\rangle_H
\mathcal{D}_{s,t}^e(x,h)$.
Next, we observe that for all $(s,x,h,v)\in[0,T]\times \mathcal{O}^{\R}\times \mathbb{H}$
 it holds a.s.\ for all $t\in[s,T]$ that
\begin{equation}  \begin{split}
  \mathcal{D}_{s,t}(x,h,v)=
\mathcal{D}_{s,t}^v(x,h)=
D_{s,t}^{x,h}(v)=
\tfrac{X_{s,t}^{x+hv}-X_{s,t}^x}{h}
=
\tfrac{\tilde{\mathcal{X}}_{s,t}^{x+hv}-\tilde{\mathcal{X}}_{s,t}^x}{h}.
\end{split}     \end{equation}
This, continuity of the random fields $\tilde{\mathcal{X}}$, $\mathcal{D}$,
and Lemma~\ref{lem:gradient} 
(applied with $U\defeq H$,
$T\defeq \Delta_T$,
$\mathbb{T}\defeq \Delta_T\cap\Q^2$,
$\mathcal{X}\defeq (\Delta_T\times O\times\Omega\ni(s,t,x,\omega)\mapsto \tilde{\mathcal{X}}_{s,t}^x(\omega)\in H)$,
$\mathcal{Z}\defeq \mathcal{D}$
in the notation of Lemma~\ref{lem:gradient})
prove that there exists $\Omega_0\in\mathcal{F}$
such that $\P(\Omega_0)=1$ and such that for all $\omega\in\Omega_0$,
$(s,t)\in\Delta_T$ it holds that
the mapping $O\ni x\mapsto \tilde{\mathcal{X}}_{s,t}^{x}(\omega)\in H$ is continuously differentiable and it holds for all $x\in O$, $v\in H$
that
\begin{equation}  \begin{split}
\tfrac{\partial}{\partial x}\mathcal{\tilde{X}}_{s,t}^x(\omega)v
=\sum_{e\in\mathbb{H}}\langle v,e\rangle_H \mathcal{D}_{s,t}(x,0,e)
=\mathcal{D}_{s,t}(x,0,v).
\end{split}     \end{equation}
This and continuity of $\mathcal{D}$ prove that for all $\omega\in\Omega_0$
it holds that $\tilde{\mathcal{X}}(\omega)|_{\Delta_T\times O}\in C^{0,1}(\Delta_T\times O,\overline{O})$.
Let $\mathcal{X}\colon\Delta_T\times{O}\times\Omega
\to \overline{O}$ be the function which satisfies for all $(s,t,x,\omega)\in\Delta_T\times{O}\times\Omega$ that
$\mathcal{X}_{s,t}^x(\omega)=\mathbbm{1}_{\Omega_0}(\omega)\tilde{\mathcal{X}}_{s,t}^x(\omega)+o\1_{\Omega\setminus\Omega_0}(\omega)$.
Then it holds
that $\mathcal{X}$ is
measurable,
that
for all 
$
  x \in \mathcal{O} 
$,
$ 
	s\in [0,T]
$ 
it holds a.s.~that 
$
  (\mathcal{X}^x_{s, t })_{t \in [s,T]} 
  = 
  (\tilde{\mathcal{X}}^x_{s, t })_{t \in [s,T]} 
  =
  (X^x_{s,t})_{t \in [s,T]}
$, 
and
that
for every $ \omega \in \Omega $
it holds
that
$
  \mathcal{X}(\omega) \in C^{0,1}( \Delta_T \times {O}, {\overline{O}})
$.
 This proves item (i) and item (ii)
 and finishes the proof of Theorem~\ref{thm:C1}.
\end{proof}


\section{Existence of a $C^2$-solution}\label{sec:5}
In this section we prove 
a strong local H\"older estimate in Lemma \ref{2l:local.Lip.C2} below
and
establish existence of a twice continuously
differentiable solution under suitable assumptions
in Theorem \ref{thm:C2} below.
First, we introduce the setting for these results.
\begin{sett} \label{2sett:exists:C2}
Let $( H, \left< \cdot , \cdot \right>_H, \left\| \cdot \right\|_H )$ and $( U, \left< \cdot , \cdot \right>_U, \left\| \cdot \right\|_U )$
be separable $\R$-Hilbert spaces,
let
$ T \in (0,\infty) $,
let $(\Omega, \F, \P, (\mathbb{F}_{t})_{t\in [0,T]})$ be a filtered probability space
satisfying the usual conditions,
let 
$
  (W_t)_{t\in[0,T]}
$
  be an $\textup{Id}_U$-cylindrical $(\mathbb{F}_t)_{t\in[0,T]}$-Wiener process,
let $\Delta_T =\{(s,t)\in [0,T]^2 \colon s \leq t\}$,
let $O \subseteq H$ be an open set,
let $\mathcal{O}\subset O$ be a convex set,
let $ \mu \in C^2(O,H)$, 
$ \sigma \in C^2(O, \HS(U,H))$,
let $X\colon \Omega\to C^{0,1}(\Delta_T\times O,\mathcal{O})$ be a function
which satisfies for all $(s,t)\in\Delta_T$, $x\in \mathcal{O}$ that
$ X^x_{s,\cdot} \colon [s,T] \times \Omega \to \mathcal{O} $
is an $(\mathbb{F}_r)_{r\in[s,T]}$-adapted stochastic process
and
that a.s.~it holds that
\begin{align} \label{2eq:def:X}
  X^x_{s,t} = 
  x
  + \int_s^{ t } \mu(X^x_{s,r} ) \, dr
  +
  \int_s^{ t } \sigma(X^x_{s,r} ) \, dW_r,
\end{align}
let $ \alpha_0,\alpha_1,\beta_0,\beta_1 ,c\in [0,\infty)$,
$ 
  V_0 
,
  V_1  \in C^{ 2 }( O , [0,\infty) ) 
$,
let
$ 
  \bar{V} \colon [0,T] \times \mathcal{O} \to [0,\infty)
$
be a measurable function which satisfies
for all
$ i \in \{ 0, 1 \} $,
$t\in[0,T]$,
$x\in O$
that 
\begin{equation}
\begin{split}
  &\Big\langle
  \mu( x )
  ,
  (\nabla V_i)(x)
  \Big\rangle_H
  +
  \tfrac{ 1 }{ 2 }
  \operatorname{trace}\!\Big(
    \sigma(x) [\sigma(x)]^* 
    ( \operatorname{Hess} V_i )( x )
  \Big)
  \\&
  +
  \tfrac{ 
    1
  }{ 
    2 
    e^{ 
      \alpha_{ i } {t} 
    }
  }
    \|
      \sigma( x )^* ( \nabla V_{ i } )( x )
    \|_U^2
  +
  \mathbbm{1}_{
    \{ 1 \}
  }(i)
  \cdot
  \bar{V}(t,x)
\leq
  \alpha_{ i } V_{ i }(x)
  +
  \beta_{ i },
\end{split}
\end{equation}
let 
$
  \constFun \colon [0,T] \to [0,\infty)
$
be a measurable and integrable function,
let $p\in[2,\infty)$, $\theta\in[0,\infty)$, $\delta\in(0,\infty)$, $q_0,q_1\in(0,\infty]$
satisfy that $\tfrac{\theta}{6p(1+\theta)^3(1+\delta)^2}=\tfrac{1}{q_0}+\tfrac{1}{q_1}$,
assume for all $t\in[0,T]$, $x,y\in \mathcal{O}$, $v\in H\setminus\{0\}$
that
\begin{equation}  \begin{split}\label{2eq:Lip.ass.strong}
  &\Big\langle v,\smallint_0^1\mu'(\lambda x+(1-\lambda)y)+\delta\,d\lambda \,\,v
  \Big\rangle_H
  +\tfrac{1+\delta}{2}
  \Big\|\smallint_0^1\sigma'(\lambda x+(1-\lambda)y)\,d\lambda \,\,v
  \Big\|_{\HS(U,H)}^2
  \\&
  +
  \tfrac{(3p(1+\theta)^3(1+\delta)^2-1)\big\|
  \big\langle v,\smallint_0^1\sigma'(\lambda x+(1-\lambda)y)\,d\lambda\,\, v
  \big\rangle_H\big\|_{\HS(U,\R)}^2}{\|v\|_H^2}
  \leq \|v\|^2_H\cdot\Big(\phi(t)
  +\tfrac{V_0(x)+V_0(y)}{2q_0T e^{\alpha_0t}}
  +\tfrac{\bar{V}(t,x)+\bar{V}(t,y)}{2q_1e^{\alpha_1t}}
  \Big),
\end{split}     \end{equation}
let $\gamma\in[\tfrac{1}{p},\infty)$ satisfy for all $x\in \mathcal{O}$ that
\begin{align}
\begin{split}
 \max\left\{ \|\mu(x)\|_H , \|\sigma(x)\|_{\HS(U,H)} \right\}
 \leq  c(1+V_0(x))^{\gamma},
\end{split}
\end{align}
assume for all $x,y\in \mathcal{O}$, $i\in\{1,2\}$ that
\begin{align} \label{2eq:growth.Dmue}
\begin{split}
 &\max\left\{
    \left\|\smallint_0^1D^i\mu\big(\lambda x+(1-\lambda)y\big)\,d\lambda 
    \right\|_{L^{(i)}(H,H)} ,
    \left\|\smallint_0^1D^i\sigma\big(\lambda x+(1-\lambda)y\big)\,d\lambda 
    \right\|_{L^{(i)}(H,\HS(U,H))}
 \right\}
 \\&
 \leq  c\left(2+V_0(x)+V_0(y)\right)^{\gamma},
\end{split}
\end{align}
assume for all $x_1,x_2,x_3,x_4\in \mathcal{O}$,
$i\in\{1,2\}$
that
\begin{equation}  \begin{split}\label{2eq:Dmu.localLip2}
  &\max\Big\{\big\|\smallint_0^1 
  D^i\mu(\lambda x_1+(1-\lambda)x_2)-D^i\mu(\lambda x_3+(1-\lambda)x_4)
  \,d\lambda
  \big\|_{L^{(i)}(H,H)},
  \\&\qquad\quad
  \big\|\smallint_0^1
     D^i\sigma(\lambda x_1+(1-\lambda)x_2)
        -D^i\sigma(\lambda x_3+(1-\lambda)x_4)
   \,d\lambda
  \big\|_{L^{(i)}(H,\HS(U,H))}
  \Big\}
  \\&
  \leq c\smallint_0^1\lambda\|x_1-x_3\|_H+(1-\lambda)\|x_2-x_4)\|_H
  \,d\lambda\,
  \big(4+\smallsum_{j=1}^4 V_0(x_i)\big)^{\gamma},
\end{split}     \end{equation}
and for all $(s,t)\in\Delta_T$,
$x\in \mathcal{O}$, $v,w\in H$, $h\in\R\setminus\{0\}$ with $x+hw\in \mathcal{O}$
let
$
	D_{s,t}^{x,h}(w),D_s^{x,h}(v,w) \colon \Omega \to H
$
be the functions
which satisfy that
\begin{align} \label{2eq:def:Zvy}
D_{s,t}^{x,h}(w)&= 
\frac{X_{s,t}^{x+h w} -X_{s,t}^{x}}{h}, \text{ and}
\\
\label{2eq:def:Dvwy}
D_{s,t}^{x,h}(v,w)&= 
\frac{\frac{\partial}{\partial x}X_{s,t}^{x+h w}(v) -\frac{\partial}{\partial x}X_{s,t}^{x}(v)}{h}.
\end{align}
\end{sett}

\begin{lemma}\label{2l:C2.implies.C0}
  Assume Setting~\ref{2sett:exists:C2}
  and let $t\in[0,T]$, $x,y\in \mathcal{O}$ satisfy $x\neq y$.
  Then
\begin{equation}  \begin{split}\label{2eq:ass.41}
  &\big\langle x-y,\mu(x)-\mu(y)\big\rangle_H+\tfrac{1}{2}\big\|\sigma(x)-\sigma(y)\big\|^2_{\HS(U,H)}
  +
  \tfrac{(3p(1+\theta)^3(1+\delta)^2-1)
    \big\|\big\langle x-y,\sigma(x)-\sigma(y)\big\rangle_H\big\|_{\HS(U,\R)}^2
    }{\|x-y\|_H^2}
  \\&
  \leq \|x-y\|^2_H\cdot\Big(\phi(t)
  +\tfrac{V_0(x)+V_0(y)}{2q_0Te^{\alpha_0t}}
  +\tfrac{\bar{V}(t,x)+\bar{V}(t,y)}{2q_1e^{\alpha_1t}}
  \Big).
\end{split}     \end{equation}
\end{lemma}
\begin{proof}[Proof of Lemma~\ref{2l:C2.implies.C0}]
  The fundamental theorem of calculus, convexity of $\mathcal{O}$, and 
  assumption~\eqref{2eq:Lip.ass.strong}
  (applied with $v\defeq x-y$
  in the notation of 
  assumption~\eqref{2eq:Lip.ass.strong})
  yield that
\begin{equation}  \begin{split}
  &\big\langle x-y,\mu(x)-\mu(y)\big\rangle_H+\tfrac{1}{2}\big\|\sigma(x)-\sigma(y)\big\|^2_{\HS(U,H)}
  +
  \tfrac{(3p(1+\theta)^3(1+\delta)^2-1)
    \big\|\big\langle x-y,\sigma(x)-\sigma(y)\big\rangle_H\big\|_{\HS(U,\R)}^2
    }{\|x-y\|_H^2}
  \\&
  =\big\langle x-y,\smallint_0^1\mu'(\lambda x+(1-\lambda)y)\,d\lambda (x-y)\big\rangle_H
  +\tfrac{1}{2}\big\|\smallint_0^1\sigma'(\lambda x+(1-\lambda)y)\,d\lambda (x-y)\big\|^2_{\HS(U,H)}
  \\&\qquad
  +
  \tfrac{(3p(1+\theta)^3(1+\delta)^2-1)
    \big\|\big\langle x-y,\smallint_0^1\sigma'(\lambda x+(1-\lambda)y)\,d\lambda (x-y)\big\rangle_H\big\|_{\HS(U,\R)}^2}{\|x-y\|_H^2}
  \\&
  \leq \|x-y\|^2_H\cdot\Big(\phi(t)
  +\tfrac{V_0(x)+V_0(y)}{2q_0Te^{\alpha_0t}}
  +\tfrac{\bar{V}(t,x)+\bar{V}(t,y)}{2q_1e^{\alpha_1t}}
  \Big).
\end{split}     \end{equation}
 This completes the proof of Lemma~\ref{2l:C2.implies.C0}.
\end{proof}
\begin{lemma}[Second order difference processes satisfy affine-linear SDEs] \label{2lem:Dvw:eq}
Assume Setting~\ref{2sett:exists:C2} and let 
$(s,t)\in \Delta_T$, $x\in \mathcal{O}$, $v,w\in H$, $h \in \R\setminus\{0\}$ satisfy  $x + w h \in \mathcal{O}$. 
Then it holds a.s.~that
\begin{align} \label{2eq:eq:Dvwy}
\begin{split}
D_{s,t}^{x,h}(v,w) 
	& = 
\int_s^t \mu' ( X_{s,r}^{x} )D_{s,r}^{x,h}(v,w)\,dr
+
\int_s^t \sigma' ( X_{s,r}^{x} )D_{s,r}^{x,h}(v,w)\,dW_r
\\&
\quad
+\int_s^t \int_0^1\mu''(\lambda X_{s,r}^{x+wh}+(1-\lambda)X_{s,r}^{x})\,d\lambda
\Big(D_{s,r}^{x,h}(w) ,\tfrac{\partial}{\partial x}X_{s,r}^{x+wh}(v)\Big)\,dr
	\\
	& \quad
+\int_s^t \int_0^1\sigma''(\lambda X_{s,r}^{x+wh}+(1-\lambda)X_{s,r}^{x})\,d\lambda
\Big(D_{s,r}^{x,h}(w) ,\tfrac{\partial}{\partial x}X_{s,r}^{x+wh}(v)\Big)\,dW_r.
\end{split}
\end{align}
\end{lemma}
\begin{proof}[Proof of Lemma~\ref{2lem:Dvw:eq}]
The chain rule, continuous differentiability of $\mu$, $\sigma$,
and of $X_{s,r}^{\cdot}(\omega)$, $r\in[s,t]$, $\omega\in\Omega$,
and equation~\eqref{2eq:def:X} imply that for all $z\in O$ it holds a.s.\ that
\begin{equation}  \begin{split}
  \tfrac{\partial}{\partial z}X_{s,t}^z(v)=v
  +\int_s^t\mu'(X_{s,r}^{z})
  \tfrac{\partial}{\partial z}X_{s,r}^z(v)\,dr
  +\int_s^t\sigma'(X_{s,r}^{z})
  \tfrac{\partial}{\partial z}X_{s,r}^z(v)\,dW_r.
\end{split}     \end{equation}
This
and
equation~\eqref{2eq:def:Dvwy}
yield that it holds a.s.~that
\begin{equation}  \begin{split}
D_{s,t}^{x,h}(v,w) 
	 = 
  \tfrac{v-v}{h}&+
\int_s^t \frac{\mu' ( X_{s,r}^{x+ w h}) \tfrac{\partial}{\partial x}X_{s,r}^{x+wh}(v) 
- \mu' ( X_{s,r}^{x} )  \tfrac{\partial}{\partial x}X_{s,r}^{x+wh}(v)
}{h}\, d r
\\&+\int_s^t\frac{
 \mu' ( X_{s,r}^{x} )  \tfrac{\partial}{\partial x}X_{s,r}^{x+wh}(v)
- \mu' ( X_{s,r}^{x} )  \tfrac{\partial}{\partial x}X_{s,r}^{x}(v)
}{h}\, d r
	\\
	& 
+\int_s^t \frac{\sigma' ( X_{s,r}^{x+ w h}) \tfrac{\partial}{\partial x}X_{s,r}^{x+wh}(v)
- \sigma' ( X_{s,r}^{x} )  \tfrac{\partial}{\partial x}X_{s,r}^{x+wh}(v)
}{h}\, d W_r
\\
&+\int_s^t\frac{
\sigma' ( X_{s,r}^{x} )  \tfrac{\partial}{\partial x}X_{s,r}^{x+wh}(v)
- \sigma' ( X_{s,r}^{x} )  \tfrac{\partial}{\partial x}X_{s,r}^{x}(v)
}{h}\, d W_r.
\end{split}     \end{equation}
This,
$\mu\in C^2(O,H)$, $\sigma\in C^2(O,\HS(U,H))$,
\eqref{2eq:def:Dvwy},
\eqref{2eq:def:Zvy},
and
the fundamental theorem of calculus
yield that it holds a.s.~that
\begin{equation}  \begin{split}
D_{s,t}^{x,h}(v,w) 
	& = 
\int_s^t \mu' ( X_{s,r}^{x} )D_{s,r}^{x,h}(v,w)\,dr
+
\int_s^t \sigma' ( X_{s,r}^{x} )D_{s,r}^{x,h}(v,w)\,dW_r
\\&
\quad
+\int_s^t \int_0^1\mu''(\lambda X_{s,r}^{x+wh}+(1-\lambda)X_{s,r}^{x})\,d\lambda
D_{s,r}^{x,h}(w) \tfrac{\partial}{\partial x}X_{s,r}^{x+wh}(v)\,dr
	\\
	& \quad
+\int_s^t \int_0^1\sigma''(\lambda X_{s,r}^{x+wh}+(1-\lambda)X_{s,r}^{x})\,d\lambda
D_{s,r}^{x,h}(w) \tfrac{\partial}{\partial x}X_{s,r}^{x+wh}(v)\,dW_r.
\end{split}     \end{equation}
The proof of Lemma~\ref{2lem:Dvw:eq} is thus completed.
\end{proof}
\subsection{Strong local H\"older estimate}\label{sec:5.1}
The following lemma proves strong local H\"older continuity
of the second-order difference quotients.
\begin{lemma}[Strong local H\"older estimate for second order difference quotients]\label{2l:local.Lip.C2}
  Assume Setting~\ref{2sett:exists:C2} and let 
  $s_1,s_2,t_1,t_2\in[0,T]$, $x_1,x_2\in O$, $v_1,v_2,w_1,w_2\in H$, $h_1,h_2\in\R\setminus\{0\}$
  satisfy that
  $s_1\leq t_1$, $s_2\leq t_2$, $x_1+w_1h_1,x_2+w_2h_2\in O$,
  and $s_1\leq s_2$.
  Then it holds that
  \begin{equation}  \begin{split}
    &\Big\|
     D_{s_1,t_1}^{x_1,h_1}(v_1,w_1)-D_{s_2,t_2}^{x_2,h_2}(v_2,w_2)
     \Big\|_{L^{p}(\P;H)}
    \\&\leq
  (1+\max_{i\in\{1,2\}}\|w_i\|_H)^2(1+\|v_1\|_H)^2(1+4c)(p^2+T)
  \\&\cdot \left(\|x_1-x_2\|_H+\| w_1 h_1- w_2 h_2\|_H
  +\sqrt{|s_1-s_2|}
  +\sqrt{|t_1-t_2|}
  +\|w_1-w_2\|_H
  +\|v_1-v_2\|_H
  \right)
  \\&\cdot
    \Bigg[e^{2\alpha_0 \gamma T}
    \left(\tfrac{12p(1+\theta)^3(1+\delta)^2(1+\gamma)}{\min\{\delta,1\}}
    +2\sqrt{T}
    +\smallint_0^T \tfrac{2\beta_0}{e^{\alpha_0 u}}\,du
    +2\max_{\iota\in\{0,1\},j\in\{1,2\}}
    V_0(x_j+\iota {w}_j{h}_j)
  \right)^{2\gamma+2}
  \\&\quad
  \cdot
  (1+c)^2\max_{\iota\in\{0,1\},j\in\{1,2\}}
  \exp\!\left(
        3\int_{0}^{T}  
        \big[
        \constFun(u)
        +
        \sum_{i=0}^1
        \tfrac{
          \beta_{ i } 
        }{
          q_{i} e^{ \alpha_{ i } u }
        }
        \big]
        \,
        du
        +
        3\sum_{ i = 0 }^1 
        \tfrac{V_i(x_j+\iota {w}_j{h}_j)}{ q_{i} } 
      \right)\Bigg]^{5}
   \\&\quad\cdot
  \max_{\iota\in\{0,1\},i\in\{1,2\}}
   4^{\gamma}e^{\alpha_i T\gamma}\Big(\tfrac{p(1+\theta)(1+\delta)\gamma}{\delta}+\int_{s_i}^{T}\tfrac{\beta_i}{e^{\alpha_iu}}\,du
   +V_0({x_i+\iota w_ih_i})\Big)^{\gamma}
  .
  \end{split}     \end{equation}
\end{lemma}
\begin{proof}[Proof of Lemma~\ref{2l:local.Lip.C2}]
  Without loss of generality we additionally assume throughout this proof
that $q_0+q_1<\infty$ (otherwise apply the result for each 
sufficiently large $n\in\N$
with $\theta_n, \delta_n, q_{0,n}, q_{1,n}\in(0,\infty)$ such that
$q_{0,n}=\min\{q_0,n\}$, $q_{1,n}=\min\{q_1,n\}$,
$\tfrac{\theta_n}{6p(1+\theta)^3(1+\delta)^2}=\tfrac{1}{\min\{q_0,n\}}+\tfrac{1}{\min\{q_1,n\}}$,
$(1+\theta_n)^3(1+\delta_n)^2=(1+\theta)^3(1+\delta)^2$
and let $n\to\infty$).
  Throughout this proof let $Y,a,\zeta^{\mu}\colon[0,T-s_2]\times\Omega\to H$, $b,\zeta^{\sigma}\colon[0,T-s_2]\times\Omega\to\HS(U,H)$,
  and $\eta\colon[0,T-s_2]\times\Omega\to L(H,\HS(U,H))$
  be the functions which satisfy for all $r\in[0,T-s_2]$ that
  \begin{align}
    \label{2eq:ar2}
    a_r&=
      \mu'(X_{s_1,s_2+r}^{x_1})D_{s_1,s_2+r}^{x_1,h_1}(v_1,w_1)
      -
      \mu'(X_{s_2,s_2+r}^{x_2})D_{s_2,s_2+r}^{x_2,h_2}(v_2,w_2)
      \\&\qquad\qquad\nonumber
      +
     \int_0^1
     \mu'' \Big( \lambda X_{s_1,s_2+r}^{x_1+w_1h_1} + (1-\lambda)
     X_{s_1,s_2+r}^{x_1}  \Big)\,d\lambda 
     \Big(D_{s_1,s_2+r}^{x_1,h_1}(w_1),\tfrac{\partial}{\partial x} X_{s_1,s_2+r}^{x_1+w_1h_1}(v_1)\Big)
     \\&\qquad\qquad\nonumber
     -
     \int_0^1\mu'' \Big( \lambda X_{s_2,s_2+r}^{x_2+w_2h_2} + (1-\lambda)
     X_{s_2,s_2+r}^{x_2}  \Big)\,d\lambda 
     \Big(D_{s_2,s_2+r}^{x_2,h_2}(w_2),\tfrac{\partial}{\partial x} X_{s_2,s_2+r}^{x_2+w_2h_2}(v_2)\Big),
  \\
    \label{2eq:br2}
   b_r&=   \sigma'(X_{s_1,s_2+r}^{x_1})D_{s_1,s_2+r}^{x_1,h_1}(v_1,w_1)
      -
      \sigma'(X_{s_2,s_2+r}^{x_2})D_{s_2,s_2+r}^{x_2,h_2}(v_2,w_2)
      \\&\qquad\qquad\nonumber
      +
     \int_0^1
     \sigma'' \Big( \lambda X_{s_1,s_2+r}^{x_1+w_1h_1} + (1-\lambda)
     X_{s_1,s_2+r}^{x_1}  \Big)\,d\lambda 
     \Big(D_{s_1,s_2+r}^{x_1,h_1}(w_1),\tfrac{\partial}{\partial x} X_{s_1,s_2+r}^{x_1+w_1h_1}(v_1)\Big)
     \\&\qquad\qquad\nonumber
     -
     \int_0^1\sigma'' \Big( \lambda X_{s_2,s_2+r}^{x_2+w_2h_2} + (1-\lambda)
     X_{s_2,s_2+r}^{x_2}  \Big)\,d\lambda 
     \Big(D_{s_2,s_2+r}^{x_2,h_2}(w_2),\tfrac{\partial}{\partial x} X_{s_2,s_2+r}^{x_2+w_2h_2}(v_2)\Big),
  \end{align}
and
  \begin{align}
    \label{2eq:Yr2}
     Y_r&=
     D_{s_1,s_2+r}^{x_1,h_1}(v_1,w_1)-D_{s_2,s_2+r}^{x_2,h_2}(v_2,w_2),
   \\\label{eq:eta}
     \eta_r&=\sigma'(X_{s_2,s_2+r}^{x_2}),
   \\
     \zeta_r^{\mu}&=\label{2eq:zeta.mue2}
      \Big(\mu'(X_{s_1,s_2+r}^{x_1})-\mu'(X_{s_2,s_2+r}^{x_2})
      \Big)D_{s_1,s_2+r}^{x_1,h_1}(v_1,w_1)
   \\&\qquad\nonumber
      +
     \int_0^1
     \mu'' \Big( \lambda X_{s_1,s_2+r}^{x_1+w_1h_1} + (1-\lambda)
     X_{s_1,s_2+r}^{x_1}  \Big)\,d\lambda 
     \Big(D_{s_1,s_2+r}^{x_1,h_1}(w_1),\tfrac{\partial}{\partial x} X_{s_1,s_2+r}^{x_1+w_1h_1}(v_1)\Big)
     \\&\qquad\nonumber
     -
     \int_0^1\mu'' \Big( \lambda X_{s_2,s_2+r}^{x_2+w_2h_2} + (1-\lambda)
     X_{s_2,s_2+r}^{x_2}  \Big)\,d\lambda 
     \Big(D_{s_2,s_2+r}^{x_2,h_2}(w_2),\tfrac{\partial}{\partial x} X_{s_2,s_2+r}^{x_2+w_2h_2}(v_2)\Big),
   \\
     \zeta_r^{\sigma}&=\label{2eq:zeta.sigma2}
      \Big(\sigma'(X_{s_1,s_2+r}^{x_1})-\sigma'(X_{s_2,s_2+r}^{x_2})
      \Big)D_{s_1,s_2+r}^{x_1,h_1}(v_1,w_1)
   \\&\qquad\nonumber
      +
     \int_0^1
     \sigma'' \Big( \lambda X_{s_1,s_2+r}^{x_1+w_1h_1} + (1-\lambda)
     X_{s_1,s_2+r}^{x_1}  \Big)\,d\lambda 
     \Big(D_{s_1,s_2+r}^{x_1,h_1}(w_1),\tfrac{\partial}{\partial x} X_{s_1,s_2+r}^{x_1+w_1h_1}(v_1)\Big)
     \\&\qquad\nonumber
     -
     \int_0^1\sigma'' \Big( \lambda X_{s_2,s_2+r}^{x_2+w_2h_2} + (1-\lambda)
     X_{s_2,s_2+r}^{x_2}  \Big)\,d\lambda 
     \Big(D_{s_2,s_2+r}^{x_2,h_2}(w_2),\tfrac{\partial}{\partial x} X_{s_2,s_2+r}^{x_2+w_2h_2}(v_2)\Big).
  \end{align}
  Note that $Y,a,b,\eta,\zeta^{\mu},\zeta^{\sigma}$
  are $(\mathbb{F}_{s_2+r})_{r\in[0,T-s_2]}$-adapted stochastic processes
  with continuous sample paths 
  and, therefore,
  measurable.
  Lemma~\ref{2lem:Dvw:eq} implies that for all $t\in[0,T-s_2]$ it holds
  a.s.\ that
  \begin{equation}  \begin{split}\label{2eq:D.Y2}
     Y_t=&D_{s_1,s_2+t}^{x_1,h_1}(v_1,w_1)-D_{s_2,s_2+t}^{x_2,h_2}(v_2,w_2)
     =
     D_{s_1,s_2}^{x_1,h_1}(v_1,w_1)
     +
     \int_{s_2}^{s_2+t}a_{r-s_2}\,dr
     +
     \int_{s_2}^{s_2+t}b_{r-s_2}\,dW_r
     \\&
     =
     D_{s_1,s_2}^{x_1,h_1}(v_1,w_1)
     +
     \int_{0}^{t}a_{r}\,dr
     +
     \int_{0}^{t}b_{r}\,d(W_{s_2+r}-W_{s_2}).
  \end{split}     \end{equation}
  We consider the one-sided affine-linear growth condition for the It\^o process $Y$.
  Equations~\eqref{2eq:ar2}, \eqref{2eq:Yr2}, and \eqref{2eq:zeta.mue2}
  imply for all $r\in[0,T-s_2]$ that
  \begin{equation}  \begin{split}\label{2eq:for.a2}
    a_r
    &=
      \mu'(X_{s_2,s_2+r}^{x_2})\Big(D_{s_1,s_2+r}^{x_1,h_1}(v_1,w_1)
      -D_{s_2,s_2+r}^{x_2,h_2}(v_2,w_2)\Big)
   \\&\qquad\qquad
      +
      \Big(\mu'(X_{s_1,s_2+r}^{x_1})-\mu'(X_{s_2,s_2+r}^{x_2})
      \Big)D_{s_1,s_2+r}^{x_1,h_1}(v_1,w_1)
   \\&\qquad\qquad
      +
     \int_0^1
     \mu'' \Big( \lambda X_{s_1,s_2+r}^{x_1+w_1h_1} + (1-\lambda)
     X_{s_1,s_2+r}^{x_1}  \Big)\,d\lambda 
     \Big(D_{s_1,s_2+r}^{x_1,h_1}(w_1),\tfrac{\partial}{\partial x} X_{s_1,s_2+r}^{x_1+w_1h_1}(v_1)\Big)
     \\&\qquad\qquad
     -
     \int_0^1\mu'' \Big( \lambda X_{s_2,s_2+r}^{x_2+w_2h_2} + (1-\lambda)
     X_{s_2,s_2+r}^{x_2}  \Big)\,d\lambda 
     \Big(D_{s_2,s_2+r}^{x_2,h_2}(w_2),\tfrac{\partial}{\partial x} X_{s_2,s_2+r}^{x_2+w_2h_2}(v_2)\Big),
   \\&
     =\mu'(X_{s_2,s_2+r}^{x_2}) Y_r+\zeta_r^{\mu}.
  \end{split}     \end{equation}
  Analogously, equations~\eqref{2eq:br2}, \eqref{2eq:Yr2}, \eqref{2eq:zeta.sigma2},
  and~\eqref{eq:eta} imply for all $r\in[0,T-s_2]$ that
  \begin{equation}  \begin{split}\label{2eq:for.b2}
    b_r
     =\sigma'(X_{s_2,s_2+r}^{x_2}) Y_r+\zeta_r^{\sigma}
     =\eta_rY_r+\zeta_r^{\sigma}.
  \end{split}     \end{equation}
  Equation~\eqref{2eq:for.a2},
  the Cauchy-Schwarz inequality, and Young's inequality
  yield for all $r\in[0,T-s_2]$ that
  \begin{equation}  \begin{split}\label{2eq:estimate.a2}
    \langle Y_r,a_r\rangle_H
    &\leq \left\langle Y_r,
     \mu' ( X_{s_2,s_2+r}^{x_2}) \,Y_r
     \right\rangle_H
    +\big\| Y_r\big\|_H
    \big\| \zeta_r^{\mu}\big\|_H
   \\&
    \leq\left\langle Y_r,
     \Big(\mu' ( X_{s_2,s_2+r}^{x_2})+\delta\Big) \,Y_r
     \right\rangle _H
   +\tfrac{1}{4\delta}
   \Big\|\zeta_r^{\mu}   \Big\|_{H}^2.
  \end{split}     \end{equation}
  Similarly, equation~\eqref{2eq:for.b2}, the Cauchy-Schwarz inequality, and Young's inequality
  imply for all $r\in[0,T-s_2]$ that
  \begin{equation}  \begin{split}\label{2eq:estimate.b2}
    &\tfrac{1}{2}\|b_r\|_{\HS(U,H)}^2
    +\tfrac{p(1+\theta)-2}{2}\tfrac{\left\|
    \langle 
     Y_r, b_r
    \rangle_H
    \right\|_{\HS(U,\R)}^2
    }{
      \|Y_r\|_H^2
    }
    =
    \tfrac{1}{2}\|\eta_rY_r+\zeta^{\sigma}_r\|_{\HS(U,H)}^2
    +\tfrac{p(1+\theta)-2}{2}\tfrac{\left\|
    \langle 
     Y_r, \eta_rY_r+\zeta^{\sigma}_r
    \rangle_H
    \right\|_{\HS(U,\R)}^2
    }{
      \|Y_r\|_H^2
    }
    \\&
    =
    \tfrac{1}{2}\|\eta_rY_r\|_{\HS(U,H)}^2
    +\langle \eta_rY_r,\zeta^{\sigma}_r\rangle_{\HS(U,H)}
    +
    \tfrac{1}{2}\|\zeta^{\sigma}_r\|_{\HS(U,H)}^2
    \\&\qquad
    +\tfrac{p(1+\theta)-2}{2}\tfrac{\left\|
    \langle 
     Y_r, \eta_rY_r
    \rangle_H
    \right\|_{\HS(U,\R)}^2
    }{
      \|Y_r\|_H^2
    }
    +\tfrac{p(1+\theta)-2}{2}\tfrac{
    2\left\langle \langle Y_r,\eta_rY_r\rangle_H,\langle Y_r,\zeta^{\sigma}_r\rangle_H
    \right\rangle_{\HS(U,\R)}
    }{
      \|Y_r\|_H^2
    }
    +\tfrac{p(1+\theta)-2}{2}\tfrac{\left\|
    \langle 
     Y_r, \zeta^{\sigma}_r
    \rangle_H
    \right\|_{\HS(U,\R)}^2
    }{
      \|Y_r\|_H^2
    }
    \\&
    \leq
    \tfrac{1}{2}\|\eta_rY_r\|_{\HS(U,H)}^2
    +\| \eta_rY_r\|_{\HS(U,H)}\left\|\zeta^{\sigma}_r\right\|_{\HS(U,H)}
    +
    \tfrac{1}{2}\|\zeta^{\sigma}_r\|_{\HS(U,H)}^2
    \\&\qquad
    +\tfrac{p(1+\theta)-2}{2}\tfrac{\left\|
    \langle 
     Y_r, \eta_rY_r
    \rangle_H
    \right\|_{\HS(U,\R)}^2
    }{
      \|Y_r\|_H^2
    }
    +\tfrac{p(1+\theta)-2}{2}\tfrac{
    2\left\| \langle Y_r,\eta_rY_r\rangle_H\right\|_{\HS(U,\R)}
    \left\|\langle Y_r,\zeta^{\sigma}_r\rangle_H\right\|_{\HS(U,\R)}
    }{
      \|Y_r\|_H^2
    }
    +\tfrac{p(1+\theta)-2}{2}\tfrac{\left\|
    \langle 
     Y_r, \zeta^{\sigma}_r
    \rangle_H
    \right\|_{\HS(U,\R)}^2
    }{
      \|Y_r\|_H^2
    }
    \\&
    \leq
    \tfrac{1+\delta}{2}\|\eta_rY_r\|_{\HS(U,H)}^2
    +\tfrac{(p(1+\theta)-2)(1+\delta)}{2}\tfrac{\left\|
    \langle 
     Y_r, \eta_rY_r
    \rangle_H
    \right\|_{\HS(U,\R)}^2
    }{
      \|Y_r\|_H^2
    }
    +
    (\tfrac{1+\delta^{-1}}{2}+\tfrac{(p(1+\theta)-2)(1+\delta^{-1})}{2})\|\zeta^{\sigma}_r\|_{\HS(U,H)}^2
    \\&
    =\tfrac{1+\delta}{2}\|\eta_rY_r\|_{\HS(U,H)}^2
    +\tfrac{(p(1+\theta)-2)(1+\delta)}{2}\tfrac{\left\|
    \langle 
     Y_r, \eta_rY_r
    \rangle_H
    \right\|_{\HS(U,\R)}^2
    }{
      \|Y_r\|_H^2
    }
    +
    \tfrac{(p(1+\theta)-1)(1+\delta^{-1})}{2}\|\zeta^{\sigma}_r\|_{\HS(U,H)}^2.
  \end{split}     \end{equation}
  Next, \eqref{2eq:estimate.a2}, \eqref{2eq:estimate.b2}, \eqref{2eq:Yr2},
  the fact that $(p(1+\theta)-2)(1+\delta)/2\leq 3p(1+\theta)^3(1+\delta)^2-1$,
  and the hypothesis~\eqref{2eq:Lip.ass.strong}
  imply for all $r\in[0,T-s_2]$ that
  \begin{equation}  \begin{split}
    &\langle 
     Y_r, a_r
    \rangle_H
    +\tfrac{1}{2}\|b_r\|_{\HS(U,H)}^2
    +\tfrac{p(1+\theta)-2}{2}\tfrac{\left\|
    \langle 
     Y_r, b_r
    \rangle_H
    \right\|_{\HS(U,\R)}^2
    }{
      \|Y_r\|_H^2
    }
    \\&
    \leq\left\langle Y_r,
     \Big(\mu' ( X_{s_2,s_2+r}^{x_2}  ) +\delta\Big)
     \,Y_r
     \right\rangle _H
    +\tfrac{1+\delta}{2}\|\sigma' ( X_{s_2,s_2+r}^{x_2}  )Y_r\|_{\HS(U,H)}^2
   \\&
   \qquad
    +\tfrac{(p(1+\theta)-2)(1+\delta)}{2}\tfrac{\left\|
    \langle 
     Y_r, \sigma' ( X_{s_2,s_2+r}^{x_2}  )Y_r
    \rangle_H
    \right\|_{\HS(U,\R)}^2
    }{
      \|Y_r\|_H^2
    }
    +
    \tfrac{(p(1+\theta)-1)(1+\delta^{-1})}{2}\|\zeta^{\sigma}_r\|_{\HS(U,H)}^2
  + \tfrac{1}{4\delta}
   \|\zeta_r^{\mu}   \|_{H}^2
    \\&
    \leq \|Y_r\|_H^2\left(\phi(s_2+r)
  +\tfrac{V_0(X_{s_2,s_2+r}^{x_2})+V_0(X_{s_2,s_2+r}^{x_2})}{2q_0Te^{\alpha_0(s_2+r)}}
  +\tfrac{\bar{V}(s_2+r,X_{s_2,s_2+r}^{x_2})+\bar{V}(s_2+r,X_{s_2,s_2+r}^{x_2})}{2q_1e^{\alpha_1(s_2+r)}}
  \right)
  \\&\qquad
    +
    \tfrac{(p(1+\theta)-1)(1+\delta^{-1})}{2}\|\zeta^{\sigma}_r\|_{\HS(U,H)}^2
  + \tfrac{1}{4\delta}
   \|\zeta_r^{\mu}   \|_{H}^2.
  \end{split}     \end{equation}
  This, \eqref{2eq:D.Y2},
  nonnegativity of $\constFun,V_0,\bar{V}$,
  Proposition~\ref{prop:moments:hilbert}
  (applied for all $t\in(0,T-s_2]$ with $p\defeq p(1+\theta)$, $T\defeq t$, $\P\defeq \P(\cdot|\mathbb{F}_{s_2})$, 
  $\mathbb{F}=(\mathbb{F}_{s_2+r})_{r\in[0,t]}$,
  $W\defeq W_{s_2+\cdot}-W_{s_2}$, $X\defeq Y$,
  \begin{align}
    \alpha&=\Big(\phi(s_2+r)
    +\tfrac{V_0(X_{s_2,s_2+r}^{x_2})}{q_0Te^{\alpha_0(s_2+r)}}
    +\tfrac{\bar{V}(s_2+r,X_{s_2,s_2+r}^{x_2})}{q_1e^{\alpha_1(s_2+r)}}\Big)_{r\in[0,t]},
  \\
  \beta&=
    \Big(\left((p(1+\theta)-1)(1+\delta^{-1})\|\zeta^{\sigma}_r\|_{\HS(U,H)}^2
  + \tfrac{1}{2\delta}
   \|\zeta_r^{\mu}\|_H^2\right)^{\frac{1}{2}}\Big)_{r\in[0,t]},
  \end{align}
  $q_1=p$, $q_2=p(1+\theta)/\theta$
  in the notation of
  Proposition~\ref{prop:moments:hilbert}),
  the fact that
  $\|\cdot\|_{L^{\frac{p(1+\theta)}{2}}(\P(\cdot|\mathbb{F}_{s_2});\R)}$ is a norm,
  the triangle inequality,
  the fact that $\tfrac{\theta}{p(1+\theta)}\geq \tfrac{1}{q_0}+\tfrac{1}{q_1}$,
  and Lemma~\ref{lem:multiple_exp:guess}
  (applied for all $t\in(0,T-s_2]$ with $T\defeq t$, $s\defeq s_2$, $\P\defeq \P(\cdot|\mathbb{F}_{s_2})$,
  $X^{1}\defeq X_{s_2,\cdot}^{x_2}$,
  $X^{2}\defeq X_{s_2,\cdot}^{x_2}$,
  $X^{3}\defeq X_{s_2,\cdot}^{x_2}$,
  $X^{4}\defeq X_{s_2,\cdot}^{x_2}$,
  $q\defeq 6p(1+\theta)^3(1+\delta)^2/\theta$
  in the notation of
  Lemma~\ref{lem:multiple_exp:guess})
  yield for all $t\in[0,T-s_2]$ that
  it holds a.s.\ that
  \begin{equation}  \begin{split}
    &\Big\|
     D_{s_1,s_2+t}^{x_1,h_1}(v_1,w_1)-D_{s_2,s_2+t}^{x_2,h_2}(v_2,w_2)
     \Big\|_{L^{p}(\P(\cdot|\mathbb{F}_{s_2});H)}
     =
    \Big\|
     Y_t
     \Big\|_{L^{p}(\P(\cdot|\mathbb{F}_{s_2});H)}
     \\&
     \leq
     \left(\|Y_0\|_{L^{p(1+\theta)}(\P(\cdot|\mathbb{F}_{s_2});H)}^2
     +\int_0^t\Big\|\Big(\tfrac{(p(1+\theta)-1)(1+\delta)}{\delta}\|\zeta_r^{\sigma}\|_{\HS(U,H)}^2
     +\tfrac{1}{2\delta}\|\zeta_r^{\mu}\|_{H)}^2\Big)^{\frac{1}{2}}\Big\|_{L^{p(1+\theta)}(\P(\cdot|\mathbb{F}_{s_2});\R)}^2
     \,ds\right)^{\!\frac{1}{2}}
  \\&
  \cdot
   \Big\|\exp\Big(\smallint_0^{t}
  \big[\phi(s_2+r)
  +\tfrac{V_0(X_{s_2,s_2+r}^{x_2})}{q_0Te^{\alpha_0(s_2+r)}}
  +\tfrac{\bar{V}(s_2+r,X_{s_2,s_2+r}^{x_2})}{q_1e^{\alpha_1(s_2+r)}}
  \big]
   \,dr\Big)\Big\|_{L^{\frac{p(1+\theta)}{\theta}}(\P(\cdot|\mathbb{F}_{s_2});\R)}
     \\&
\leq
     \Big\|
     D_{s_1,s_2}^{x_1,h_1}(v_1,w_2)
     \Big\|_H
     +\left(\int_0^t\tfrac{(p(1+\theta)-1)(1+\delta)}{\delta}\|\zeta_r^{\sigma}\|_{L^{p(1+\theta)}(\P(\cdot|\mathbb{F}_{s_2});\HS(U,H))}^2
     +\tfrac{1}{2\delta}\|\zeta_r^{\mu}\|_{L^{p(1+\theta)}(\P(\cdot|\mathbb{F}_{s_2});H)}^2\,ds
     \right)^{\!\frac{1}{2}}
  \\&\qquad
  \cdot
   \exp\Big(\smallint_{s_2}^{T}
   \big[
  \phi(r)
        +
        \tfrac{
          \beta_{ 0 }
        }{
          q_{0} e^{ \alpha_{ 0 } r }
        }
        + 
        \tfrac{ \beta_{ 1 } }{ 	q_{ 1 } e^{ \alpha_{ 1 } r } }
        \big]
        \,
        dr
        +\sum_{i=0}^1 \tfrac{V_i(x_2)}{q_ie^{\alpha_is_2}}
   \Big).
  \end{split}     \end{equation}
  This implies that
  \begin{equation}  \begin{split}\label{2eq:deltaD}
    &\Big\|
     D_{s_1,t_2}^{x_1,h_1}(v_1,w_1)-D_{s_2,t_2}^{x_2,h_2}(v_2,w_2)
     \Big\|_{L^{p}(\P;H)}
     \\&
    =\Big\|\Big\|
     D_{s_1,s_2+t_2-s_2}^{x_1,h_1}(v_1,w_1)-D_{s_2,s_2+t_2-s_2}^{x_2,h_2}(v_2,w_2)
     \Big\|_{L^{p}(\P(\cdot|\mathbb{F}_{s_2});H)}
     \Big\|_{L^{p}(\P;\R)}
 \\&
 \leq
     \left(
    \Big\|
     D_{s_1,s_2}^{x_1,h_1}(v_1,w_1)
     \Big\|_{L^{p}(\P;H)}
     +
     \sup_{r\in[0,T-s_2]}\max\big\{
    \big\|\zeta^{\mu}_r\big\|_{L^{p(1+\theta)}(\P;H)},
    \big\|\zeta^{\sigma}_r\big\|_{L^{p(1+\theta)}(\P;\HS(U,H))}
    \big\}
    \sqrt{\tfrac{t_2 p(1+\theta)(1+\delta)}{\delta}}
    \right)
  \\&
  \cdot
   \exp\Big(\smallint_{s_2}^{T}
   \big[
  \phi(r)
        +
        \tfrac{
          \beta_{ 0 }
        }{
          q_{0} e^{ \alpha_{ 0 } r }
        }
        + 
        \tfrac{ \beta_{ 1 } }{ 	q_{ 1 } e^{ \alpha_{ 1 } r } }
        \big]
        \,
        dr
        +\sum_{i=0}^1 \tfrac{V_i(x_2)}{q_ie^{\alpha_is_2}}
   \Big).
  \end{split}     \end{equation}
  Moreover, \eqref{2eq:zeta.mue2},
  the fact that for all $a_1,a_2,a_3\in\R$, $b_1,b_2,b_3\in\R$ it holds that
    $a_1a_2a_3-b_1b_2b_3=(a_1-b_1)a_2a_3+b_1(a_2-b_2)a_3+b_1b_2(a_3-b_3)$,
  and
  the triangle inequality
  yield for all $r\in[0,T-s_2]$ that
  \begin{equation}  \begin{split}
   &\Big\|\zeta_r^{\mu}   \Big\|_{L^{p(1+\theta)}(\P;H)}
   \leq \left\|\Big\|
      \mu'(X_{s_1,s_2+r}^{x_1})-\mu'(X_{s_2,s_2+r}^{x_2})
      \Big\|_{L(H,H)}
      \Big\|D_{s_1,s_2+r}^{x_1,h_1}(v_1,w_1)\Big\|_H
      \right\|_{L^{p(1+\theta)}(\P;\R)}
   \\&\nonumber
      +
      \bigg\|
      \Big\|
     \int_0^1
     \mu'' \Big( \lambda X_{s_1,s_2+r}^{x_1+w_1h_1} + (1-\lambda)
     X_{s_1,s_2+r}^{x_1}  \Big)
     -\mu'' \Big( \lambda X_{s_2,s_2+r}^{x_2+w_2h_2} + (1-\lambda)
     X_{s_2,s_2+r}^{x_2}  \Big)\,d\lambda 
     \Big\|_{L^{(2)}(H,H)}
     \\&\qquad
     \cdot
     \Big\|
     D_{s_1,s_2+r}^{x_1,h_1}(w_1)
     \Big\|_H
     \Big\|\tfrac{\partial}{\partial x} X_{s_1,s_2+r}^{x_1+w_1h_1}(v_1)
     \Big\|_H
     \bigg\|_{L^{p(1+\theta)}(\P;\R)}
     \\&\nonumber
      +
      \bigg\|
      \Big\|
     \int_0^1
          \mu'' \Big( \lambda X_{s_2,s_2+r}^{x_2+w_2h_2} + (1-\lambda)
     X_{s_2,s_2+r}^{x_2}  \Big)\,d\lambda 
     \Big\|_{L^{(2)}(H,H)}
     \\&\qquad
     \cdot
     \Big\|
     D_{s_1,s_2+r}^{x_1,h_1}(w_1)
     -D_{s_2,s_2+r}^{x_2,h_2}(w_2)
     \Big\|_H
     \Big\|\tfrac{\partial}{\partial x} X_{s_1,s_2+r}^{x_1+w_1h_1}(v_1)
     \Big\|_H
     \bigg\|_{L^{p(1+\theta)}(\P;\R)}
     \\&\nonumber
      +
      \bigg\|
      \Big\|
     \int_0^1
          \mu'' \Big( \lambda X_{s_2,s_2+r}^{x_2+w_2h_2} + (1-\lambda)
     X_{s_2,s_2+r}^{x_2}  \Big)\,d\lambda 
     \Big\|_{L^{(2)}(H,H)}
     \\&\qquad
     \cdot
     \Big\|
     D_{s_2,s_2+r}^{x_2,h_2}(w_2)
     \Big\|_H
     \Big\|
       \tfrac{\partial}{\partial x} X_{s_1,s_2+r}^{x_1+w_1h_1}(v_1)
       -
       \tfrac{\partial}{\partial x} X_{s_2,s_2+r}^{x_2+w_2h_2}(v_2)
     \Big\|_H
     \bigg\|_{L^{p(1+\theta)}(\P;\R)}.
  \end{split}     \end{equation}
  Then~\eqref{2eq:growth.Dmue},
  \eqref{2eq:Dmu.localLip2},
  and
  H\"older's inequality
  (applied with
  $\tfrac{1}{p(1+\theta)}=3\tfrac{1}{3p(1+\theta)(1+\delta)}+\tfrac{\delta}{p(1+\theta)(1+\delta)}$)
  show for all $r\in[0,T-s_2]$ that
  \begin{equation}  \begin{split}\label{2eq:estimate.zetamue.1}
   &\Big\|\zeta_r^{\mu}   \Big\|_{L^{{p(1+\theta)}}(\P;H)}
   \leq
     \Big(1+\Big\|D_{s_1,s_2+r}^{x_1,h_1}(v_1,w_1)
     \Big\|_{L^{3{p(1+\theta)}(1+\delta)}(\P;H)}
     \Big)
   \\&\cdot
     \Big(1+
     \max_{i\in\{1,2\}}
     \Big\|D_{s_i,s_2+r}^{x_i,h_i}(w_i)\Big\|_{L^{3p(1+\theta)(1+\delta)}(\P;H)}
     \Big)
     \Big(1+
     \Big\|
       \tfrac{\partial}{\partial x} X_{s_1,s_2+r}^{x_1+w_1h_1}(v_1)
     \Big\|_{L^{3{p(1+\theta)}(1+\delta)}(\P;H)}
     \Big)
   \\&\cdot
   \bigg(
     2c\max_{\iota\in\{0,1\}}\Big\|X_{s_1,s_2+r}^{x_1+\iota w_1h_1}-X_{s_2,s_2+r}^{x_2+\iota w_2h_2}
     \Big\|_{L^{3{p(1+\theta)}(1+\delta)}(\P;H)}
     \\&\qquad
     +c
     \Big\|
     D_{s_1,s_2+r}^{x_1,h_1}(w_1)
     -D_{s_2,s_2+r}^{x_2,h_2}(w_2)
     \Big\|_{L^{3{p(1+\theta)}(1+\delta)}(\P;H)}
     \\&\qquad
     +
     c\Big\|
       \tfrac{\partial}{\partial x} X_{s_1,s_2+r}^{x_1+w_1h_1}(v_1)
       -
       \tfrac{\partial}{\partial x} X_{s_2,s_2+r}^{x_2+w_2h_2}(v_2)
     \Big\|_{L^{3{p(1+\theta)}(1+\delta)}(\P;H)}
     \bigg)
   \\&\cdot
   \max_{\iota\in\{0,1\}}
     \Big\|
     \Big(4
     +V_0(X_{s_1,s_2+r}^{x_1+\iota w_1h_1})
     +V_0(X_{s_1,s_2+r}^{x_1})
     +V_0(X_{s_2,s_2+r}^{x_2+\iota w_2h_2})
     +V_0(X_{s_2,s_2+r}^{x_2})
     \Big)^{\gamma}
   \Big\|_{L^{\frac{{p(1+\theta)}(1+\delta)}{\delta }}(\P;\R)}.
  \end{split}     \end{equation}
  This, Fatou's lemma, and the triangle inequality yield for all $r\in[0,T-s_2]$ that
  \begin{equation}  \begin{split}\label{2eq:estimate.zetamue.2}
   &\Big\|\zeta_r^{\mu}   \Big\|_{L^{{p(1+\theta)}}(\P;H)}
   \leq
     \Big(1+\liminf_{(0,\infty)\ni\eps\to 0}\Big\|
     \tfrac{D_{s_1,s_2+r}^{x_1+h_1w_1,\eps}(v_1)-D_{s_1,s_2+r}^{x_1,\eps}(v_1)}{h_1}
     \Big\|_{L^{3{p(1+\theta)}(1+\delta)}(\P;H)}
     \Big)
   \\&\cdot
     \Big(1+
     \max_{i\in\{1,2\}}
     \tfrac{\big\|X_{s_i,s_2+r}^{x_i+w_ih_i}-X_{s_i,s_2+r}^{x_i}\big\|_{L^{3p(1+\theta)(1+\delta)}(\P;H)}}{|h_i|}
     \Big)
   \\&\cdot
     \Big(1+
    \liminf_{(0,\infty)\ni\eps\to 0}
    \tfrac{\big\|X_{s_1,s_2+r}^{x_1+w_1h_1+\eps v_1}-X_{s_1,s_2+r}^{x_1+w_1h_1}
     \big\|_{L^{3{p(1+\theta)}(1+\delta)}(\P;H)}
     }{\eps}
     \Big)
   \\&\cdot
   \bigg(
     2c\max_{\iota\in\{0,1\}}\Big\|X_{s_1,s_2+r}^{x_1+\iota w_1h_1}-X_{s_2,s_2+r}^{x_2+\iota w_2h_2}
     \Big\|_{L^{3{p(1+\theta)}(1+\delta)}(\P;H)}
     \\&\qquad
     +c
     \Big\|
     D_{s_1,s_2+r}^{x_1,h_1}(w_1)
     -D_{s_2,s_2+r}^{x_2,h_2}(w_2)
     \Big\|_{L^{3{p(1+\theta)}(1+\delta)}(\P;H)}
     \\&\qquad
     +
     c\liminf_{(0,\infty)\ni\eps\to 0}\Big\|
       D_{s_1,s_2+r}^{x_1+w_1h_1,\eps}(v_1)
       -
       D_{s_2,s_2+r}^{x_2+w_2h_2,\eps}(v_2)
     \Big\|_{L^{3{p(1+\theta)}(1+\delta)}(\P;H)}
     \bigg)
   \\&\cdot
   4^\gamma
   \max_{\iota\in\{0,1\},i\in\{1,2\}}
     \Big\|
     1 +V_0(X_{s_i,s_2+r}^{x_i+\iota w_ih_i})
   \Big\|_{L^{\frac{{p(1+\theta)}(1+\delta)\gamma}{\delta }}(\P;\R)}^{\gamma}.
  \end{split}     \end{equation}
  Lemma \ref{2l:C2.implies.C0}
  and Lemma \ref{lem:Hoelder}
  (applied for all $t\in[0,T]$, $u_1\in[0,t]$, $u_2\in[u_1,t]$, $y_1,y_2\in O$
  with
  $\mathcal{O}\defeq O$,
  $p\defeq 6p(1+\theta)^2(1+\delta)^2$,
  $s_1\defeq u_1$,
  $s_2\defeq u_2$,
  $t_1\defeq t$,
  $t_2\defeq t$,
  $x_1\defeq y_1$,
  $x_2\defeq y_2$
  in the notation of Lemma \ref{lem:Hoelder})
  yield for all $t\in[0,T]$, $u_1,u_2\in[0,t]$, $y_1,y_2\in O$, $\tilde{v}_1,\tilde{v}_2\in H$ with $u_1\leq u_2$ that
\begin{equation}  \begin{split}\label{eq:fromLemma32}
  &\|X_{u_1,t}^{y_1}-X_{u_2,t}^{y_2}\|_{L^{3p(1+\theta)(1+\delta)}(\P;H)}
  \leq\|X_{u_1,t}^{y_1}-X_{u_2,t}^{y_2}\|_{L^{6p(1+\theta)^2(1+\delta)^2}(\P;H)}
  \\&
  \leq
\|y_1-y_2\|_H
  \exp\!\left(
        \int_{u_1}^{T}  
        \big[
        \constFun(u)
        +
        \sum_{i=0}^1
        \tfrac{
          \beta_{ i } 
        }
        { q_i e^{ \alpha_{ i } u } }
        \big]
        \,
        du
        +
        \sum_{ i = 0 }^1 
        \tfrac{V_i(y_1)+V_i(y_2)}{2q_ie^{\alpha_i u_1} } 
      \right)
  \\&\quad+
c e^{\alpha_0 \gamma|u_2-u_1|}\Big|6p(1+\theta)^3(1+\delta)^2\gamma+e^{-\alpha_0 u_1}V_0(y_2)
+\textstyle\int_{u_1}^T\tfrac{\beta_0}{e^{\alpha_0u}}\,du\Big|^{\gamma}
\sqrt{|u_2-u_1|}
\\&
\qquad\quad\cdot
\big(\sqrt{T}+6p(1+\theta)^3(1+\delta)^2\big)
  \exp\!\left(
        \int_{u_1}^{T}  
        \big[
        \constFun(u)
        +
        \tfrac{
          \beta_{ 0 } 
        }
        { q_0 e^{ \alpha_{ 0 } u } }
        + 
        \tfrac{ \beta_{ 1 } }
        { q_1 e^{ \alpha_{ 1 } u } }
        \big]
        \,
        du
        +
        \sum_{ i = 0 }^1 
        \tfrac{V_i(y_2)}{q_i e^{\alpha_i u_1} } 
      \right)
\\&\leq \left(\|y_1-y_2\|_H+\sqrt{|u_1-u_2|}\right)
  \\&\quad\cdot
    e^{2\alpha_0 \gamma T}
    \left(\tfrac{12p(1+\theta)^3(1+\delta)^2(1+\gamma)}{\min\{\delta,1\}}
    +2\sqrt{T}
    +\smallint_0^T \tfrac{2\beta_0}{e^{\alpha_0 u}}\,du
    +2\max_{\iota\in\{0,1\},j\in\{1,2\}}
    V_0(y_j+\iota \tilde{v}_j\tilde{h}_j)
  \right)^{2\gamma+2}
  \\&\quad
  \cdot
  (1+c)^2\max_{\iota\in\{0,1\},j\in\{1,2\}}
  \exp\!\left(
        3\int_{0}^{T}  
        \big[
        \constFun(u)
        +
        \sum_{i=0}^1
        \tfrac{
          \beta_{ i } 
        }{
          q_{i} e^{ \alpha_{ i } u }
        }
        \big]
        \,
        du
        +
        3\sum_{ i = 0 }^1 
        \tfrac{V_i(y_j+\iota \tilde{v}_j\tilde{h}_j)}{ q_{i} } 
      \right).
\end{split}     \end{equation}
In addition,
  Lemma \ref{l:local.Lip.C1}
  (applied for all $t\in[0,T]$, $u_1\in[0,t]$, $u_2\in[u_1,t]$, $y_1,y_2\in O$,
  $\tilde{h}_1,\tilde{h}_2\in\R\setminus\{0\}$,
  $\tilde{v}_1,\tilde{v}_2\in H$ satisfying that
  $y_1+\tilde{h}_1\tilde{v}_1,y_2+\tilde{h}_2\tilde{v}_2\in O$
  with
  $\mathcal{O}\defeq O$,
  $p\defeq 3p(1+\theta)(1+\delta)$,
  $s_1\defeq u_1$,
  $s_2\defeq u_2$,
  $t_1\defeq t$,
  $t_2\defeq t$,
  $x_1\defeq y_1$,
  $x_2\defeq y_2$,
  $v_1\defeq \tilde{v}_1$,
  $v_2\defeq \tilde{v}_2$,
  $h_1\defeq \tilde{h}_1$,
  $h_2\defeq \tilde{h}_2$
  in the notation of Lemma \ref{l:local.Lip.C1})
  yields 
  for all $t\in[0,T]$, $u_1,u_2\in[0,t]$, $y_1,y_2\in O$,
  $\tilde{h}_1,\tilde{h}_2\in\R\setminus\{0\}$,
  $\tilde{v}_1,\tilde{v}_2\in H$ with
  $y_1+\tilde{h}_1\tilde{v}_1,y_2+\tilde{h}_2\tilde{v}_2\in O$ and $u_1\leq u_2$ that
  \begin{equation}  \begin{split}\label{eq:fromLemma44}
    &\Big\|
     D_{u_1,t}^{y_1,\tilde{h}_1}(\tilde{v}_1)-D_{u_2,t}^{y_2,\tilde{h}_2}(\tilde{v}_2)
     \Big\|_{L^{3p(1+\theta)(1+\delta)}(\P;H)}
   \\&
   \leq
  \left(\|y_1-y_2\|_H\sqrt{t-u_2}+\|\tilde{v}_1\tilde{h}_1-\tilde{v}_2\tilde{h}_2\|_H+\|\tilde{v}_1-\tilde{v}_2\|_H+\sqrt{|u_1-u_2|}
  \right)
  \\&\quad
  \cdot
    e^{2\alpha_0 \gamma T}
    \left(\tfrac{12p(1+\theta)^3(1+\delta)^2(1+\gamma)}{\min\{\delta,1\}}
    +2\sqrt{T}
    +\smallint_0^T \tfrac{2\beta_0}{e^{\alpha_0 u}}\,du
    +2\max_{\iota\in\{0,1\},j\in\{1,2\}}
    V_0(y_j+\iota \tilde{v}_j\tilde{h}_j)
  \right)^{2\gamma+2}
  \\&\quad
  \cdot
\|\tilde{v}_1\|_H
  (1+c)^2\max_{\iota\in\{0,1\},j\in\{1,2\}}
  \exp\!\left(
        3\int_{0}^{T}  
        \big[
        \constFun(u)
        +
        \sum_{i=0}^1
        \tfrac{
          \beta_{ i } 
        }{
          q_{i} e^{ \alpha_{ i } u }
        }
        \big]
        \,
        du
        +
        3\sum_{ i = 0 }^1 
        \tfrac{V_i(y_j+\iota \tilde{v}_j\tilde{h}_j)}{ q_{i} } 
      \right).
  \end{split}     \end{equation}
  Next, \eqref{2eq:estimate.zetamue.2},
  \eqref{eq:fromLemma44} (applied for all 
  $r\in[0,T-s_2]$, $\eps\in(0,1)$ with
  $t\defeq s_2+r$,
  $u_1\defeq s_1$,
  $u_2\defeq s_1$,
  $y_1\defeq x_1+h_1w_1$, 
  $y_2\defeq x_1$, 
  $\tilde{h}_1\defeq \eps$,
  $\tilde{h}_2=\eps$,
  $\tilde{v}_1=v_1$,
  $\tilde{v}_2=v_1$
  in the notation of \eqref{eq:fromLemma44}),
  \eqref{eq:fromLemma32} (applied for all $i\in\{1,2\}$ with
  $u_1\defeq s_i$,
  $u_2\defeq s_i$,
  $y_1\defeq x_i+w_ih_i$, 
  $y_2\defeq x_i$ 
  in the notation of \eqref{eq:fromLemma32}),
  \eqref{eq:fromLemma32} (applied for all $\eps\in(0,1)$ with
  $u_1\defeq s_1$,
  $u_2\defeq s_1$,
  $y_1\defeq x_1+w_1h_1+\eps v_1$, 
  $y_2\defeq x_1+w_1h_1$ 
  in the notation of \eqref{eq:fromLemma32}),
  \eqref{eq:fromLemma32} (applied for all $\iota\in\{0,1\}$ with
  $u_1\defeq s_1$,
  $u_2\defeq s_2$,
  $y_1\defeq x_1+\iota w_1h_1$, 
  $y_2\defeq x_2+\iota w_2h_2$ 
  in the notation of \eqref{eq:fromLemma32}),
  \eqref{eq:fromLemma44} (applied with
  $u_1\defeq s_1$,
  $u_2\defeq s_2$,
  $y_1\defeq x_1$, 
  $y_2\defeq x_2$, 
  $\tilde{h}_1\defeq h_1$,
  $\tilde{h}_2\defeq h_2$,
  $\tilde{v}_1\defeq w_1$,
  $\tilde{v}_2\defeq w_2$
  in the notation of \eqref{eq:fromLemma44}),
  \eqref{eq:fromLemma44} (applied for all $\eps\in(0,1)$ with
  $u_1\defeq s_1$,
  $u_2\defeq s_2$,
  $y_1\defeq x_1+w_1h_1$, 
  $y_2\defeq x_2+w_2h_2$, 
  $\tilde{h}_1\defeq \eps$,
  $\tilde{h}_2\defeq \eps$,
  $\tilde{v}_1\defeq v_1$,
  $\tilde{v}_2\defeq v_2$
  in the notation of \eqref{eq:fromLemma44}),
  Lemma~\ref{2l:C2.implies.C0},
  and Lemma \ref{l:exp_mom.moments}
  (applied for all $\iota\in\{0,1\}$, $i\in\{1,2\}$, $r\in[0,T-s_2]$
  with $s\defeq s_i$, $X\defeq X_{s_i,\cdot}^{x_i+\iota w_ih_i}$,
  $\alpha\defeq \alpha_i$, $\beta\defeq \beta_i$, $V\defeq V_i$, $t\defeq s_2+r$,
  $p\defeq \tfrac{p(1+\theta)(1+\delta)\gamma}{\delta}$
  in the notation of Lemma \ref{l:exp_mom.moments})
  yield for all $r\in[0,T-s_2]$ that
  \begin{equation}  \begin{split}\label{2eq:estimate.mue}
   &\Big\|\zeta_r^{\mu}   \Big\|_{L^{p}(\P;H)}
  \leq
  (1+\|w_1\|_H\|v_1\|_H)(1+\max_{i\in\{1,2\}}\|w_i\|_H)(1+\|v_1\|_H)
  \\&\cdot c\left(4\|x_1-x_2\|_H+4\| w_1 h_1- w_2 h_2\|_H+4\sqrt{|s_1-s_2|}
  +\|w_1-w_2\|_H
  +\|v_1-v_2\|_H
  \right)(1+T)
  \\&\cdot
    \Bigg[e^{2\alpha_0 \gamma T}
    \left(\tfrac{12p(1+\theta)^3(1+\delta)^2(1+\gamma)}{\min\{\delta,1\}}
    +2\sqrt{T}
    +\smallint_0^T \tfrac{2\beta_0}{e^{\alpha_0 u}}\,du
    +2\max_{\iota\in\{0,1\},j\in\{1,2\}}
    V_0(x_j+\iota {w}_j{h}_j)
  \right)^{2\gamma+2}
  \\&\quad
  \cdot
  (1+c)^2\max_{\iota\in\{0,1\},j\in\{1,2\}}
  \exp\!\left(
        3\int_{0}^{T}  
        \big[
        \constFun(u)
        +
        \sum_{i=0}^1
        \tfrac{
          \beta_{ i } 
        }{
          q_{i} e^{ \alpha_{ i } u }
        }
        \big]
        \,
        du
        +
        3\sum_{ i = 0 }^1 
        \tfrac{V_i(x_j+\iota  w_jh_j)}{ q_{i} } 
      \right)\Bigg]^{1+1+1+1+1+1}
   \\&\quad\cdot
  \max_{\iota\in\{0,1\},i\in\{1,2\}}
   4^{\gamma}e^{\alpha_i (s_2+r)\gamma}\Big(\tfrac{p(1+\theta)(1+\delta)\gamma}{\delta}+\int_{s_i}^{s_2+r}\tfrac{\beta_i}{e^{\alpha_iu}}\,du
   +V_0({x_i+\iota w_ih_i})\Big)^{\gamma}
  .
  \end{split}     \end{equation}
  An analogous argumentation shows for all $r\in[0,T-s_2]$ that
  \begin{equation}  \begin{split}\label{2eq:estimate.sigma}
   &\Big\|\zeta_r^{\sigma}   \Big\|_{L^{p}(\P;\HS(U,H))}
  \leq
  (1+\max_{i\in\{1,2\}}\|w_i\|_H)^2(1+\|v_1\|_H)^2
  \\&\cdot 4c\left(\|x_1-x_2\|_H+\| w_1 h_1- w_2 h_2\|_H+\sqrt{|s_1-s_2|}
  +\|w_1-w_2\|_H
  +\|v_1-v_2\|_H
  \right)
  \\&\cdot
    \Bigg[e^{2\alpha_0 \gamma T}
    \left(\tfrac{12p(1+\theta)^3(1+\delta)^2(1+\gamma)}{\min\{\delta,1\}}
    +2\sqrt{T}
    +\smallint_0^T \tfrac{2\beta_0}{e^{\alpha_0 u}}\,du
    +2\max_{\iota\in\{0,1\},j\in\{1,2\}}
    V_0(x_j+\iota {w}_j{h}_j)
  \right)^{2\gamma+3}
  \\&\quad
  \cdot
  (1+c)^2\max_{\iota\in\{0,1\},j\in\{1,2\}}
  \exp\!\left(
        3\int_{0}^{T}  
        \constFun(u)
        +
        \sum_{i=0}^1
        \tfrac{
          \beta_{ i } 
        }{
          q_{i} e^{ \alpha_{ i } u }
        }
        \,
        du
        +
        3\sum_{ i = 0 }^1 
        \tfrac{V_i(x_j+\iota {w}_j{h}_j)}{ q_{i} } 
      \right)\Bigg]^{4}
   \\&\quad\cdot
  \max_{\iota\in\{0,1\},i\in\{1,2\}}
   4^{\gamma}e^{\alpha_i (s_2+r)\gamma}\Big(\tfrac{p(1+\theta)(1+\delta)\gamma}{\delta}+\int_{s_i}^{s_2+r}\tfrac{\beta_i}{e^{\alpha_iu}}\,du
   +V_0({x_i+\iota w_ih_i})\Big)^{\gamma}
  .
 \end{split}     \end{equation}
  Next,
  Fatou's lemma and
  \eqref{eq:fromLemma44}
  (applied for all $t\in[s_1,T]$, $\eps\in(0,\infty)$ which satisfy that $x_1+h_1w_1+\eps v_1,x_1+\eps v_1\in O$
  with
  $u_1\defeq s_1$,
  $u_2\defeq s_1$,
  $y_1\defeq x_1+h_1w_1$,
  $y_2\defeq x_1$,
  $\tilde{h}_1\defeq \eps$,
  $\tilde{h}_2\defeq \eps$,
  $\tilde{v}_1\defeq v_1$,
  $\tilde{v}_2\defeq v_1$
  in the notation of
  \eqref{eq:fromLemma44})
  prove for all $t\in[s_1,T]$ that
  \begin{equation}  \begin{split}\label{eq:afterFatou1}
    &\Big\|
     D_{s_1,t}^{x_1,h_1}(v_1,w_1)
     \Big\|_{L^{p}(\P;H)}
    \leq\Big\|
     D_{s_1,t}^{x_1,h_1}(v_1,w_1)
     \Big\|_{L^{3p}(\P;H)}
   \\&
     \leq \liminf_{\{\tilde{\eps}\in(0,\infty)\colon x_1+h_1w_1+\tilde{\eps} v_1,x_1+\tilde{\eps} v_1\in D\}\ni
     \eps\to0}\tfrac{\|D_{s_1,t}^{x_1+h_1w_1,\eps}(v_1)-D_{s_1,t}^{x_1,\eps}(v_1)\|_{L^{3p(1+\theta)(1+\delta)}(\P;H)}}
     {|h_1|}
     \\&
     \leq
  \|w_1\|_H\sqrt{t-s_1}
  \\&\quad
  \cdot
    e^{2\alpha_0 \gamma T}
    \left(\tfrac{12p(1+\theta)^3(1+\delta)^2(1+\gamma)}{\min\{\delta,1\}}
    +2\sqrt{T}
    +\smallint_0^T \tfrac{2\beta_0}{e^{\alpha_0 u}}\,du
    +2\max_{\iota\in\{0,1\}}
    V_0(x_1+\iota h_1 w_1)
  \right)^{2\gamma+2}
  \\&\quad
  \cdot
\|{v}_1\|_H
  (1+c)^2\max_{\iota\in\{0,1\}}
  \exp\!\left(
        3\int_{0}^{T}  
        \big[
        \constFun(u)
        +
        \sum_{i=0}^1
        \tfrac{
          \beta_{ i } 
        }{
          q_{i} e^{ \alpha_{ i } u }
        }
        \big]
        \,
        du
        +
        3\sum_{ i = 0 }^1 
        \tfrac{V_i(x_1+\iota h_1 w_1)}{ q_{i} } 
      \right).
  \end{split}     \end{equation}
  This,
  \eqref{2eq:deltaD},
  \eqref{2eq:estimate.mue},
  and
  \eqref{2eq:estimate.sigma} show that
  \begin{equation}  \begin{split}\label{eq:afterFatou2}
    &\Big\|
     D_{s_1,t_2}^{x_1,h_1}(v_1,w_1)-D_{s_2,t_2}^{x_2,h_2}(v_2,w_2)
     \Big\|_{L^{p}(\P;H)}
    \\&\leq
  (1+\max_{i\in\{1,2\}}\|w_i\|_H)^2(1+\|v_1\|_H)^2(1+4c)
  \\&\cdot \left(\|x_1-x_2\|_H+\| w_1 h_1- w_2 h_2\|_H+\sqrt{|s_1-s_2|}
  +\|w_1-w_2\|_H
  +\|v_1-v_2\|_H
  \right)(1+T)
  \\&\cdot
    \Bigg[e^{2\alpha_0 \gamma T}
    \left(\tfrac{12p(1+\theta)^3(1+\delta)^2(1+\gamma)}{\min\{\delta,1\}}
    +2\sqrt{T}
    +\smallint_0^T \tfrac{2\beta_0}{e^{\alpha_0 u}}\,du
    +2\max_{\iota\in\{0,1\},j\in\{1,2\}}
    V_0(x_j+\iota {w}_j{h}_j)
  \right)^{2\gamma+2}
  \\&\quad
  \cdot
  (1+c)^2\max_{\iota\in\{0,1\},j\in\{1,2\}}
  \exp\!\left(
        3\int_{0}^{T}  
        \big[
        \constFun(u)
        +
        \sum_{i=0}^1
        \tfrac{
          \beta_{ i } 
        }{
          q_{i} e^{ \alpha_{ i } u }
        }
        \big]
        \,
        du
        +
        3\sum_{ i = 0 }^1 
        \tfrac{V_i(x_j+\iota {w}_j{h}_j)}{ q_{i} } 
      \right)\Bigg]^{5}
   \\&\quad\cdot
  \max_{\iota\in\{0,1\},i\in\{1,2\}}
   4^{\gamma}e^{\alpha_i T\gamma}\Big(\tfrac{p(1+\theta)(1+\delta)\gamma}{\delta}+\int_{s_i}^{T}\tfrac{\beta_i}{e^{\alpha_iu}}\,du
   +V_0({x_i+\iota w_ih_i})\Big)^{\gamma}
  .
  \end{split}     \end{equation}
  Next, we derive a temporal regularity estimate.
  Lemma~\ref{2lem:Dvw:eq},
  the triangle inequality,
  the Burkholder-Davis-Gundy type inequality in \cite[Lemma 7.7]{dz92}, 
  H\"older's inequality (applied with $\tfrac{1}{p}=\tfrac{1}{3p}+\tfrac{1}{3p}+\tfrac{1}{3p}$),
  and
  \eqref{2eq:growth.Dmue}
  prove for all $u_1\in[s_1,T]$, $u_2\in[u_1,T]$ that
  \begin{equation}  \begin{split}
    &\left\|D_{s_1,u_2}^{x_1,h_1}(v_1,w_1)-D_{s_1,u_1}^{x_1,h_1}(v_1,w_1)
    \right\|_{L^{p}(\P;H)}
    \\&
    \leq\int_{u_1}^{u_2}\left\|
    \left\|
      \mu' \Big( X_{s_1,r}^{x_1} \Big)
      \right\|_{L(H,H)}
      \|D_{s_1,r}^{x_1,h_1}(v_1,w_1) \|_H
      \right\|_{L^{p}(\P;\R)}
      \,dr
    \\&\quad
    +\left(\tfrac{p(p-1)}{2}\int_{u_1}^{u_2}\left\|
    \left\|
      \sigma' \Big( X_{s_1,r}^{x_1 } \Big)
      \right\|_{L(H,\HS(U,H))}
      \|D_{s_1,r}^{x_1,h_1}(v_1,w_1) \|_H
      \right\|_{L^{p}(\P;\R)}^2\,dr
    \right)^{\frac{1}{2}}
    \\&\quad+
    \int_{u_1}^{u_2}
    \left\|
    \int_0^1
      \mu'' \Big( \lambda X_{s_1,r}^{x_1+w_1h_1}
        +(1-\lambda)X_{s_1,r}^{x_1}\Big)
        \,d\lambda
      \Big(D_{s_1,r}^{x_1,h_1}(w_1),
      \tfrac{\partial}{\partial x}X_{s_1,r}^{x_1+w_1h_1}(v_1)
      \Big)
      \right\|_{L^{p}(\P;H)}
      \,dr
    \\&\quad
    +\bigg(\tfrac{p(p-1)}{2}\int_{u_1}^{u_2}\Big\|
    \int_0^1
      \sigma'' \Big( \lambda X_{s_1,r}^{x_1+w_1h_1}
        +(1-\lambda)X_{s_1,r}^{x_1}\Big)
        \,d\lambda
     \\&\qquad\qquad\qquad\qquad
     \cdot
      \Big(D_{s_1,r}^{x_1,h_1}(w_1),
      \tfrac{\partial}{\partial x}X_{s_1,r}^{x_1+w_1h_1}(v_1)
      \Big)
      \Big\|_{L^{p}(\P;\HS(U,H))}^2\,dr
    \bigg)^{\frac{1}{2}}
    \\&
\leq
    \sup_{r\in[s_1,T]}
    \left\|
    c\left(2+2V_0(X_{s_1,r}^{x_1+w_1h_1})+2V_0(X_{s_1,r}^{x_1})\right)^{\gamma}
    \right\|_{L^{3p}(\P;\R)}
    \sqrt{u_2-u_1}
    \cdot\Big(\sqrt{T}+\sqrt{\tfrac{p(p-1)}{2}}\Big)
\\&\quad
    \cdot\Big(
    \sup_{r\in[s_1,T]}
   \left\|
      D_{s_1,r}^{x_1,h_1}(v_1,w_1)
    \right\|_{L^{3p}(\P;H)}
    +
    \sup_{r\in[s_1,T]}
    \big(
   \left\|
      D_{s_1,r}^{x_1,h_1}(w_1)
    \right\|_{L^{3p}(\P;H)}
   \left\|
      \tfrac{\partial}{\partial x}X_{s_1,r}^{x_1+w_1h_1}(v_1)
    \right\|_{L^{3p}(\P;H)}
    \big)
    \Big).
  \end{split}     \end{equation}
  This, the triangle inequality,
  Lemma~\ref{l:exp_mom.moments},
  \eqref{eq:afterFatou1},
  and
  \eqref{eq:fromLemma32}
  yield that
  \begin{equation}  \begin{split}
    &\left\|D_{s_1,t_1}^{x_1,h_1}(v_1,w_1)-D_{s_1,t_2}^{x_1,h_1}(v_1,w_1)
    \right\|_{L^{p}(\P;H)}
    \\&
    \leq \sqrt{|t_1-t_2|}\big(\sqrt{T}+p\big)
    c2^\gamma e^{\alpha_0 T\gamma}
    \Big(6p\gamma+\int_{s_1}^T\tfrac{2\beta_0}{e^{\alpha_0r}}\,dr
    +\sum_{\iota\in\{0,1\}}V_0(x_1+\iota w_1h_1)\Big)^{\gamma}
  \\&\quad
  \cdot
  \Big(
\|{v}_1\|_H
  \|w_1\|_H\sqrt{T-s_1}
  +
\|v_1\|_H \|w_1\|_H
 \Big)
  \\&\quad
  \cdot
  \Biggl[
    e^{2\alpha_0 \gamma T}
    \left(\tfrac{12p(1+\theta)^3(1+\delta)^2(1+\gamma)}{\min\{\delta,1\}}
    +2\sqrt{T}
    +\smallint_0^T \tfrac{2\beta_0}{e^{\alpha_0 u}}\,du
    +2\max_{\iota\in\{0,1\}}
    V_0(x_1+\iota h_1 w_1)
  \right)^{2\gamma+2}
  \\&\quad
  \cdot
  (1+c)^2\max_{\iota\in\{0,1\}}
  \exp\!\left(
        3\int_{0}^{T}  
        \big[
        \constFun(u)
        +
        \sum_{i=0}^1
        \tfrac{
          \beta_{ i } 
        }{
          q_{i} e^{ \alpha_{ i } u }
        }
        \big]
        \,
        du
        +
        3\sum_{ i = 0 }^1 
        \tfrac{V_i(x_1+\iota h_1 w_1)}{ q_{i} } 
      \right)
    \Biggr]^2.
  \end{split}     \end{equation}
  This, the triangle inequality,
  and \eqref{eq:afterFatou2}
  yield that
  \begin{equation}  \begin{split}
    &\Big\|
     D_{s_1,t_1}^{x_1,h_1}(v_1,w_1)
     -D_{s_2,t_2}^{x_2,h_2}(v_2,w_2)
     \Big\|_{L^{p}(\P;H)}
    \\&\leq
    \Big\|
     D_{s_1,t_1}^{x_1,h_1}(v_1,w_1)
     -D_{s_1,t_2}^{x_1,h_1}(v_1,w_1)
     \Big\|_{L^{p}(\P;H)}
     +
    \Big\|
     D_{s_1,t_2}^{x_1,h_1}(v_1,w_1)
     -D_{s_2,t_2}^{x_2,h_2}(v_2,w_2)
     \Big\|_{L^{p}(\P;H)}
    \\&\leq
  (1+\max_{i\in\{1,2\}}\|w_i\|_H)^2(1+\|v_1\|_H)^2(1+4c)(p^2+T)
  \\&\cdot \left(\|x_1-x_2\|_H+\| w_1 h_1- w_2 h_2\|_H
  +\sqrt{|s_1-s_2|}
  +\sqrt{|t_1-t_2|}
  +\|w_1-w_2\|_H
  +\|v_1-v_2\|_H
  \right)
  \\&\cdot
    \Bigg[e^{2\alpha_0 \gamma T}
    \left(\tfrac{12p(1+\theta)^3(1+\delta)^2(1+\gamma)}{\min\{\delta,1\}}
    +2\sqrt{T}
    +\smallint_0^T \tfrac{2\beta_0}{e^{\alpha_0 u}}\,du
    +2\max_{\iota\in\{0,1\},j\in\{1,2\}}
    V_0(x_j+\iota {w}_j{h}_j)
  \right)^{2\gamma+2}
  \\&\quad
  \cdot
  (1+c)^2\max_{\iota\in\{0,1\},j\in\{1,2\}}
  \exp\!\left(
        3\int_{0}^{T}  
        \big[
        \constFun(u)
        +
        \sum_{i=0}^1
        \tfrac{
          \beta_{ i } 
        }{
          q_{i} e^{ \alpha_{ i } u }
        }
        \big]
        \,
        du
        +
        3\sum_{ i = 0 }^1 
        \tfrac{V_i(x_j+\iota {w}_j{h}_j)}{ q_{i} } 
      \right)\Bigg]^{5}
   \\&\quad\cdot
  \max_{\iota\in\{0,1\},i\in\{1,2\}}
   4^{\gamma}e^{\alpha_i T\gamma}\Big(\tfrac{p(1+\theta)(1+\delta)\gamma}{\delta}+\int_{s_i}^{T}\tfrac{\beta_i}{e^{\alpha_iu}}\,du
   +V_0({x_i+\iota w_ih_i})\Big)^{\gamma}
  .
  \end{split}     \end{equation}
  This completes the proof of Lemma~\ref{2l:local.Lip.C2}.
\end{proof}
\subsection{Existence of a $C^2$-solution}
The following theorem establishes existence of twice
continuously
differentiable
solutions of SDEs.

\begin{theorem}[Existence of a $C^2$-solution] \label{thm:C2}
Let $( H, \left< \cdot , \cdot \right>_H, \left\| \cdot \right\|_H )$ and $( U, \left< \cdot , \cdot \right>_U, \left\| \cdot \right\|_U )$
be separable $\R$-Hilbert spaces,
assume that $\dim(H)<\infty$,
let
$ T \in (0,\infty) $,
let $(\Omega, \F, \P, (\mathbb{F}_{t})_{t\in [0,T]})$ be a filtered probability space
satisfying the usual conditions,
let 
$
  (W_t)_{t\in[0,T]}
$
  be an $\textup{Id}_U$-cylindrical $(\mathbb{F}_t)_{t\in[0,T]}$-Wiener process,
let $\Delta_T =\{(s,t)\in [0,T]^2 \colon s \leq t\}$,
let $O \subseteq H$ be an open and convex set,
let $ \mu \in C^2(O,H)$, 
$ \sigma \in C^2(O, \HS(U,H))$,
for all $s\in [0,T]$, $x\in O$
let $ X^x_{s,\cdot} \colon [s,T] \times \Omega \to O $
be an $(\mathbb{F}_t)_{t\in[s,T]}$-adapted stochastic process
with continuous sample paths
which satisfies that for all $t\in [s,T]$ it holds a.s.~that
\begin{align} \label{thm2eq:def:X}
  X^x_{s,t} = 
  x
  + \int_s^{ t } \mu(X^x_{s,r} ) \, dr
  +
  \int_s^{ t } \sigma(X^x_{s,r} ) \, dW_r,
\end{align}
let $ \alpha_0,\alpha_1,\beta_0,\beta_1 ,c\in [0,\infty)$,
$ 
  V_0 
,
  V_1  \in C^{ 2 }( O , [0,\infty) ) 
$,
let
$ 
  \bar{V} \colon [0,T] \times O \to [0,\infty)
$
be a measurable function which satisfies
for all
$ i \in \{ 0, 1 \} $,
$t\in[0,T]$,
$x\in O$
that 
$\P\big(	\int_t^T | \bar{V}(r, X_{t,r}^x) | \,dr <\infty\big)=1$ and
\begin{equation}
\begin{split}
  &\Big\langle
  \mu( x )
  ,
  (\nabla V_i)(x)
  \Big\rangle_H
  +
  \tfrac{ 1 }{ 2 }
  \operatorname{trace}\!\Big(
    \sigma(x) [\sigma(x)]^* 
    ( \operatorname{Hess} V_i )( x )
  \Big)
  \\&
  +
  \tfrac{ 
    1
  }{ 
    2 
    e^{ 
      \alpha_{ i } {t} 
    }
  }
    \|
      \sigma( x )^* ( \nabla V_{ i } )( x )
    \|_U^2
  +
  \mathbbm{1}_{
    \{ 1 \}
  }(i)
  \cdot
  \bar{V}(t,x)
\leq
  \alpha_{ i } V_{ i }(x)
  +
  \beta_{ i },
\end{split}
\end{equation}
let 
$
  \constFun \colon [0,T] \to [0,\infty)
$
be a measurable function which satisfies that
$\int_0^T \constFun(r) \, dr <\infty$,
let $p\in(2\dim(H)+6,\infty)$, $\theta\in[0,\infty)$, $\delta\in(0,\infty)$, $q_0,q_1\in(0,\infty]$
satisfy that $\tfrac{\theta}{6p(1+\theta)^3(1+\delta)^2}=\tfrac{1}{q_0}+\tfrac{1}{q_1}$,
assume that for all $t\in[0,T]$, $x,y\in O$, $v\in H\setminus\{0\}$
it holds that
\begin{equation}  \begin{split}\label{thm2eq:Lip.ass.strong}
  &\Big\langle v,\smallint_0^1\mu'(\lambda x+(1-\lambda)y)+\delta\,d\lambda \,\,v
  \Big\rangle_H
  +\tfrac{1+\delta}{2}
  \Big\|\smallint_0^1\sigma'(\lambda x+(1-\lambda)y)\,d\lambda \,\,v
  \Big\|_{\HS(U,H)}^2
  \\&
  +
  \tfrac{(3p(1+\theta)^3(1+\delta)^2-1)\big\|
  \big\langle v,\smallint_0^1\sigma'(\lambda x+(1-\lambda)y)\,d\lambda\,\, v
  \big\rangle_H\big\|_{\HS(U,\R)}^2}{\|v\|_H^2}
  \leq \|v\|^2_H\cdot\Big(\phi(t)
  +\tfrac{V_0(x)+V_0(y)}{2q_0T e^{\alpha_0t}}
  +\tfrac{\bar{V}(t,x)+\bar{V}(t,y)}{2q_1e^{\alpha_1t}}
  \Big),
\end{split}     \end{equation}
let $\gamma\in[\tfrac{1}{p},\infty)$ satisfy that for all $x\in O$ it holds that
\begin{align}
\begin{split}
 \max\left\{ \|\mu(x)\|_H , \|\sigma(x)\|_{\HS(U,H)} \right\}
 \leq  c(1+V_0(x))^{\gamma},
\end{split}
\end{align}
satisfy that for all $x,y\in O$, $i\in\{1,2\}$ it holds that
\begin{align} \label{thm2eq:growth.Dmue}
\begin{split}
 &\max\left\{
    \left\|\smallint_0^1D^i\mu\big(\lambda x+(1-\lambda)y\big)\,d\lambda 
    \right\|_{L^{(i)}(H,H)} ,
    \left\|\smallint_0^1D^i\sigma\big(\lambda x+(1-\lambda)y\big)\,d\lambda 
    \right\|_{L^{(i)}(H,\HS(U,H))}
 \right\}
 \\&
 \leq  c\left(2+V_0(x)+V_0(y)\right)^{\gamma},
\end{split}
\end{align}
and satisfy that for all $x_1,x_2,x_3,x_4\in O$,
$i\in\{1,2\}$
it holds that
\begin{equation}  \begin{split}\label{thm2eq:Dmu.localLip2}
  &\max\Big\{\big\|\smallint_0^1 
  D^i\mu(\lambda x_1+(1-\lambda)x_2)-D^i\mu(\lambda x_3+(1-\lambda)x_4)
  \,d\lambda
  \big\|_{L^{(i)}(H,H)},
  \\&\qquad\quad
  \big\|\smallint_0^1
     D^i\sigma(\lambda x_1+(1-\lambda)x_2)
        -D^i\sigma(\lambda x_3+(1-\lambda)x_4)
   \,d\lambda
  \big\|_{L^{(i)}(H,\HS(U,H))}
  \Big\}
  \\&
  \leq c\smallint_0^1\lambda\|x_1-x_3\|_H+(1-\lambda)\|x_2-x_4)\|_H
  \,d\lambda\,
  \big(4+\smallsum_{j=1}^4 V_0(x_i)\big)^{\gamma}.
\end{split}     \end{equation}
Then there exists a measurable
function
$
  \mathcal{X} \colon \Delta_T \times {O} \times \Omega \to \overline{O}
$
such that
\begin{enumerate}[(i)]
  \item  
for all 
$
  x \in O, 
$ 
$ 
	s\in [0,T]
$ 
it holds a.s.~that 
$
  (\mathcal{X}^x_{s, t })_{t \in [s,T]} 
  = 
  (X^x_{s,t})_{t \in [s,T]}
$,
and
\item
for every $ \omega \in \Omega $
it holds
that
$
  \mathcal{X}(\omega) \in C^{0,2}( \Delta_T \times {O}, \overline{O})
$.
\end{enumerate}
\end{theorem}
\begin{proof}[Proof of Theorem~\ref{thm:C2}]
Without loss of generality we additionally assume throughout this proof
that $\dim(H)\geq 1$ and that $O\neq\emptyset$.
Throughout this proof
let $d\in\N$, $\mathbb{H}\subseteq H$ satisfy that $d=\dim(H)$
and that $\mathbb{H}$ is an orthonormal basis of $H$,
let $O^{\R}$, $O_0^{\R}$
be the sets which satisfy that
$O^{\R}=\cap_{v\in\mathbb{H}}
\{(x,h)\in
O\times\R\colon x+vh\in O\}$
and $O_0^{\R}=\{(x,h)\in O^{\R}\colon h\neq 0\}$,
and
let $K_n\subseteq \Delta_T\times H\times \R$,
$n\in \N$,
be the sets which satisfy for all $n\in \N$ that
$K_n=\{(s,t,x,h)\in \Delta_T\times O_0^{\R} \colon s^2+t^2+\|x\|_H^2+h^2 \leq n^2, \inf(\{\|x-y\|\colon y\in O^c\}\cup\{2\})\leq\frac{1}{n} \}$.
 The fact that $p>2d+6$
 and Theorem~\ref{thm:C1}
 yield that there exists a
  measurable
  function $\tilde{\mathcal{X}}\colon\Delta_T\times O\times\Omega\to\overline{O}$
  such that
for all 
$
  x \in O, 
$ 
$ 
	s\in [0,T]
$ 
it holds a.s.~that 
$
  (\tilde{\mathcal{X}}^x_{s, t })_{t \in [s,T]} 
  = 
  (X^x_{s,t})_{t \in [s,T]}
$, 
and such that
for every $ \omega \in \Omega $
it holds
that
$
  \tilde{\mathcal{X}}(\omega) \in C^{0,1}( \Delta_T \times O, \overline{O})
$.
For the rest of this proof, for all $(s,t)\in\Delta_T$, $x\in O$, $v,w\in H$, $h\in\R\setminus\{0\}$ with $x+hw \in O$
let $D_{s,t}^{x,h}(v,w)\colon\Omega\to H$ be the function which satisfies that
\begin{equation}  \begin{split}
  D_{s,t}^{x,h}(v,w)&= 
  \frac{\frac{\partial}{\partial x}\tilde{\mathcal{X}}_{s,t}^{x+h w}(v) -\frac{\partial}{\partial x}\tilde{\mathcal{X}}_{s,t}^{x}(v)}{h}.
\end{split}     \end{equation}
Then
Lemma~\ref{2l:local.Lip.C2}, the fact that $\int_0^T \constFun(r) \, dr <\infty$,
and
boundedness of the functions $V_0,V_1$ on each of the 
relatively compact
subsets $\{x\in O\colon \exists s,t,h\in\R \text{ s.t.\ }(s,t,x,h)\in K_n\}\subseteq O$,
$n\in\N$,
demonstrate for all $n\in\N$, $v,w\in\mathbb{H}$ that
\begin{equation}\label{thmeq:C0.ass.Hoelder2} 
  \sup\bigg(\bigg\{
  \tfrac{ 
  \big(\E\big[  \big\|D_{s_1,t_1}^{x_1,h_1}(v,w) - D_{s_2,t_2}^{x_2,h_2}(v,w)   \big\|^{p}_{H}\big]\big)^{\frac{1}{p}}
  }
  {\left( |s_1-s_2|^2+|t_1-t_2|^2+\|x_1-x_2\|_H^{2}+|h_1-h_2|^2
  \right)^{\frac{1}{4}}}
  \colon
  \substack{ (s_1,t_1,x_1,h_1), (s_2,t_2,x_2,h_2) \in K_n \colon
  \\(s_1,t_1,x_1,h_1)\neq (s_2,t_2,x_2,h_2)}
  \bigg\}\cup\{0\}\bigg)
  <\infty.
\end{equation}
In particular this implies for all $n\in\N$, $v,w\in\mathbb{H}$ that
\begin{equation}  \begin{split}
  &\sup
  \Big(\Big\{
  \Big(\E\Big[  \|D_{s,t}^{x,h}(v,w)   \|^{p}_{H}\Big]\Big)^{\frac{1}{p}}
  \colon
  (s,t,x,h) \in  K_n\Big\}\cup\{0\}\Big)
  \\&
  \leq
  \sup
  \bigg(\bigg\{
  \tfrac{ 
  \big(\E\big[  \|D_{s,t}^{x,h}(v,w) - D_{s,s}^{x,h}(v,w)   \|^{p}_{H}\big]
  \big)^{\frac{1}{p}}
  }
  {\left( |s-s|^2+|t-s|^2+\|x-x\|_H^{2}+|h-h|^2  \right)^{\frac{1}{4}}}
  \sqrt{T}
  \colon
  (s,t,x,h) \in  K_n,t\neq s\bigg\}\cup\{0\}\bigg)
  <\infty.
\end{split}     \end{equation}
This, \eqref{thmeq:C0.ass.Hoelder2},
Proposition~\ref{p:KolChen}
(applied for every $v,v\in\mathbb{H}$ with 
$H \defeq  \R\times\R\times H\times \R$,
$D\defeq \Delta_T\times O_0^{\R}$,
$E\defeq H$,
$F\defeq H$,
$p\defeq p$,
$\alpha\defeq \nicefrac12$,
$X\defeq \left( \Delta_T\times O_0^{\R} \ni (s,t,x,h)
\mapsto D_{s,t}^{x,h}(v,w) \in H \right)$
in the notation of Proposition~\ref{p:KolChen}),
and path continuity of $D_{s,\cdot}^{x,h}(v,w)$, 
$(s,x,h,v,w)\in[0,T]\times O_0^{\R}\times\mathbb{H}\times\mathbb{H}$,
establish for all $v,w\in\mathbb{H}$ the existence of a
measurable function
$
  \mathcal{D}^{v,w} \colon \overline{\Delta_T\times O_0^{\R}}
  \times \Omega \to H
$
which satisfies that for all $\omega \in \Omega$ it holds that 
$\mathcal{D}^{v,w}(\omega) \in C(\overline{\Delta_T\times O_0^{\R}} , H)$
and which satisfies
that
for all $(s,t,x,h)\in\Delta_T\times O_0^{\R}$
it holds a.s.~that
$(\mathcal{D}_{s,t}^{v,w}(x,h))_{t\in[s,T]}=(D_{s,t}^{x,h}(v,w))_{t\in[s,T]}$.
Note that $O^{\R}\subseteq\overline{O_0^{\R}}$.
Let $\mathcal{D}\colon \Delta_T\times
O^{\R}\times H\times H\times\Omega
\to H$ be the function which satisfies that for all
$(s,t,x,h)\in\Delta_T\times O^{\R}$, $v,w\in H$
it holds that
$\mathcal{D}_{s,t}(x,h,v,w)=\sum_{e,\tilde{e}\in\mathbb{H}}\langle v,e\rangle_H\langle w,\tilde{e}\rangle_H
\mathcal{D}_{s,t}^{e,\tilde{e}}(x,h)$.
Next, we observe that for all $(s,x,h,v,w)\in[0,T]\times O_0^{\R}\times \mathbb{H}\times \mathbb{H}$
 it holds a.s.\ for all $t\in[s,T]$ that
\begin{equation}  \begin{split}
  \mathcal{D}_{s,t}(x,h,v,w)=
\mathcal{D}_{s,t}^{v,w}(x,h)=
D_{s,t}^{x,h}(v,w)=
  \frac{\frac{\partial}{\partial x}\tilde{\mathcal{X}}_{s,t}^{x+h w}(v) -\frac{\partial}{\partial x}\tilde{\mathcal{X}}_{s,t}^{x}(v)}{h}
.
\end{split}     \end{equation}
This, continuity of the random fields $\tfrac{\partial}{\partial x}\tilde{\mathcal{X}}$, $\mathcal{D}$,
and Lemma~\ref{lem:gradient} 
(applied for all $v\in\mathbb{H}$ with $U\defeq H$,
$T\defeq \Delta_T$,
$\mathbb{T}\defeq \Delta_T\cap\Q^2$,
$\mathcal{X}\defeq (\Delta_T\times O\times\Omega\ni(s,t,x,\omega)\mapsto \tfrac{\partial}{\partial x}\tilde{\mathcal{X}}_{s,t}^x(v,\omega)\in H)$,
$\mathcal{Z}\defeq (\Delta_T\times O^{\R}\times H\times\Omega\ni(s,t,x,h,w)\mapsto \mathcal{D}_{s,t}(x,h,v,w,\omega)\in H)$
in the notation of Lemma~\ref{lem:gradient})
prove that there exists $\Omega_0\in\mathcal{F}$
such that $\P(\Omega_0)=1$ and such that for all $\omega\in\Omega_0$,
$(s,t)\in\Delta_T$, $v\in\mathbb{H}$ it holds that
the mapping $O\ni x\mapsto \tilde{\mathcal{X}}_{s,t}^{x}(v,\omega)\in H$
is continuously differentiable and it holds for all $x\in O$, $v\in\mathbb{H}$, $w\in H$
that
\begin{equation}  \begin{split}
\tfrac{\partial}{\partial x}\big(\tfrac{\partial}{\partial x}\mathcal{\tilde{X}}_{s,t}^x(v,\omega)\big)(w)
=\sum_{h\in\mathbb{H}}\langle w,h\rangle_H \mathcal{D}_{s,t}(x,0,v,h,\omega)
=\mathcal{D}_{s,t}(x,0,v,w,\omega).
\end{split}     \end{equation}
This and continuity of $\mathcal{D}$ prove that for all $\omega\in\Omega_0$
it holds that $\tilde{\mathcal{X}}(\omega)\in C^{0,2}(\Delta_T\times O,\overline{O})$.
Let $\mathcal{X}\colon\Delta_T\times{O}\times\Omega
\to \overline{O}$ be the function which satisfies for all $(s,t,x,\omega)\in\Delta_T\times{O}\times\Omega$ that
$\mathcal{X}_{s,t}^x(\omega)=\mathbbm{1}_{\Omega_0}(\omega)\tilde{\mathcal{X}}_{s,t}^x(\omega)$.
Then it holds
that $\mathcal{X}$ is
measurable,
that
for all 
$
  x \in O 
$,
$ 
	s\in [0,T]
$ 
it holds a.s.~that 
$
  (\mathcal{X}^x_{s, t })_{t \in [s,T]} 
  = 
  (\tilde{\mathcal{X}}^x_{s, t })_{t \in [s,T]} 
  =
  (X^x_{s,t})_{t \in [s,T]}
$, 
and
that
for every $ \omega \in \Omega $
it holds
that
$
  \mathcal{X}(\omega) \in C^{0,2}( \Delta_T \times {O}, {\overline{O}})
$.
 This proves items (i) -- (iii)
 and finishes the proof of Theorem~\ref{thm:C2}.
\end{proof}
The following corollary simplifies the assumptions
of Theorems \ref{thm:C2}.
\begin{cor}[Existence of a $C^2$-solution] \label{c:C2}
  Let $d,m\in\N$,
  $ \alpha,\beta \in [0,\infty) $, $T,c\in(0,\infty)$, $p\in(6(d+3)(1+1/c)^3,\infty)$,
  let $\|\cdot\|$, $\langle,\rangle$ denote the standard norm and the standard scalar product on $\R^d$,
  let $\|\cdot\|_{\R^m}$ denote the standard norm on $\R^m$,
  let $\|\cdot\|_{\textup{F}}$ denote the Frobenius norm on $\R^{d\times m}$,
  let $(\Omega, \F, \P, (\mathbb{F}_{t})_{t\in [0,T]})$ be a filtered probability space
  satisfying the usual conditions,
  let 
  $
    W\colon[0,T]\times\Omega\to\R^m
  $
  be a standard $(\mathbb{F}_t)_{t\in[0,T]}$-Wiener process,
  let $ \mu \in C^3(\R^d,\R^d)$, 
  $ \sigma \in C^3(\R^d, \R^{d\times m})$,
  $ 
    V  \in C^{ 2 }( \R^d , [0,\infty) ) 
  $,
  $ \bar{V}\in C(\R^d,[0,\infty) )$,
  assume that $\mu'''$ and $\sigma'''$ grow at most polynomially at infinity,
  assume that $\exists \gamma\in(0,\infty)\colon \sup_{x\in\R^d}\|x\|^{\gamma}/(1+V(x))<\infty$,
  and assume for all $x,y,v\in \R^d$ that
  {\small
  \begin{equation}\begin{split}\label{eq:C2.conditions}
    &\Big\langle v,\smallint_0^1\mu'\big(\lambda (x-y)+y\big)\,d\lambda \,\,v
    \Big\rangle
    +\tfrac{p-1}{2}
    \Big\|\smallint_0^1\sigma'\big(\lambda (x-y)+y\big)\,d\lambda \,\,v
    \Big\|_{\textup{F}}^2
  \leq \|v\|^2\cdot\Big(c
    +\tfrac{V(x)+V(y)}{4c p T e^{\alpha T}}
    +\tfrac{\bar{V}(x)+\bar{V}(y)}{4c p e^{\alpha T}}
    \Big),
\\
    &\Big\langle
  \mu( x )
  ,
  (\nabla V)(x)
  \Big\rangle
  +
  \tfrac{ 1 }{ 2 }
  \operatorname{trace}\!\Big(
    \sigma(x) [\sigma(x)]^* 
    ( \operatorname{Hess} V )( x )
  \Big)
  +
  \tfrac{ 
    1
  }{ 
    2 
  }
    \|
      \sigma( x )^* ( \nabla V )( x )
    \|_{\R^m}^2
  +
  \bar{V}(x)
\leq
  \alpha V(x)
  +
  \beta.
\end{split}
\end{equation}}
Then there exists a measurable
function
$
  X \colon \{(s,t)\in[0,T]\colon s\leq t\} \times \R^d \times \Omega \to \R^d
$
such that
\begin{enumerate}[(i)]
\item
for every $ \omega \in \Omega $
it holds
that
$
  X(\omega) \in C^{0,2}( \{(s,t)\in[0,T]\colon s\leq t\} \times \R^d, \R^d)
$ and
  \item  
for all 
$
  x \in \R^d 
$,
$ 
	s\in [0,T]
$,
$t\in[s,T]$
it holds a.s.~that 
\begin{equation}  \begin{split}
  X^x_{s,t} = 
  x
  + \int_s^{ t } \mu(X^x_{s,r} ) \, dr
  +
  \int_s^{ t } \sigma(X^x_{s,r} ) \, dW_r.
\end{split}     \end{equation}
\end{enumerate}
\end{cor}
\begin{proof}
  The assumption that
  $\mu'''$ and $\sigma'''$ grow at most polynomially at infinity
  and the assumption
  $\exists \gamma\in(0,\infty)\colon \sup_{x\in\R^d}\|x\|^{\gamma}/(1+V(x))<\infty$
  ensure that there exist $\gamma,\tilde{c}\in[1,\infty)$ such that for all $x\in\R^d$ it holds that
  \begin{align}
  \begin{split}
   \max\left\{ \|\mu(x)\| , \|\sigma(x)\|_{\textup{F}} \right\}
   \leq  \tilde{c}(1+V(x))^{\gamma},
  \end{split}
  \end{align}
  such that for all $x,y\in \R^d$, $i\in\{1,2\}$ it holds that
  \begin{align} \label{c2eq:growth.Dmue}
  \begin{split}
   &\max\left\{
      \left\|\smallint_0^1D^i\mu\big(\lambda x+(1-\lambda)y\big)\,d\lambda 
      \right\|_{L^{(i)}(\R^d,\R^d)} ,
      \left\|\Big\|\smallint_0^1D^i\sigma\big(\lambda x+(1-\lambda)y\big)\,d\lambda 
      \Big\|_{\textup{F}}\right\|_{L^{(i)}(\R^d,\R)}
   \right\}
   \\&
   \leq  \tilde{c}\left(2+V(x)+V(y)\right)^{\gamma},
  \end{split}
  \end{align}
  and such that for all $x_1,x_2,x_3,x_4\in \R^d$,
  $i\in\{1,2\}$
  it holds that
  \begin{equation}  \begin{split}\label{c2eq:Dmu.localLip2}
    &\max\Big\{\big\|\smallint_0^1 
    D^i\mu(\lambda x_1+(1-\lambda)x_2)-D^i\mu(\lambda x_3+(1-\lambda)x_4)
    \,d\lambda
    \big\|_{L^{(i)}(\R^d,\R^d)},
    \\&\qquad\quad
    \Big\|\big\|\smallint_0^1
       D^i\sigma(\lambda x_1+(1-\lambda)x_2)
          -D^i\sigma(\lambda x_3+(1-\lambda)x_4)
     \,d\lambda
    \big\|_{\textup{F}}\Big\|_{L^{(i)}(\R^d,\R)}
    \Big\}
    \\&
    \leq \tilde{c}\smallint_0^1\lambda\|x_1-x_3\|+(1-\lambda)\|x_2-x_4)\|
    \,d\lambda\,
    \big(4+\smallsum_{j=1}^4 V(x_i)\big)^{\gamma}.
  \end{split}     \end{equation}
  Moreover, let $\delta\in(0,p)$, $\tilde{p}\in(2d+6,\infty)$
  satisfy that $1+\delta+6\tilde{p}(1+\frac{1}{c})^3(1+\delta)^2-2=p-1$.
  In addition, local Lipschitz continuity of $\mu$, $\sigma$, and 
  the Lyapunov condition in \eqref{eq:C2.conditions}
  yield existence of a global solution of the SDE with drift coefficient $\mu$
  and diffusion coefficient $\sigma$; cf., e.g., \cite{GyoengyKrylov1996}.
  The assertion follows then from Theorem \ref{thm:C2}
  (applied with
   $H\defeq \R^d$,
   $U\defeq \R^m$,
   $O\defeq \R^d$,
   $\alpha_0\defeq\alpha$,
   $\alpha_1\defeq\alpha$,
   $\beta_0\defeq \beta$,
   $\beta_1\defeq \beta$,
   $c\defeq\tilde{c}$,
   $V_0\defeq V$,
   $V_1\defeq V$,
   $\phi\defeq ([0,T]\ni t\mapsto c\in[0,\infty))$,
   $p\defeq\tilde{p}$,
   $\theta\defeq 1/c$,
   $q_0\defeq 2c(p-\delta)$,
   $q_1\defeq 2c(p-\delta)$
   in the notation of Theorem \ref{thm:C2}).
   This completes the proof of Corollary~\ref{c:C2}.
\end{proof}

\subsection*{Acknowledgement}
This project has been partially supported by the Deutsche Forschungsgesellschaft (DFG) via RTG 2131 \textit{High-dimensional Phenomena in Probability -- Fluctuations and Discontinuity}.

\addcontentsline{toc}{section}{Bibliography}

\begin{thebibliography}{18}
\providecommand{\natexlab}[1]{#1}
\providecommand{\url}[1]{\texttt{#1}}
\expandafter\ifx\csname urlstyle\endcsname\relax
  \providecommand{\doi}[1]{doi: #1}\else
  \providecommand{\doi}{doi: \begingroup \urlstyle{rm}\Url}\fi

\bibitem[Attanasio(2010)]{Attanasio2010}
Stefano Attanasio.
\newblock Stochastic flows of diffeomorphisms for one-dimensional {SDE} with
  discontinuous drift.
\newblock \emph{Electronic Communications in Probability}, 15:\penalty0
  213--226, 2010.

\bibitem[Cox et~al.(2013)Cox, Hutzenthaler, and
  Jentzen]{CoxHutzenthalerJentzen2013v3}
Sonja~G. Cox, Martin Hutzenthaler, and Arnulf Jentzen.
\newblock Local {L}ipschitz continuity in the initial value and strong
  completeness for nonlinear stochastic differential equations.
\newblock \emph{Mem. Amer. Math. Soc. \textup{(accepted, arXiv:1309.5595v3)}},
  pages 1--94, 2013.

\bibitem[{Da Prato} and Zabczyk(1992)]{dz92}
Giuseppe {Da Prato} and Jerzy Zabczyk.
\newblock \emph{Stochastic equations in infinite dimensions}, volume~44 of
  \emph{Encyclopedia of Mathematics and its Applications}.
\newblock Cambridge University Press, Cambridge, 1992.
\newblock ISBN 0-521-38529-6.

\bibitem[del Moral and Singh(2019)]{DelMoralSingh2019}
Pierre del Moral and Sumeetpal~Sidhu Singh.
\newblock \emph{Backward {I}t{\^o}-{V}entzell and stochastic interpolation
  formulae}.
\newblock PhD thesis, INRIA, 2019.

\bibitem[Fang and Zhang(2005)]{FangZhang2005}
Shizan Fang and Tusheng Zhang.
\newblock A study of a class of stochastic differential equations with
  non-{L}ipschitzian coefficients.
\newblock \emph{Probab. Theory Related Fields}, 132\penalty0 (3):\penalty0
  356--390, 2005.
\newblock ISSN 0178-8051.
\newblock \doi{10.1007/s00440-004-0398-z}.
\newblock URL \url{http://dx.doi.org/10.1007/s00440-004-0398-z}.

\bibitem[Fang et~al.(2007)Fang, Imkeller, and Zhang]{FangImkellerZhang2007}
Shizan Fang, Peter Imkeller, and Tusheng Zhang.
\newblock Global flows for stochastic differential equations without global
  {L}ipschitz conditions.
\newblock \emph{Ann. Probab.}, 35\penalty0 (1):\penalty0 180--205, 2007.
\newblock ISSN 0091-1798.
\newblock \doi{10.1214/009117906000000412}.
\newblock URL \url{http://dx.doi.org/10.1214/009117906000000412}.

\bibitem[Flandoli et~al.(2010)Flandoli, Gubinelli, and
  Priola]{FlandoliGubinelliPriola2010}
Franco Flandoli, Massimiliano Gubinelli, and Enrico Priola.
\newblock {Flow of diffeomorphisms for SDEs with unbounded H{\"o}lder
  continuous drift}.
\newblock \emph{Bulletin des sciences mathematiques}, 134\penalty0
  (4):\penalty0 405--422, 2010.

\bibitem[Gy{\"o}ngy and Krylov(1996)]{GyoengyKrylov1996}
Istv{\'a}n Gy{\"o}ngy and Nicolai Krylov.
\newblock Existence of strong solutions for {I}t\^o's stochastic equations via
  approximations.
\newblock \emph{Probab. Theory Related Fields}, 105\penalty0 (2):\penalty0
  143--158, 1996.
\newblock ISSN 0178-8051.
\newblock \doi{10.1007/BF01203833}.
\newblock URL \url{http://dx.doi.org/10.1007/BF01203833}.

\bibitem[Hudde et~al.(2018)Hudde, Hutzenthaler, Jentzen, and
  Mazzonetto]{HuddeHutzenthalerJentzenMazzonetto2018}
Anselm Hudde, Martin Hutzenthaler, Arnulf Jentzen, and Sara Mazzonetto.
\newblock {On the It\^o-Alekseev-Gr\"obner formula for stochastic differential
  equations}.
\newblock \emph{arXiv preprint arXiv:1812.09857}, 2018.

\bibitem[Hudde et~al.(2021)Hudde, Hutzenthaler, and Mazzonetto]{HHM2021}
Anselm Hudde, Martin Hutzenthaler, and Sara Mazzonetto.
\newblock {A stochastic {G}ronwall inequality and applications to moments,
  strong completeness, strong local Lipschitz continuity, and perturbations}.
\newblock \emph{Annales de l'Institut Henri Poincar\'e, Probabilit\'es et
  Statistiques}, 57\penalty0 (2):\penalty0 603 -- 626, 2021.
\newblock \doi{10.1214/20-AIHP1064}.
\newblock URL \url{https://doi.org/10.1214/20-AIHP1064}.

\bibitem[Hutzenthaler and Jentzen(2015)]{HutzenthalerJentzen2015Memoires}
Martin Hutzenthaler and Arnulf Jentzen.
\newblock Numerical approximations of stochastic differential equations with
  non-globally {L}ipschitz continuous coefficients.
\newblock \emph{Mem. Amer. Math. Soc.}, 4:\penalty0 1--112, 2015.

\bibitem[Hutzenthaler et~al.(2018)Hutzenthaler, Jentzen, Kruse, Nguyen, and von
  Wurstemberger]{HJK+18}
Martin Hutzenthaler, Arnulf Jentzen, Thomas Kruse, Tuan~A. Nguyen, and Philippe
  von Wurstemberger.
\newblock Overcoming the curse of dimensionality in the numerical approximation
  of semilinear parabolic partial differential equations.
\newblock \emph{arXiv:1807.01212}, 2018.

\bibitem[Kunita(1990)]{Kunita1990}
Hiroshi Kunita.
\newblock \emph{Stochastic flows and stochastic differential equations},
  volume~24 of \emph{Cambridge Studies in Advanced Mathematics}.
\newblock Cambridge University Press, Cambridge, 1990.
\newblock ISBN 0-521-35050-6.

\bibitem[Li(1994)]{Li1994}
Xue-Mei Li.
\newblock Strong {$p$}-completeness of stochastic differential equations and
  the existence of smooth flows on noncompact manifolds.
\newblock \emph{Probab. Theory Related Fields}, 100\penalty0 (4):\penalty0
  485--511, 1994.
\newblock ISSN 0178-8051.
\newblock \doi{10.1007/BF01268991}.
\newblock URL \url{http://dx.doi.org/10.1007/BF01268991}.

\bibitem[Li and Scheutzow(2011)]{LiScheutzow2011}
Xue-Mei Li and Michael Scheutzow.
\newblock Lack of strong completeness for stochastic flows.
\newblock \emph{Ann. Probab.}, 39\penalty0 (4):\penalty0 1407--1421, 2011.

\bibitem[Schenk-Hopp{\'e}(1996)]{SchenkHoppe1996Deterministic}
Klaus~Reiner Schenk-Hopp{\'e}.
\newblock Deterministic and stochastic {D}uffing-van der {P}ol oscillators are
  non-explosive.
\newblock \emph{Z. Angew. Math. Phys.}, 47\penalty0 (5):\penalty0 740--759,
  1996.
\newblock ISSN 0044-2275.
\newblock \doi{10.1007/BF00915273}.
\newblock URL \url{http://dx.doi.org/10.1007/BF00915273}.

\bibitem[Schmalfu\ss(1997)]{Schmalfuss1997}
Bj\"orn Schmalfu\ss.
\newblock The random attractor of the stochastic {L}orenz system.
\newblock \emph{Zeitschrift f\"ur Angewandte Mathematik und Physik (ZAMP)},
  48:\penalty0 951--975, 1997.

\bibitem[Zhang(2010)]{Zhang2010}
Xicheng Zhang.
\newblock Stochastic flows and {B}ismut formulas for stochastic {H}amiltonian
  systems.
\newblock \emph{Stochastic Process. Appl.}, 120\penalty0 (10):\penalty0
  1929--1949, 2010.
\newblock ISSN 0304-4149.
\newblock \doi{10.1016/j.spa.2010.05.015}.
\newblock URL \url{http://dx.doi.org/10.1016/j.spa.2010.05.015}.

\end{thebibliography}

\end{document}